%% file: QFrob.tex
\documentclass{amsart}
\usepackage[dvipsnames]{xcolor}

\usepackage{amssymb}
\usepackage{tikz}
\usetikzlibrary{arrows}

\usepackage{float}

\RequirePackage{color}
\definecolor{myred}{rgb}{0.75,0,0}
\definecolor{mygreen}{rgb}{0,0.5,0}
\definecolor{myblue}{rgb}{0,0,0.65}

\RequirePackage{ifpdf}
\ifpdf
  \IfFileExists{pdfsync.sty}{\RequirePackage{pdfsync}}{}
  \RequirePackage[pdftex,
   colorlinks = true,
   urlcolor = myblue, 
   citecolor = mygreen, 
   linkcolor = myred, 
   pagebackref,
   bookmarksopen=true]{hyperref}
\else
  \RequirePackage[hypertex]{hyperref}
\fi

\RequirePackage{ae, aecompl, aeguill} 

\def\un{\underline}

\newcommand{\gl}{{\mathfrak{gl}}}
\renewcommand{\sl}{{\mathfrak{sl}}}
\newcommand{\into}{\hookrightarrow}
\newcommand{\onto}{\twoheadrightarrow}
\newcommand{\co}{\colon}
\newcommand{\ot}{\otimes}

\newcommand{\namedto}[1]{\stackrel{#1}{\longrightarrow}}
\newcommand{\sumset}{\stackrel{\scriptstyle{\oplus}}{\scriptstyle{\subset}}}

\usepackage[all]{xy}
\usepackage{graphicx}
\newcommand{\ig}[2]{\vcenter{\xy (0,0)*{\includegraphics[scale=#1]{./#2}} \endxy}}

\newtheorem{thm}{Theorem}[section]
\newtheorem{lem}[thm]{Lemma}
\newtheorem{prop}[thm]{Proposition}
\newtheorem{cor}[thm]{Corollary}
\newtheorem{conj}[thm]{Conjecture}

\theoremstyle{definition}
\newtheorem{defn}[thm]{Definition}
\newtheorem{notation}[thm]{Notation}
\newtheorem{ex}[thm]{Example}

\theoremstyle{remark}
\newtheorem{rem}[thm]{Remark}

\def\ZZ{\mathbb{Z}}
\def\QQ{\mathbb{Q}}
\def\CC{\mathbb{C}}
\def\BB{\mathbb{B}}
\def\XX{\mathbb{X}}
\def\NN{\mathbb{N}}
\def\RR{\mathbb{R}}
\DeclareMathOperator{\rt}{root}
\DeclareMathOperator{\lon}{long}
\DeclareMathOperator{\Sym}{Sym}
\DeclareMathOperator{\End}{End}
\DeclareMathOperator{\boring}{bore}

\DeclareMathOperator{\aff}{aff}
\DeclareMathOperator{\fin}{fin}
\DeclareMathOperator{\BS}{BS}
\def\de{\delta}
\def\De{\Delta}
\def\ze{\zeta}
\def\La{\Lambda}
\def\la{\lambda}
\def\al{\alpha}
\def\be{\beta}

\def\si{\sigma}
\def\Om{\Omega}
\def\om{\omega}
\def\ga{\gamma}
\newcommand{\yb}{\mathbf{y}}
\newcommand{\AC}{\mathcal{A}}
\newcommand{\pa}{\partial}
\newcommand{\NC}{NC}
\newcommand{\NH}{NH}
\newcommand{\stair}{P}
\newcommand{\ab}{\mathbf{a}}
\DeclareMathOperator{\cw}{cw}
\DeclareMathOperator{\ws}{ws}
\newcommand{\bTheta}{\bar{\Theta}}
\newcommand{\zz}{\mathbf{a}}
\newcommand{\maxcalc}{20}

\DeclareMathOperator{\bottom}{bot}
\DeclareMathOperator{\blahblah}{foobar}

\title[]{Frobenius extensions and the exotic nilCoxeter algebra for $G(m,m,3)$}

\author[]{Ben Elias}
\address{University of Oregon.}
\email{belias@uoregon.edu}

\author[]{Daniel Juteau}
\address{LAMFA, Universit\'{e} de Picardie Jules Verne.}
\email{daniel.juteau@u-picardie.fr}

\author[]{Benjamin Young}
\address{University of Oregon.}
\email{bjy@uoregon.edu}

\begin{document}

\begin{abstract} In a previous paper of the first author, the type $A_{n-1}$ affine Cartan matrix was $q$-deformed to produce a deformation of the reflection representation of the
affine Weyl group $W_{\aff}$. This deformation plays a role in the quantum geometric Satake equivalence. In this paper we introduce the study of $q$-deformed divided difference
operators.

When $q$ is specialized to a primitive $2m$-th root of unity, this reflection representation of $W_{\aff}$ factors through a quotient, the complex reflection group $G(m,m,n)$. The
divided difference operators now generate a finite-dimensional algebra we call the exotic nilCoxeter algebra. This algebra is new and has surprising features. In addition to the
usual braid relations, we prove a new relation called the roundabout relation.

A classic result of Demazure, for Weyl groups, states that the polynomial ring of the reflection representation is a Frobenius extension over its subring of invariant polynomials, and describes how the Frobenius trace can be constructed within the nilCoxeter algebra. We study the analogous Frobenius extension for $G(m,m,n)$, and identify the Frobenius trace within the exotic nilCoxeter algebra for $G(m,m,3)$. 


%
\end{abstract}

\maketitle

\tableofcontents

\input{Introduction.tex}

\input{RefRep.tex}

\input{Polys.tex}
\input{Exotic.tex}
\input{FormulaROUpaperone.tex}
\input{Experimental.tex}

\bibliographystyle{plain}
\bibliography{mastercopy}

\end{document}

%% file: Introduction.tex
The affine Weyl group $W_{\aff}$ can be viewed as the semidirect product of the finite Weyl group $W_{\fin}$ with its root lattice $\La_{\rt}$. In \cite{EQuantumI}, the first author
introduced a $q$-deformation of the affine Cartan matrix in type $\tilde{A}_{n-1}$, leading to a $q$-deformation of the reflection representation of $W_{\aff}$. The existence of a
one-parameter deformation of this representation was shown by Lusztig in \cite{LuszPeriodic97}, stemming from earlier work in \cite{LusSqint} (see \cite[Section 5.3]{EQuantumI} for more
details). However, the parametrization in \cite[Equation (1.1)]{EQuantumI} is special in that, when $q$ is set equal to a $2m$-th root of unity, the $m$-th multiples of the root lattice
generate the kernel of the action. Thus we obtain a faithful representation of the quotient $W_m:= W_{\aff}/m \La_{\rt}$, a finite group also known as the complex reflection group
$G(m,m,n)$. The result is a well-known representation of $G(m,m,n)$, though not usually studied through the lens of $W_{\aff}$.

This paper is concerned with very natural, purely algebraic questions about this $q$-deformed reflection representation at a root of unity and its polynomial ring, which we now
summarize. Let $R_m$ be the symmetric algebra of the reflection representation, and $R_m^{W_m}$ the subring of invariant polynomials. The ring $R_m$ comes equipped an $R_m^{W_m}$-linear
endomorphism $\pa_s$ for each simple reflection $s$ in the affine Weyl group, called a Demazure operator or divided difference operator. We study the \emph{exotic nilCoxeter algebra} $\NC(m,m,n)$, the subring of
$\End_{R_m^{W_m}}(R_m)$ generated by these Demazure operators. We provide some new relations which hold in these algebras for all $m$ and $n$. We prove that $R_m^{W_m} \subset R_m$ is a
Frobenius extension. When $n=3$, we announce a result from a followup paper, which gives a precise relationship between the Frobenius trace map and the nilCoxeter algebra. Finally, we discuss experimental data on the presentation
of $\NC(m,m,3)$ by generators and relations.

\begin{ex} \label{ex:n2start} Let $n=2$. The affine Weyl group $W_{\aff}$ is the infinite dihedral group with simple reflections $\{s,t\}$. The product $st$ is translation by a root, so $(st)^m$ generates $m \La_{\rt}$. The quotient group $G(m,m,2)$ is isomorphic to the finite dihedral group of type $I_2(m)$, with its usual Coxeter presentation. The $q$-deformed Cartan matrix of $W_{\aff}$ is 
\begin{equation} \left( \begin{array}{cc} 2 & -(q+q^{-1}) \\ -(q+q^{-1}) & 2 \end{array} \right). \end{equation}
Setting $q = 1$ we recover the usual Cartan matrix in type $\tilde{A}_1$. When $\theta = \frac{\pi}{m}$ and $q$ is specialized to $e^{i \theta}$, then $q + q^{-1} = 2 \cos(\theta)$, and we recover the Cartan matrix of the usual reflection representation of $I_2(m)$. Thus we can apply classical results to study the (not very exotic) nilCoxeter algebra. In particular, $\NC(m,m,2)$ is a quotient of the nilCoxeter
algebra of $W_{\aff}$ by a single new relation, the braid relation of length $m$. The graded dimension of $\NC(m,m,2)$ is the Poincar\'{e} polynomial of $W_m$, and any element of length $2m$ in the affine Weyl group gives rise to a Frobenius trace.
\end{ex}

\begin{ex} \label{ex:m2n3start} Let $m = 2$ and $n=3$. Let $\{s, t, u\}$ denote the simple reflections of $W_{\aff}$, with $m_{st} = m_{su} = m_{tu} = 3$. We think of $\{s,t\}$ as
generating the finite Weyl group, and $u$ as the affine reflection. We use these conventions for all examples with $n=3$.

The group $G(2,2,3)$ is abstractly isomorphic to $S_4$, via an
isomorphism which sends $s \mapsto (12)$, $t \mapsto (13)$, and $u \mapsto (14)$. Note that $sts \mapsto (23)$ and thus the images of $sts$ and $u$ commute. The element $stsu \in
W_{\aff}$ is translation by the highest root. The kernel of the map $W_{\aff} \to S_4$ is generated by $(stsu)^2$. For more details see \S\ref{ssec:223intro}.

Like the previous example, this quotient $G(2,2,3)$ of $W_{\aff}$ is itself a Coxeter group $S_4$. Unlike the previous example, the presentation of $G(2,2,3)$ as a quotient of $W_{\aff}$ is unrelated to the Coxeter presentation of $S_4$, a fact with significant implications. The ordinary nilCoxeter algebra of $S_4$ has dimension $24$, with graded dimension
\begin{equation} 1 + 3v + 5v^2 + 6v^3 + 5v^4 + 3v^5 + v^6\end{equation} 
matching the usual Poincar\'{e} polynomial of $S_4$. Meanwhile, with a different length function induced from $W_{\aff}$, $G(2,2,3)$ has Poincar\'{e} polynomial
\begin{equation} \label{eq:PPofG223} 1 + 3v + 6v^2 + 9v^3 + 5v^4. \end{equation}
The reader may be surprised to learn that $\NC(2,2,3)$ has dimension $36$, with Poincar\'{e} polynomial
\begin{equation} \label{eq:PPofNC223} 1 + 3v + 6v^2 + 9v^3 + 10 v^4 + 6 v^5 + v^6. \end{equation}
Like the ordinary nilCoxeter algebra of $S_4$, $\NC(2,2,3)$ is one dimensional in degree $-6$, the negative-most degree, and any non-zero element in this degree serves as a Frobenius trace. All reduced expressions of length $6$ in $W_{\aff}$ give rise to non-zero elements in this degree, with the important exception of cyclic expressions like $stustu$ or $tsutsu$. A presentation of $\NC(2,2,3)$ can be found in \S\ref{ssec:223intro}. \end{ex}

\begin{rem} For $m \ge 3$, $G(m,m,3)$ is not a Coxeter group. The dimension of $\NC(m,m,3)$ for $m \ge 2$ forms a sequence $(36, 84, 153, 243, \ldots)$ which does not yet appear in the OEIS \cite{OEIS}. \end{rem}

While this paper is mostly algebraic and low-tech, and the topic seems esoteric, there is very strong motivation coming from geometric representation theory. In
\cite{EQuantumI}, the first author gave a reformulation of the geometric Satake equivalence as an equivalence between two monoidal categories: colored $\sl_n$-webs, which describe
morphisms between representations of $U(\sl_n)$, and certain singular Soergel bimodules for $W_{\aff}$. More surprisingly, it was observed that the $q$-deformation of the reflection
representation of $W_{\aff}$ gives rise to a $q$-deformation of singular Soergel bimodules which matches a standard $q$-deformation in representation theory, representations of the
quantum group $U_q(\sl_n)$. Specializing to a root of unity, this produces new connections between, on one side, tilting modules for quantum groups at roots of unity, and on the other
side, singular Soergel bimodules for $G(m,m,n)$. The overarching goal of this paper is to develop the algebraic theory needed to study the category of singular Soergel bimodules
associated to this reflection representation of $G(m,m,n)$.

Categorification of complex reflection groups and their representation theory has been a long-standing open problem since the introduction of Spetses \cite{BMMSpetses}. While we do not
believe that these singular Soergel bimodules will categorify the Hecke algebra of $G(m,m,n)$ itself, we do expect a reasonably close relationship. We give more details on this
motivational material in \S\ref{ssec:introGRT}.

\begin{rem} For some unrelated initial progress on categorifying complex reflection groups using constructions similar to Soergel bimodules, see \cite{GobetThiel}. \end{rem}
	
\textbf{Acknowledgments} The first author was supported by NSF grants DMS-1800498 and DMS-2201387.  We thank Ulrich Thiel for several discussions, including a very helpful explanation of his work and the history behind the study of these Frobenius extensions. Major progress on this paper was made while the first two authors were visiting the Institute of Advanced Study, a visit supported by NSF grant DMS-1926686. Both authors would like to thank the Institute for their hospitality during the long Covid year. We would also like to thank our bubble buddy Anne Dranowski. The first and third author appreciate the support given to their research group by NSF grant DMS-2039316.

\section{Extended Introduction} \label{sec:intro}

\subsection{Complex reflection groups and Frobenius extensions} \label{ssec:tryme}

Let $W$ be a finite group acting faithfully on a complex vector space $V$. A \emph{(complex) reflection} is an endomorphism of $V$ of finite order whose $1$-eigenspace has codimension $1$. Let $R$ denote the polynomial ring $\Sym(V)$, and $R^W$ the subring of $W$-invariant polynomials. It is a
classic theorem of Shephard and Todd \cite{ShephardTodd} that the following statements are equivalent: \begin{itemize} \item The ring $R^W$ is also a polynomial ring, and $R$ is free over
$R^W$ of finite rank, \item $W$ is generated by reflections, whence $W$ is known as a \emph{complex reflection group}. \end{itemize} When $W$ is infinite and generated by reflections (such as an affine Weyl group), it is often the case that $R^W$ is a polynomial ring but with fewer generators, so that $R$ will be free over $R^W$ of infinite rank.


When $W$ is a finite Coxeter group acting on its reflection representation, Demazure \cite{Demazure} proved a stronger statement, which is that $R$ is a Frobenius extension over $R^W$.
There are many equivalent ways of stating what a Frobenius extension is, but the one we focus on in this paper is the existence of a \emph{Frobenius trace} map $\pa_W \co R \to
R^W$, an $R^W$-linear map which is \emph{non-degenerate} in the sense that the $R^W$-bilinear pairing \begin{equation} R \times R \to R^W, \qquad (f,g) \mapsto \pa_W(fg)
\end{equation} is perfect. That is, $R$ admits two bases (as a free module over $R^W$) which are dual with respect to this pairing. Note that any Frobenius extension is free of finite rank, so the Shephard-Todd Theorem is a prerequisite. It is easy to provide an explicit construction of the operator $\pa_W$:
\begin{equation} \label{eq:defnpaW} \pa_W(f) = \frac{\sum_{w \in W} (-1)^{\ell(w)} w(f)}{\Pi_{\alpha \in \Phi^+} \alpha}. \end{equation}
Here $\Phi^+$ represents the set of positive roots, viewed as linear polynomials in $R$. That is, to obtain $\pa_W(f)$ one antisymmetrizes the polynomial $f$, and then divides by the ``canonical antisymmetric polynomial.''

\begin{rem} \label{rmk:coinvariant} Let $C$ be the \emph{coinvariant algebra}, the quotient of $R$ by the ideal $I$ generated by positive degree elements of $R^W$. When $W$ is a finite Coxeter group, $C$ is
a finite dimensional algebra, whose graded dimension matches the graded rank of $R$ over $R^W$. An $R^W$-linear operator $R \to R^W$ is a Frobenius trace if and only if the induced operator $C \to \CC$ is a Frobenius trace. The literature commonly focuses on the coinvariant algebra rather than the extension $R^W
\subset R$. When $W$ is a Weyl group, $C$ is isomorphic to the cohomology ring of the flag variety, and the Frobenius trace can be described as integration over this compact manifold. That integration induces a perfect pairing is the statement of Poincar\'{e} duality. \end{rem}

Moreover, Demazure proved that the Frobenius trace has an alternate construction within the nilCoxeter algebra. For each simple reflection $s$ in a Coxeter group $W$, let
$R^s \subset R$ denote the subring of $s$-invariants. Then $R^s \subset R$ is a Frobenius extension with Frobenius trace $\pa_s$:
\begin{equation} \pa_s(f) = \frac{f - s(f)}{\alpha_s}. \end{equation} The \emph{nilCoxeter algebra} is the subalgebra of
$\End_{R^W}(R)$ generated by $\pa_s$ for all simple reflections $s$. Demazure proves that, whenever $s_1 s_2 \cdots s_d$ is a reduced expression for the longest element of $W$, then
\begin{equation} \label{paWfromlongest}\pa_W = \pa_{s_1} \circ \pa_{s_2} \cdots \cdots \circ \pa_{s_d}.\end{equation} There are a number of significant consequences to this
innocuous algebraic fact, some of which we discuss in \S\ref{ssec:introGRT}. Demazure operators also play a major role in algebraic combinatorics, where they are used to construct Schubert polynomials.

Now suppose $W$ is a complex reflection group rather than a Coxeter group. There is no analogue of the nilCoxeter algebra in the literature; there is not necessarily a sign
representation or a theory of roots which would enable a definition like \eqref{eq:defnpaW}. Nonetheless, $R^W \subset R$ is still a Frobenius extension, though we had to ask many an
expert before finding one who knew this fact. The following remark discusses the proof, as explained to us by Ulrich Thiel (and as found within his thesis \cite[Proposition 17.33]{ThielThesis}). However, this proof does not supply an explicit Frobenius trace operator $\pa_W$, it only implies that one must exist.

\begin{rem} As noted in Remark \ref{rmk:coinvariant}, it is equivalent to prove that the coinvariant algebra is a Frobenius algebra (over the ground field). A similar (and sufficient) structure is that of a Poincar\'{e} duality algebra, a positively graded ring with a one-dimensional top degree for which multiplication to that degree is a non-degenerate pairing. In \cite[Theorem 5.7.4]{NeuselSmithBook} one can find a general result showing that when $R$ is an $n$-dimensional polynomial ring (or more generally a Gorenstein ring) and $I$ is an ideal generated by $n$ algebraically-independent homogeneous elements, then $R/I$ is a Poincar\'{e} duality algebra. This applies to the construction of the coinvariant algebra. One can find the details carefully spelled out in \cite[Chapter 17]{ThielThesis}. \end{rem}

Is there a description of the Frobenius trace map $\pa_W$ analogous to \eqref{eq:defnpaW}? Is there a way of describing $\pa_W$ using an analogue of the
nilCoxeter algebra? We are unaware of any systematic approach to these questions in the literature, but for the complex groups $G(m,d,n)$ there is a history of studying related questions, which we discuss in \S\ref{ssec:introhistory}. The nilCoxeter algebra we study in this paper is different from those in the literature.

\subsection{Questions and results} \label{ssec:whynot}

Let us now set notation and fix our setting precisely. For simplicity and symmetry we abandon the variable $q$ for a nicer variable $z$, the relationship being that $z^n = q^{-2}$. More details can be found in \S\ref{sec:qref}.

\begin{notation} Fix $n \ge 2$. Let $\Om = \ZZ/n\ZZ$ be the vertices in the affine Dynkin diagram in type $\widetilde{A}_{n-1}$. Let $W_{\aff}$ be the affine Weyl group, with simple
reflections $S = \{s_i\}_{i \in \Om}$. \end{notation}

\begin{defn} Let $z$ be a formal variable, and let $V_{z}$ be the free $\CC[z,z^{-1}]$-module with basis $\{x_i\}_{i \in \Om}$. It has an action of $W_{\aff}$ defined as
follows: \begin{equation} s_i(x_i) = z x_{i+1}, \quad s_i(x_{i+1}) = z^{-1} x_i, \quad s_i(x_j) = x_j \text{ if } j \notin \{i,i+1\}. \end{equation} For any
$m \ge 2$ let $V_m$ be the $\CC$-vector space obtained by specializing $z$ to a primitive $(nm)$-th root of unity $\ze \in \CC$. \end{defn}

This representation $V_z$ is a deformation\footnote{More precisely, $V_z$ comes from a deformed affine Cartan matrix. However, specializing $z=1$ yields the inflation to $W_{\aff}$ of the permutation representation of $S_n$. This is a representation of $W_{\aff}$ where $S$ acts by reflections, but is not what is commonly called the reflection representation of $W_{\aff}$.} of the reflection representation of $W_{\aff}$. The specialization $V_m$ is a faithful representation of the quotient group $W_m$.

\begin{notation} Let $R_z$ be the polynomial ring $\Sym(V_z)$ over the base ring $\CC[z,z^{-1}]$, and $R_m$ be the polynomial ring $\Sym(V_m)$ over $\CC$. Both are graded so that $V_z$ (resp. $V_m$) appears in degree $1$. \end{notation}
	
\begin{defn} For each $i \in \Om$ define
certain linear maps $R_z \to R_z$ of degree $-1$, called \emph{divided difference operators} or \emph{Demazure operators}, as follows: \begin{equation} \pa_i(f) = \frac{f - s_i f}{x_i - z x_{i+1}}. \end{equation}
These maps descend to $R_m$. \end{defn}

\begin{defn}
The subalgebra of $\End_{R_z^{W_{\aff}}}(R_z)$ generated by Demazure operators $\pa_i$ for $i \in \Om$ is called the \emph{deformed affine nilCoxeter algebra} $\NC(z,n)$.
The subalgebra of $\End_{R_m^{W_m}}(R_m)$ generated by Demazure operators $\pa_i$ for $i \in \Om$ is called the \emph{exotic nilCoxeter algebra} $\NC(m,m,n)$. \end{defn}
	
A consequence of the Shephard-Todd theorem is that $\NC(m,m,n)$ is finite-dimensional. In this paper we ask the following questions.
\begin{enumerate}
	\item What is the graded dimension of $\NC(m,m,n)$?
	\item What are the relations between Demazure operators $\pa_i$? Can we find a presentation of $\NC(m,m,n)$?	
	\item We know $R_m^{W_m} \subset R_m$ is a Frobenius extension. Can we explicitly construct a Frobenius trace $\pa_{W_m}$ using an antisymmetrization formula similar to \eqref{eq:defnpaW}?
	\item For which words (i.e. sequences $(s_1, \ldots, s_d)$ of simple reflections) in $S$ is $\pa_{s_1} \circ \cdots \circ \pa_{s_d}$ an invertible scalar multiple of the Frobenius trace $\pa_{W_m}$?
\end{enumerate}

As noted in Example \ref{ex:n2start}, these questions all have classical answers for $n=2$, because $G(m,m,2)$ is the dihedral group with its usual Coxeter presentation.

Theorem \ref{thm:JFrob}, our first main result, reproves that $R_m^{W_m} \subset R_m$ is a Frobenius extension for all $m \ge 2$ and $n \ge 3$. It constructs an explicit Frobenius trace map $\pa_{W_m}$ using an ``antisymmetrization'' formula which generalizes \eqref{eq:defnpaW}. We explicitly compute the pairing matrix on a particular basis of monomials, and relate it to a pairing matrix in finite type which
is known to be nondegenerate.

\begin{rem} Demazure's lovely proof for Coxeter groups \cite[Proposition 4, Corollary afterwards, Theorem 2]{Demazure} critically uses the fact that multiplication by the longest
element is an involution which induces a symmetry on the Poincar\'{e} polynomial of $W$. In our setting we have no unique longest element, nor is the Poincar\'{e} polynomial of either $W_m$ or
$\NC(m,m,n)$ symmetric, c.f. \eqref{eq:PPofG223} and \eqref{eq:PPofNC223}. We explain these issues and the additional questions they raise in \S\ref{ssec:proofcomparison}. \end{rem}

\begin{rem} Note that the degree of $\pa_{W_m}$ is $-\binom{n}{2} m$, not $-nm$. Many of our examples use $n=3$, a special case where $n = \binom{n}{2}$, and we wish to avert confusion. \end{rem}

Now we consider the exotic nilCoxeter algebra. It is easy to show that $\NC(m,m,n)$ satisfies familiar quadratic relations like (using notation from Example \ref{ex:m2n3start})
\[ \pa_s \pa_s = 0 \]
and slightly unfamiliar braid relations like (recall that $\ze$ is an $nm$-th root of unity)
\[ \ze \pa_s \pa_t \pa_s = \pa_t \pa_s \pa_t. \]
 Our second main result is Theorem \ref{thm:roundabout}, which proves a reasonably elegant pair of relations in $\NC(m,m,n)$ called the \emph{roundabout relations}. Living in degree $(n-1)m$, the roundabout relation is a linear relation between the Demazure operators of the $n$ clockwise (resp. anticlockwise) cyclic words of that length.

\begin{ex} When $n=3$ and $m=3$, the clockwise roundabout relation states that
\begin{equation} \pa_s \pa_t \pa_u \pa_s \pa_t \pa_u + \ze^{-3} \pa_t \pa_u \pa_s \pa_t \pa_u \pa_s + \ze^{-6} \pa_u \pa_s \pa_t \pa_u \pa_s \pa_t = 0. \end{equation}
By right-multiplying this relation with $\pa_s \pa_t$, one can quickly deduce that
\begin{equation} \pa_s \pa_t \pa_u \pa_s \pa_t \pa_u \pa_s \pa_t = 0, \end{equation}
and the same for any cyclic word of length $8$.
\end{ex}

\begin{ex} When $n=2$, the roundabout relation is the usual length $m$ braid relation for $I_2(m)$. As a consequence, the Demazure operator associated to any alternating word of length $m+1$ is zero. \end{ex}

We deduce our roundabout relations from Theorem \ref{thm:thetakrotates}, which proves an interesting result about similar operators in $\NC(z,n)$, before specializing to a root of
unity. This is one of many instances of a key theme in this paper, that the study of $G(m,m,n)$ is clarified by putting it in the larger context of $W_{\aff}$. It is
difficult to prove general results about $\NC(m,m,n)$ by induction, since there is no a priori relationship between $\NC(m,m,n)$ and $\NC(m', m', n)$. It is easier to prove results
about $\NC(z,n)$ by induction on length, and then specialize to a root of unity.

For the rest of this section we focus on the case of $n=3$. The roundabout relations in degree $2m$ imply that the Demazure operator of any cyclic word of length $2m+2$ is zero. This then implies that
the Demazure operators of many words of length $3m$ are zero. Meanwhile, every other word of length $3m$ has a nonzero Demazure operator! To state the theorem concisely, we use the following parametrization of $W_{\aff}$.

\begin{defn} \label{defn:abiintro} Fix $a, b \ge 0$, and $i \in \Om$. Let $\un{w}(a,b,i)$ be the word in $\Om$ defined as follows. \begin{itemize}
	\item It begins with a clockwise cycle of length $a+1$, and ends with an anticlockwise cycle of length $b+1$. These cycles overlap in one letter, making the total length $a+b+1$.
	\item The final letter is $i$. \end{itemize}
We let $\pa_{(a,b,i)}$ denote the corresponding operator in $\NC(m,m,3)$, a composition of Demazure operators for simple reflections.
\end{defn}

\begin{ex} We have $\un{w}(3, 5, 2) = (1,2,3,1,3,2,1,3,2)$. One can view this as a clockwise word $1231$ of length $4 = a+1$ and an anticlockwise word $132132$ of length $6 = b+1$ which overlap in the middle index $1$. The last index is $i = 2$. We have 
	\[ \pa_{(3,5,2)} = \pa_1 \pa_2 \pa_3 \pa_1 \pa_3 \pa_2 \pa_1 \pa_3 \pa_2.\]
Using the notation of Example \ref{ex:m2n3start}, we might instead write $\un{w}(3,5,t) = (s,t,u,s,u,t,s,u,t)$.
\end{ex}
	
Every non-identity element of $W_{\aff}$ has a unique reduced expression of the form $\un{w}(a,b,i)$ for a unique triple $(a,b,i)$, see Lemma \ref{lem:abiparametrize}. If $a \ge 2m+1$ or $b \ge 2m+1$ then
$\un{w}(a,b,i)$ contains a cyclic word of length $2m+2$, so $\pa_{(a,b,i)}$ will vanish by the roundabout relation.
	
\begin{thm} \label{thm:whennonzerointro} Let $i \in \Om$ and $a,b \ge 0$ with $a+b+1 = 3m$. The operator $\pa_{(a,b,i)} \in \NC(m,m,3)$ is nonzero, and hence a Frobenius trace, if and only if $a, b \le 2m$. \end{thm}
		
\begin{ex} The operator $\pa_{(3,5,2)}$ is a Frobenius trace inside $\NC(3,3,3)$. \end{ex}
	
Theorem \ref{thm:whennonzerointro} is stated again as Corollary \ref{cor:actuallymaintheorem}, and proven as a consequence of our third main result, which we announce as Theorem
\ref{thm:whatisw0}. The proof of Theorem \ref{thm:whatisw0} can be found in the companion paper \cite{EJY2}. This theorem gives an explicit computation of $\pa_{(a,b,i)}(\stair)$, where $\stair = x_1^{2m} x_2^m$, when $a+b+1 = 3m$. The end result is a scalar, which we denote $\Xi_m(a,b,i)$. For degree reasons, any Demazure operator in degree $-3m$
is either zero or a scalar multiple of the Frobenius trace $\pa_{W_m}$ constructed previously. We prove that $\pa_{W_m}(\stair) = 1$, so that the scalar computed by Theorem
\ref{thm:whatisw0} is also the proportion between $\pa_{(a,b,i)}$ and $\pa_{W_m}$.

We motivate the importance of knowing these scalars precisely in Remark \ref{rem:wanttoknowscalar}. We continue the discussion of these
scalars and other computational results in \S\ref{ssec:computations}. For now, we note the surprise appearance of quantum binomial coefficients (evaluated at $q = \ze^{\frac{-n}{2}}$) in these scalars.


Finally, let us discuss the presentation of $\NC(m,m,3)$, and its graded dimension. The roundabout relations (together with the quadratic and braid relations) are already
sufficient to ensure that the nilCoxeter algebra is finite dimensional, but many more relations are needed to cut it down to size. We explored this question by computer for $m \le
\maxcalc$, using MAGMA \cite{MAGMA}. The rest of the presentation of $\NC(m,m,3)$ is mysterious and slightly chaotic.

For $2 \le m \le 5$, one needs $3(m-1)$ relations in degree $3m-1$. The purpose of these relations seems to be to ensure that the dimension of $\NC(m,m,3)$ in degree
$3m-1$ is exactly $6$. For $6 \le m \le 13$ one instead has $3(m-5)$ relations in degree $3m-2$, ensuring that the dimension of $\NC(m,m,3)$ in degree $3m-2$ is exactly $21$ (in
degree $3m-1$ it is still dimension $6$). This suggests that the precise presentation is not as significant as the graded dimension, which has interesting asymptotics. We provide
the results of our computer calculations with interpretation in \S\ref{ssec:results}. Code can be found at \cite{EJYcode}.

\subsection{Motivation from geometric representation theory} \label{ssec:introGRT}

Let $(W,S)$ be a Coxeter system acting on a vector space $V$, where $S$ acts by reflections. A subset $I \subset S$ is called \emph{finitary} if the parabolic subgroup $W_I$ that it generates is finite. One can consider the polynomial ring $R = \Sym(V)$, and its invariant subring $R^I := R^{W_I}$. If $I \subset J$ then $R^J \subset R^I$. We usually write $R^W$ rather than $R^S$.

Let $I_{\bullet} = [[I_0 \supset I_1 \subset I_2 \supset \cdots \subset I_d]]$ be a collection of finitary subsets of $S$, where $I_k$ is a subset of $I_{k \pm 1}$ whenever $k$ is odd. To this sequence we may associate an $(R^{I_0},R^{I_d})$-bimodule
\begin{equation} \label{eq:singBS} \BS(I_{\bullet}) := {}_{R^{I_0}} R^{I_1} \ot_{R^{I_2}} R^{I_3} \ot_{R^{I_4}} \cdots \ot_{R^{I_{d-2}}} R^{I_{d-1}} {}_{R^{I_d}}, \end{equation}
called a \emph{(singular) Bott-Samelson bimodule}\footnote{We largely ignore grading shifts in this introduction.}. Taking all direct summands of singular Bott-Samelson bimodules, we obtain the \emph{singular Soergel bimodules}.

Singular Soergel bimodules form a graded additive 2-category, with one object for each finitary $I \subset S$. A $1$-morphism from $I$ to $J$ is a singular Soergel
$(R^J,R^I)$-bimodule, and composition of 1-morphisms is given by tensor product. Thus the bimodule in \eqref{eq:singBS} is a $1$-morphism from $I_d$ to $I_0$. The $2$-morphisms are bimodule maps.

What makes singular Soergel bimodules particularly well behaved is the aforementioned result of Demazure \cite{Demazure}, stating that $R^I \subset R$ is a Frobenius extension
whenever $I$ is finitary. For a Frobenius extension, induction and restriction are biadjoint (up to grading shift). The same goes for $R^J \subset R^I$ whenver $I \subset J$ and
both are finitary. Taking the tensor product with a singular Bott-Samelson bimodule is the same as an iterated composition of induction and restriction functors.

Meanwhile, when $W_I$ is infinite, the ring extension $R^I \subset R$ is not of finite rank, and is not Frobenius. This was the reason we restricted to finite parabolic subgroups in the
definition of singular Soergel bimodules. In particular, when $W = W_{\aff}$, we can consider all invariant polynomial rings $R^I$ except for $R^{W_{\aff}}$ itself.

The geometric Satake equivalence is an equivalence between representations of a Lie group $G$ and equivariant perverse sheaves on the affine Grassmannian of the Langlands dual Lie
group. Using results of Soergel and Harterich \cite{Soer90, Harterich}, this result was reformulated in a purely algebraic fashion by the first-named author in \cite{EQuantumI}. The category of
equivariant perverse sheaves is replaced by a $2$-category built using singular Soergel bimodules for the reflection representation\footnote{For this purpose one could use either the usual reflection representation of the affine Weyl group, where the simple roots are linearly independent, or the specialization of $V_z$ at $z=1$, where they are linearly dependent.} of the affine Weyl group $W_{\aff}$.
To each representation of $G$ one can associate a particular $(R^J, R^I)$-bimodule, where $J$ and $I$ are (possibly different) copies of the finite Dynkin diagram inside the affine
Dynkin diagram. The morphisms between representations precisely match the degree zero bimodule maps between these singular Soergel bimodules. Note that (singular) Soergel bimodules
are amenable to algebraic and computational (i.e. diagrammatic) approaches, see \cite[Chapter 24]{EMTW}.

As noted earlier, in type $\tilde{A}_{n-1}$ \cite{EQuantumI} introduced a $q$-deformation $V_q$ of the affine reflection representation, which is related to $V_z$ by a change of variables. This produces a deformed $2$-category of singular Soergel bimodules, and the degree zero bimodule maps encode\footnote{In \cite{EQuantumI} only the cases $n=2, 3$ are proven.} morphisms between representations of $U_q(\sl_n)$. 


This paper initiates the study of what happens when $q$ (or $z$) is specialized to a root of unity. The result is related to representations of quantum groups at roots of unity, of
course, but is a richer and weirder object. New degree zero morphisms appear between singular Soergel bimodules, not corresponding directly to morphisms between representations of
$U_q(\sl_n)$! We hope to understand these new morphisms.

By our Theorem \ref{thm:JFrob}, when $q$ is specialized to a root of unity so that the action of $W_{\aff}$ factors through $G(m,m,n)$, the ring extension $R^{W_{\aff}} \subset R$ is a Frobenius extension. One can now formally extend the definition of Bott-Samelson bimodules to allow factoring through $R^{W_{\aff}}$,
introducing bimodules like 
\begin{equation} \label{deep} R^I \ot_{R^{W_{\aff}}} R^J. \end{equation} It turns out that such bimodules already appear as direct summands of existing Bott-Samelson
bimodules, so the formal extension is unnecessary. For each $I, J \subset S$ there is a unique indecomposable bimodule factoring through $R^{W_{\aff}}$, namely the one in \eqref{deep}, which we might call the \emph{deep} $(R^I, R^J)$-bimodule. These deep bimodules form a monoidal ideal in the $2$-category: direct sums of deep bimodules are closed under induction and restriction.

The question ``which bimodules contain a deep bimodule as a direct summand'' is directly related to Question (4) from \S\ref{ssec:whynot} (which sequences of simple reflections
induce the Frobenius trace), as we will soon explain. Earlier we said that deep bimodules were already direct summands of Bott-Samelson bimodules; more precisely, there are several non-isomorphic indecomposable bimodules when $q$ is a formal variable, summands of different Bott-Samelson bimodules, which all become isomorphic to the deep bimodule when $q$ is specialized to a root of unity. This gives rise to new morphisms between Bott-Samelson bimodules, by projecting and including through the (now) common summand.
This seems to explain all the new morphisms between singular Soergel bimodules which arise when $q$ is set to a root of unity.

Let us demonstrate this in the case $n=2$. Here, the representation of the infinite dihedral group $W_{\aff}$ actually factors through a finite dihedral group $W_m = G(m,m,2)$. This
being a Coxeter group rather than just a complex reflection group, Soergel theory was already developed. An intense diagrammatic study of morphisms in this setting can be found in
\cite{ECathedral}.

To be concrete, assume that $m=4$, though everything generalizes to arbitrary $m \ge 2$. Let $s$ and $t$ denote the two simple reflections and write $W$ for $W_{\aff}$. There are two $(R^s, R^t)$ bimodule maps \[ R^s \ot_{R^W} R^t \namedto{i} {}_{R^s} R \ot_{R^t} R \ot_{R^s} R_{R^t}
\namedto{p} R^s \ot_{R^W} R^t. \] (Looking on \cite[bottom of p49]{ECathedral}, take the diagram $C_k$ for $k=m$ and chop it in half horizontally; the bottom half is $i$, the top half is $p$, and the composition $C_m$ is $p \circ i$.) If $p \circ i$ is the identity map of $R^s \ot_{R^W} R^t$, then they form the inclusion and projection map of a direct summand, and
\begin{equation} \label{itsasummand} R^s \ot_{R^W} R^t \sumset R \ot_{R^t} R \ot_{R^s} R, \quad \text{ as } (R^s,R^t)\text{-bimodules}. \end{equation}
	
That $p \circ i = 1$ is equivalent (see the proof of \cite[Claim 6.5]{ECathedral}) to the algebraic fact that \[ \pa^s_W = \pa_t \pa_s \pa_t \qquad \text{ or equivalently } \qquad \pa_W = \pa_s \pa_t \pa_s
\pa_t.\] Here $\pa^s_W$ is the Frobenius trace for $R^W \subset R^s$. In other words, it is Demazure's description of the Frobenius trace from \eqref{paWfromlongest} which implies \eqref{itsasummand}. Inducing to $R$ on both the left and
right, we deduce that the deep bimodule $B_{w_0} := R \ot_{R^{W}} R$ is a direct summand of \[ \BS(stst) := R \ot_{R^s} R \ot_{R^t} R \ot_{R^s} R \ot_{R^t} R.\]


Before we specialized $q$ to a root of unity, the Demazure operators $\pa_s$ and $\pa_t$ satisfied the relations of the nilCoxeter algebra of $W_{\aff}$, with no braid relation.
After specializing they satisfy a braid relation $\pa_W := \pa_s \pa_t \pa_s \pa_t = \pa_t \pa_s \pa_t \pa_s$ of length $m$, giving two descriptions of the Frobenius trace for $R^W
\subset R$. Thus by the same computation, $B_{w_0}$ is also a direct summand of $\BS(tsts)$, defined analogously to $\BS(stst)$. This common summand leads to a morphism $\BS(stst)
\to \BS(tsts)$, the $2m$-valent vertex from \cite{ECathedral}. The $2m$-valent vertex is a new morphism which appears after specializing, and generates all such new morphisms, c.f.
\cite[Lemma 6.8, Claim 6.20]{ECathedral}. To summarize, by finding multiple descriptions of the Frobenius trace using compositions of Demazure operators for simple reflections, one discovers new morphisms between Bott-Samelson bimodules.

\begin{rem} \label{rem:wanttoknowscalar} Above we had $\pa_W := \pa_s \pa_t \pa_s \pa_t = \pa_t \pa_s \pa_t \pa_s$ with no scalars required. In a related setting discussed in \cite[Appendix]{ECathedral}, one might have e.g. $\pa_W = a \pa_s \pa_t \pa_s \pa_t$ for some scalar $a$. This scalar affects the relations between morphisms, see e.g. \cite[(A.9), (A.10)]{ECathedral}. Thus, it is important to know the Frobenius trace precisely, not just up to scalar, in order to present the category of Soergel bimodules by generators and relations. \end{rem}

Now we consider the other side of the Satake equivalence, still with $n=2$. The right-hand side of \eqref{itsasummand} corresponds to the representation $L_1^{\ot m-1}$ (where $L_1$ is
the standard representation of $U_q(\sl_2)$). In \cite[Definition 5.28, Theorem 5.29]{ECathedral} it was proven that the idempotent $i \circ p$ agrees with the (negligible)
Jones-Wenzl projector, so that the deep summand corresponds to the negligible tilting module $T_{m-1}$. In fact, the deep monoidal ideal corresponds to the negligible ideal in the
category of quantum group representations!

When $n\ge 3$, we expect similar results on both sides of the Satake equivalence. We expect descriptions of $\pa_W$ as a composition of other Demazure operators to explain the existence
of direct summands inside Soergel bimodules, corresponding to negligible summands inside certain tensor products of representations. We expect new morphisms (generalizations of the
$2m$-valent vertex) to arise from common direct summands of this form, and to generate all the new morphisms between Soergel bimodules. This gives a strong practical motivation for the
study of Question (4) from \S\ref{ssec:whynot}, which we answer when $n=3$. The monoidal ideal generated by $R^{W_{\aff}}$ will correspond to some monoidal ideal in representations of
$U_q(\sl_n)$, most likely one properly contained in the negligible ideal. We also hope that the connection to representation theory will ultimately explain why quantum binomial
coefficients appeared in our formulas for $\Xi_m(a,b,i)$.

\begin{rem} Soergel bimodules for $G(m,m,3)$ should lead to a new construction of (some of) the \emph{trihedral $2$-representations} described by Mackaay, Mazorchuk, Miemietz, and Tubbenhauer in their fascinating paper \cite{MMMTTrihedral}. The introduction to that paper is very readable and gives additional motivation on the topic. \end{rem}

\begin{rem} Many geometric constructions which exist for Weyl groups have combinatorial and algebraic analogues which continue to exist for Coxeter groups and even complex reflection
groups. For example, Soergel bimodules are the algebraic version of perverse sheaves on the flag variety, and are defined for all Coxeter groups even in the absence of a flag variety.
The Spetses program \cite{BMMSpetses} hypothesizes some sort of unipotent character sheaves for complex reflection groups, despite the lack of geometry. More precisely,
\cite{BMMSpetses} studies unipotent characters, which are the combinatorial shadow of unipotent character sheaves. For Weyl groups, character sheaves can be understood \cite{BFO,
BzNcharacter} as the Drinfeld center\footnote{An object of the Drinfeld center is an object $Z$ together with a functorial isomorphism $(-) \ot Z \to Z \ot (-)$.} of the categorified
$J$-ring, which is built from Soergel bimodules (one can also study the Drinfeld center of the homotopy category of Soergel bimodules as a surrogate). Recent work of Rogel-Thiel
\cite{RogelThiel} has verified that Soergel bimodules for finite dihedral groups $G(m,m,2)$ do indeed satisfy the properties desired by the Spetses program. One
wonders what new central objects will appear using Soergel bimodules for $G(m,m,n)$ when $n > 2$. \end{rem}

\subsection{Computations} \label{ssec:computations}

We return on a more technical level to the discussion of Theorem \ref{thm:whennonzerointro}, and the study of the operators $\pa_{(a,b,i)}$. We now state formulas for these operators
when applied to specific monomials.

\begin{thm}
Fix $i \in \Om$ and $a, b \ge 0$ such that $a+b+1 = 3m$ and $m-1 \le a,b \le 2m$. Define $\alpha$ and $\beta$ by $a = 2\alpha + 2$ or $2 \alpha + 1$, and $b = 2 \beta + 2$ or $2 \beta + 1$, depending on
parity. Note that the condition $m-1 \le a \le 2m$ is equivalent to the condition $\bottom \le \alpha, \beta \le m-1$, for an integer $\bottom$ which is approximately $\frac{m}{2}$, see \S\ref{ssec:evalatrou} for details. Then
\begin{equation} \label{eq:paabiintro} \pa_{(a,b,i)}(x_1^{2m} x_2^m) = \eta(a,b) \cdot m^2 \cdot {m-1-\bottom \brack \beta - \bottom}_q, \end{equation}
where $\eta(a,b)$ is a sign times a power of the root of unity $\ze$. The quantum binomial coefficient is with respect to the parameter $q = \ze^{\frac{-n}{2}}$  For the precise formula, see Theorem \ref{thm:whatisw0}.
\end{thm}

We found the formula \eqref{eq:paabiintro} to be highly satisfying and surprising. Here is a visualization of this result for $m=3$.
\begin{equation} \label{m3n3binomials} \ig{.15}{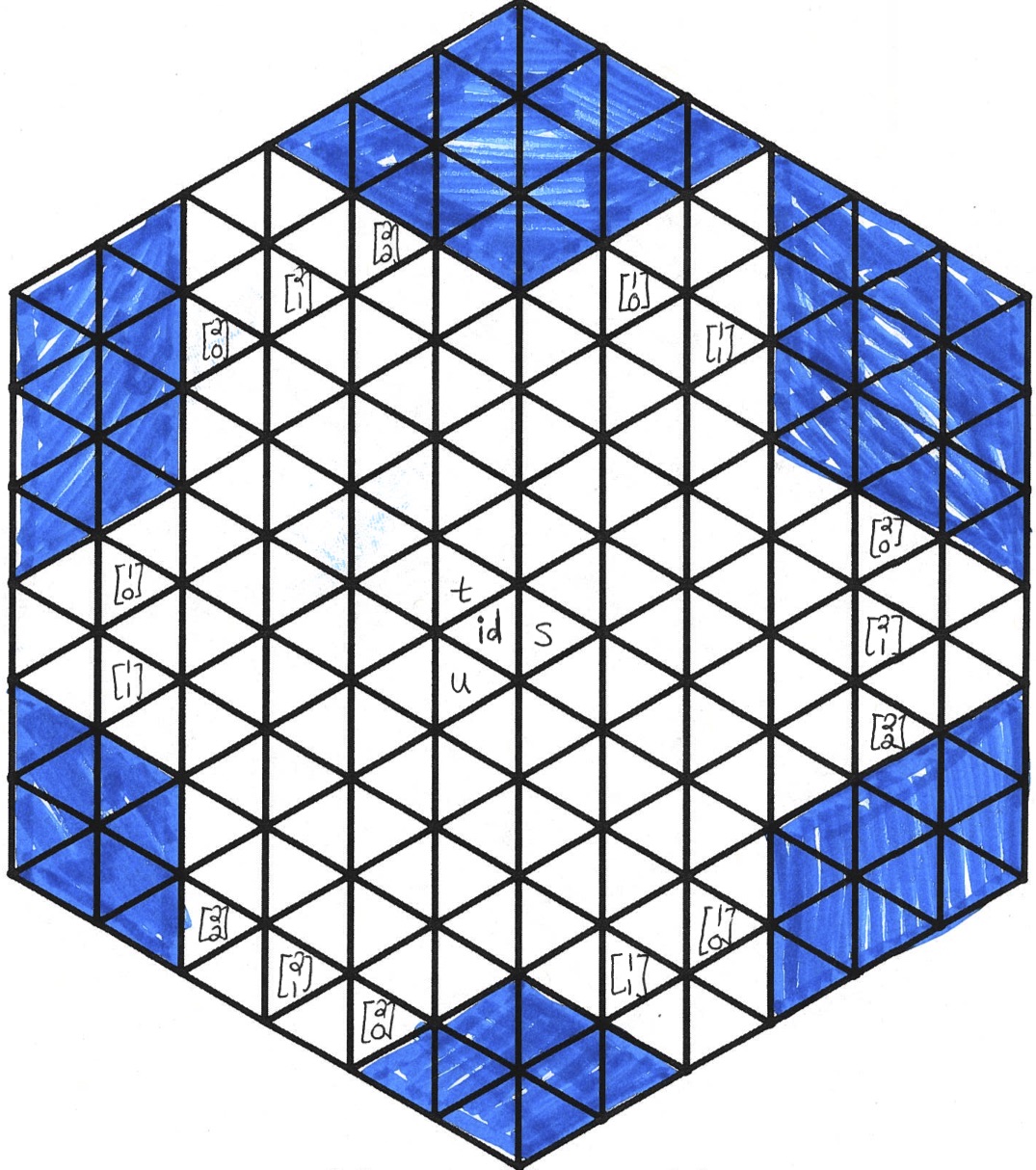} \end{equation}
Each alcove in this grid represents an element of $W_{\aff}$. The blue alcoves represent elements $w \in W_{\aff}$ for which $\pa_w = 0$ when $m=3$ by the roundabout relation. The remaining alcoves representing elements of length $9 = 3m$ have been labeled with the binomial coefficients appearing in \eqref{eq:paabiintro} (ignoring the factors $\eta(a,b)$ and $m^2$). Even though this is a relatively small example, the pattern is clear and it continues for higher $m$.

The formula \eqref{eq:paabiintro} is independent of $i$. Indeed, the bottom degree of $\NC(m,m,n)$ admits the trivial
representation of the rotation operator on the affine Dynkin diagram, a fact which can be seen from the explicit formula for $\pa_{W_m}$ as well.


The quantum binomial coefficient ${m-1-\bottom \brack \beta - \bottom}$ is defined using quantum numbers in the variable $q = \ze^{-\frac{3}{2}}$. To define this scalar
we introduced a square root of $\ze$, and $\eta$ is technically a half-power of $\ze$. The final result lives in $\ZZ[\ze]$. This description of the formula is for convenience; no half-powers of $\ze$ are technically necessary.


In order to prove \eqref{eq:paabiintro} we ended up proving a much stronger result about Demazure operators in $\NC(z,3)$, before specializing $z$ to a root of unity. We found
an explicit closed formula for the scalar \begin{equation} \label{eq:whatnastycomputesintro} \pa_{(a,b,i)}(x_1^k x_2^{a+b+1-k}) \end{equation} for all quadruples $(a,b,i,k)$. After
applying various symmetries, this suffices to explicitly compute $\pa_{(a,b,i)}(f)$ for any monomial $f$ of the appropriate degree (so that the result is a scalar). Thus we reiterate the same theme: it is hard to prove \eqref{eq:paabiintro} by induction, but the computation of \eqref{eq:whatnastycomputesintro} before specializing is amenable to inductive proof, albeit a difficult one.

The formula for \eqref{eq:whatnastycomputesintro} is significantly nastier than the formula
from \eqref{eq:paabiintro}, though it has its own elegance. For example, when $a$ is even and $b$ is odd and $k \ne 0$ and $a+b+1-k \ne 0$ then \begin{align} \label{eq:paabikintro} \pa_{(a,b,i)}(x_1^k x_2^{a+b+1-k})
= \eta'(a,b,i) \cdot & [\alpha]! \cdot [\beta]! \cdot [\alpha+\beta+2-k] \cdot (q-q^{-1})^{\alpha+\beta+1} \\ \nonumber & \cdot \sum_{j=0}^{\beta} {k-1 \brack \beta-j} {\alpha+\beta+1-k \brack j} q^{j(-3(\alpha+\beta+2) +
2k)}, \end{align} where $\eta'$ is a sign and a (half-)power of $z$ (depending on $i$).

\begin{rem} The $q$-Chu-Vandermonde identity states that
\begin{equation} \label{qchuvander} \sum_{j=0}^{\beta} {M \brack \beta-j} {N \brack j} q^{j(M+N)} = q^{N \beta} {M+N \brack \beta}. \end{equation}
However, the sum in \eqref{eq:paabikintro} does not match the left-hand side of \eqref{qchuvander} (for $M = k-1$ and $N = \alpha+\beta-(k-1)$) because the exponent of $q$ does not match. As a consequence, we do not have a simpler form for our sum, analogous to the right-hand side of \eqref{qchuvander}. This sum should be viewed as some new and unusual $q$-deformation of $\binom{\alpha+\beta}{\beta}$. \end{rem}

We note that our formula for \eqref{eq:whatnastycomputesintro} has many edge cases (e.g. when $k=0$ or $a=0$, etcetera), in addition to twelve main cases (the parity of $a$ and $b$, the choice of $i$). Only the main cases are important when considering \eqref{eq:paabiintro} at a root of unity, but the edge cases are important for the inductive proof of correctness. The proof involves interesting manipulations of sums involving quantum binomial coefficients, and a truly immense amount of bookkeeping.

After specializing $z \mapsto \ze$ and $k = 2m$, \eqref{eq:paabikintro} can be simplified to \eqref{eq:paabiintro}. This simplification is far from straightforward. 

\begin{rem} The reader may be surprised by the factor of $m^2$ in \eqref{eq:paabiintro}, and wonder where it came from in \eqref{eq:paabikintro}. Quantum numbers at roots of unity
satisfy many interesting relations. For example, $(q-q^{-1})^{m-1} [m-1]!$ is equal to $m$, up to a power of $\sqrt{-1}$, see \cite[Section 6.2]{EJY2}. The sum of double
binomial coefficients also simplifies at a root of unity, becoming equal (up to unit) to a quantum factorial number, see \cite[Section 6.3]{EJY2}. \end{rem}

While a closed formula for Demazure operators might be of independent interest, the computation of these scalars is extremely technical. The quantum number manipulations in the proof
also seem to speak to a different audience. We wished for this paper to be more accessible, as the first introduction to the exotic nilCoxeter algebra. We have opted to place our
general formula for \eqref{eq:whatnastycomputesintro}, together with the verification that it specializes to \eqref{eq:paabiintro} at a root of unity, in a companion paper \cite{EJY2}.

Both the formulas \eqref{eq:paabiintro} and \eqref{eq:paabikintro} were discovered by staring very hard at computer calculations and extracting the pattern. Readable code which
verifies these formulas may be found at \cite{EJYcode}; we gratefully used the MAGMA language \cite{MAGMA}.

\begin{rem} Let us briefly note some of additional difficulties, invisible now that the problem is solved. MAGMA thinks of the cyclotomic field $\QQ(\ze)$ as $\QQ[z]/\Phi(z)$ for
the cyclotomic polynomial $\Phi$, and records scalars as polynomials in $z$ of degree less than $\deg(\Phi)$. It is not easy to recognize a quantum binomial coefficient, or even a
power of $\ze$, from its description as a low-degree polynomial! Moreover, there are many coincidences between binomial coefficients at roots of unity, making it hard to find a
pattern. \end{rem}


\subsection{Other approaches to Demazure operators for $G(m,d,n)$} \label{ssec:introhistory}

We end the introduction by discussing previous approaches to Demazure operators and Frobenius extensions for $G(m,d,n)$. We begin this discussion here, and continue in
\S\ref{ssec:rampetas}. Our discussion of the literature is for context, and will play no role in this paper.

Famously, the affine Weyl group has two presentations, the Coxeter presentation (viewing it as a Coxeter group) and the loop presentation (viewing it as a semidirect product of a
finite Weyl group with a lattice). There are two corresponding presentations of its reflection representation. Similarly, the extended affine Weyl group in type $A$ (obtained from
the affine Weyl group by adjoining an operator $\si$ which rotates the Dynkin diagram) has both a loop and a ``Coxeter-esque'' presentation. As noted above, $G(m,m,n)$ is a quotient
of this affine Weyl group. Similarly, $G(m,1,n)$ is a quotient of the extended affine Weyl group, and $G(m,d,n)$ is a quotient of a subgroup of the partially extended affine Weyl
group, which only adjoins certain powers of $\si$. Thus each of these complex reflection groups inherits two presentations, the loop-esque and Coxeter-esque presentations. For reasons we explain below, we call the Coxeter-esque presentation the \emph{exotic presentation}.
Meanwhile, $G(2,2,n)$ is the Weyl group of type $D_n$, which has a Coxeter presentation, and there is another standard presentation of $G(m,m,n)$ and its reflection representation
which generalizes that of $D_n$.

It is an oversimplification, though a morally accurate one, to say that the existing literature on $G(m,d,n)$ uses either the loop-esque presentation or the type-$D$-style
presentation, when constructing length functions and Demazure operators. We are not aware of any previous works studying $G(m,d,n)$ using its exotic presentation. For our
applications to quantum geometric Satake, it is essential that we use the exotic presentation, e.g. that our nilCoxeter algebra is generated by the operators $\pa_i$
associated to the simple reflections of the affine Weyl group.

\begin{rem} The paper \cite{BMR} introduces a class of presentations for complex reflection groups analogous to the Coxeter presentations. The exotic presentation of $G(m,m,n)$
involves taking the Coxeter presentation of the affine Weyl group and adding one more relation. This one relation involves all the simple reflections at once! In particular, it does
not fit into the framework of \cite{BMR}, whose relations involve at most three simple reflections at once. \end{rem}

For the complex reflection group $G(m,1,n)$, Shoji and Rampetas \cite{RamShoIandII} define a collection of Demazure operators $\De_w$ acting on the appropriate analogue of $R$ and $C$,
one for each element $w$ of the group. They build closely upon work of Bremke and Malle \cite{BreMal1, BreMal2}, which defines (using the loop presentation of $G(m,1,n)$) a length
function $\ell$ on $G(m,1,n)$, together with a root system. The Demazure operator $\De_w$ has degree $-\ell(w)$. Shoji and Rampetas prove \cite[part II, Theorem 2.18]{RamShoIandII} that
the space spanned by their Demazure operators is dual to the coinvariant algebra, and go one step further to prove that $\De_{w_0}$ is a Frobenius trace. However, their operators
$\De_w$ do not span a subalgebra within the endomorphisms of the polynomial ring.

Rampetas \cite{Ramp} went on to define Demazure operators for $G(m,m,n)$. Again he builds on Bremke-Malle \cite{BreMal2}, which defines a length function for $G(m,m,n)$ using the
type-$D$-style presentation. That the space of Demazure operators is dual to the coinvariant algebra is only conjectured, and no mention is made of Frobenius extensions. Again,
Rampetas' Demazure operators do not span a subalgebra. We emphasize that our nilCoxeter algebra will be completely distinct from algebra generated by Rampetas' operators, see \S\ref{ssec:otherviews} and \S\ref{ssec:rampetas}.

Let us mention a few more papers related to this study. Shoji in \cite{ShojiGrpn} generalized the length functions of Bremke and Malle to $G(m,d,n)$. In \cite{HughesMorris}, one can find root systems for a large class of (mostly exceptional) complex reflection groups. In \cite{TotaroGe1nSchubert} one can find the beginnings of Schubert calculus for $G(m,1,n)$, with a formula related to the Pieri rule for the top class in the coinvariant ring. In \cite{KirillovMaeno} a Nichols algebra is defined for arbitrary complex reflection groups which contains the coinvariant algebra, and the Shoji-Rampetas algebra in type $G(m,1,n)$ shown to be related to this Nichols algebra, giving another perspective on their statement of duality. In \cite{OrtizGKM} one can find another approach to $G(m,1,n)$ and Schubert calculus for its coinvariant ring, using a variant on moment graphs suited to the loop presentation, and yet another algebra of Demazure operators is studied.


\subsection{Outline of paper} \label{ssec:outline}

We have just completed \S\ref{sec:intro}, which serves as an introduction to both this paper and the companion paper \cite{EJY2}. In particular, \S\ref{ssec:introGRT} motivated the study of the exotic nilCoxeter algebra using the quantum geometric Satake equivalence, and \S\ref{ssec:introhistory} described some of the earlier work of Shoji and Rampetas on Demazure operators for $G(m,d,n)$. In \S\ref{ssec:outline} you can find a helpful outline of the paper.

In \S\ref{sec:qref}, we describe and discuss the representations $V_z$, $V_m$, and their variants. We also discuss how root systems work in these representations. We give more
detail than is needed, so that this paper may serve as a helpful reference. In \S\ref{sec:polys} we prove results about the polynomial ring of $V_m$ and its invariant subring. We
define an antisymmetrization operator, and prove that $R_m^{W_m} \subset R_m$ is a Frobenius extension. In \S\ref{ssec:proofcomparison} we contrast our proof with Demazure's proof
for Coxeter groups. In \S\ref{sec:nilcox} we introduce the exotic nilCoxeter algebra and prove the roundabout relations, as well as related formulas for generic $z$. In
\S\ref{ssec:rampetas} we give more background on the Demazure operators of Rampetas and Shoji-Rampetas.

In \S\ref{sec:closedformularou} we specialize to $n=3$. We discuss the closed formula for Demazure operators acting on monomials in $\NC(z,3)$ and its specialization to
$\NC(m,m,3)$. These formulas are proven in the companion paper, though we give some indication of their proof here. Using these formulas, it is straightforward to prove Theorem \ref{thm:whennonzerointro}.

In \S\ref{sec:exex}, we give more details on the examples of $\NC(2,2,3)$ and $\NC(3,3,3)$, and provide experimental results on $\NC(m,m,3)$ for $m \le \maxcalc$. The reader can skip
directly to \S\ref{sec:exex} if desired, for some juicy surprises.

%% file: RefRep.tex
\section{The deformed reflection representation} \label{sec:qref}

We establish basic properties of the representation $V_z$, and define the corresponding root system. For context we also discuss a host of other realizations of affine Weyl groups which appear in the literature. The reader interested in our main theorems need only look at Definitions \ref{defn:boring} and \ref{defn:Vz} for the definitions of $V_{\boring}$ and $V_z$, and at Definition \ref{defn:symmetries} and \S\ref{ssec:basicpropsofrealizations} for some symmetries and basic properties. Then one is able to read the remainder of the paper, starting from \S\ref{ssec:roots}, save for some historical comments.

\subsection{The affine Weyl group} \label{ssec:affineweyl}

Let $W_{\aff}$ be the affine Weyl group associated to the symmetric group $S_n$. It has two common descriptions. Viewing $W_{\aff}$ as a Coxeter group, it is generated by a set of simple reflections $S = \{s_i\}$ indexed by $\Om = \ZZ/n\ZZ$. We have $m_{s_i s_j} = 3$ if $j = i \pm 1$, and $m_{s_i
s_j} = 2$ otherwise. In this paper all indices will be considered modulo $n$. One can think that $\{s_1, \ldots, s_{n-1}\}$ generate the finite Weyl group $S_n$, while $s_0$ is the \emph{affine reflection}.

Meanwhile, we can also describe $W$ as a semidirect product $S_n \ltimes \La_{\rt}$, where $\La_{\rt}$ is the root lattice of $S_n$. Elements of $\La_{\rt}$ are called \emph{translations}. A nice way to pass between these two descriptions is as follows. Let
\begin{equation} t_{\lon} := s_1 s_2 \cdots s_{n-2} s_{n-1} s_{n-2} \cdots s_2 s_1, \end{equation}
which is $(1n)$ in cycle notation. This is the reflection in $S_n$ corresponding to the highest root $\alpha_{\lon}$. Then $s_0 t_{\lon}$ is translation by $\alpha_{\lon}$. All other roots are in the $S_n$ orbit of $\alpha_{\lon}$, so all other translations by roots are conjugate to $s_0 t_{\lon}$ under $S_n$.

\subsection{Several realizations of the affine Weyl group} \label{ssec:qrefdefined}

\begin{defn} (See \cite[Definition 3.1]{EWGr4sb}) Let $(W,S)$ be a Coxeter group and $\Bbbk$ a commutative domain. A \emph{realization} of $W$ over $\Bbbk$ is the data of a free
$\Bbbk$ module $V$ with a $W$-action, together with a set of roots $\{\al_s\} \subset V$ and coroots $\{\al_s^\vee\} \subset V^*$ indexed by $s \in S$. We require that the action of
$s \in S$ on $V$ is given by the formula \[ s(v) = v - \langle \al_s^\vee, v \rangle \al_s. \] We also require $\langle \al_s^\vee, \al_s \rangle = 2$ for all $s$, and $\langle
\al_s^\vee, \al_t \rangle = 0$ whenever $m_{st} = 2$. The \emph{Cartan matrix} of the realization is the $S \times S$ matrix with entries $\langle \al_s^\vee, \al_t \rangle$.
\end{defn}

A realization is a formalization of what it means to be a ``reflection representation.''

\begin{defn} \label{defn:boring} The ``boring'' realization of $W_{\aff}$ over $\Bbbk$ is the free module $V_{\boring} = \Bbbk^n$ with basis $\{y_i\}_{i \in \Om}$, where
\begin{equation} s_i(y_i) = y_{i+1}, \quad s_i(y_{i+1}) = y_i, \quad s_i(y_j) = y_j \text{ if } j \notin \{i,i+1\}. \end{equation}
If $\{y_i^*\}$ is the dual basis to $\{y_i\}$, then the simple roots and coroots are
\begin{equation} \al_i = y_i - y_{i+1}, \qquad \al_i^\vee = y_i^* - y_{i+1}^*, \end{equation}
and the Cartan matrix is the usual affine Cartan matrix. \end{defn}

As $s_0$ and $t_{\lon}$ both act as the transposition $(1n)$, the translation $s_0 t_{\lon}$ acts trivially on $V_{\boring}$. By conjugation, all generators of $\La_{\rt}$ act trivially on $V_{\boring}$, and this representation factors through the quotient map $W_{\aff} \onto S_n$.  There is a common method to modify $V_{\boring}$ to get a faithful representation of $W_{\aff}$ with the same Cartan matrix, see for instance \cite[Section 2.1.1]{MacThi}. 

\begin{defn} \label{defn:Vde} There is a $\Bbbk$-linear realization on the  $V_{\de} = \Bbbk^n \oplus \Bbbk \cdot \de$, where $W_{\aff}$ fixes $\de$. The action of $S_n$ on $\Bbbk^n$ is as before, while the action of $s_0$ is modified so that
\begin{equation} s_0(y_n) = y_1 - \de, \qquad s_0(y_1) = y_n + \de, \qquad s_0(y_j) = y_j \text{ if } j \notin \{1,n\}. \end{equation} 
One sets
\begin{equation} \al_0 = y_n - y_1 + \de, \qquad \al_0^\vee = y_n^* - y_1^*. \end{equation}
Now we can see that the operator $s_0 t_{\lon}$ has infinite order on $V_{\de}$, sending $y_1 \mapsto y_1 - \de$ and $y_n \mapsto y_n + \de$. \end{defn}

What is less well-known is another method to modify $V_{\boring}$ to get a faithful representation of $W_{\aff}$, this time over the base ring $\AC_q := \Bbbk[q,q^{-1}]$.

\begin{defn} \label{defn:Vq} There is a $\AC_q$-linear realization of $W_{\aff}$, acting on the free module $V_q := \AC_q^n$ with basis $\{v_i\}$. The action of $S_n$ is to permute the $v_i$ as usual, while the action of $s_0$ is modified so that
\begin{equation} \label{qrefrep} s_0(v_n) = q^{-2} v_1, \qquad s_0(v_1) = q^2 v_n, \qquad s_0(v_j) = v_j \text{ if } j \notin \{1,n\}. \end{equation}
Letting $\{v_i^*\}$ denote the dual basis of $\{v_i\}$, for $i \ne 0$ we have
\begin{equation} \al_i = v_i - v_{i+1}, \qquad \al_i^\vee = v_i^* - v_{i+1}^*. \end{equation}
Meanwhile, one has
\begin{equation} \al_0 = q v_n - q^{-1} v_1, \qquad \al_0^\vee = q^{-1} v_n^* - q v_1^*. \end{equation}
The Cartan matrix is
\begin{equation} \label{slnqCartan} \left( \begin{array}{cccccc}
2 & -1 & 0 & \cdots & 0 & -q^{-1} \\
-1 & 2 & -1 & \cdots & 0 & 0 \\
0 & -1 & 2 & \cdots &  & 0 \\
\vdots & \vdots & \vdots & \ddots & -1 & 0 \\
0 & 0 &   & -1 & 2 & -q \\
-q & 0 & 0 & \cdots & -q^{-1} & 2
\end{array} \right), \end{equation}
though in the special case $n=2$ the Cartan matrix is
\begin{equation} \label{sl2qCartan} \left( \begin{array}{cc}
2 & -(q+q^{-1}) \\
-(q+q^{-1}) & 2 \end{array} \right). \end{equation}
\end{defn}

These Cartan matrices first appeared in \cite{EQuantumI}. Setting $q = 1$ recovers $V_{\boring}$ and the usual Cartan matrix. Unlike the usual setting, this Cartan matrix has nonzero determinant $2 - q^2 - q^{-2}$. Now the translation operator $s_0 t_{\lon}$ acts by
\begin{equation} \label{transopq} s_0 t_{\lon}(v_1) = q^{-2} v_1, \qquad s_0 t_{\lon}(v_n) = q^2 v^n, \qquad s_0 t_{\lon}(v_j) = v_j \text{ if } j \notin \{1,n\}. \end{equation}
Once again, it has infinite order.

\begin{rem} \label{rmk:rootversion} For all three realizations $V_{\boring}$, $V_{\de}$, and $V_q$, one could consider the \emph{root subrealization} spanned by the simple roots.
Morally, this subrealization is associated to $\sl_n$ rather than $\gl_n$. These three realizations have three different behaviors: \begin{itemize} \item For $V_{\boring}$ the sum of
the simple roots is zero. So $V_{\boring}$ has rank $n$, and the root subrealization has rank $n-1$. \item In $V_{\de}$ the simple roots are linearly independent, and their sum is the
$W$-invariant element $\de$. So $V_{\de}$ has rank $n+1$, and the root subrealization has rank $n$. \item In $V_q$ the simple roots are linearly independent, but no linear combination
is $W$-invariant or particularly special. Both $V_q$ and the root subrealization have rank $n$, and indeed after base change to a field (where $2 - q^2 - q^{-2} \ne 0$) they are equal. \end{itemize} \end{rem}

\begin{rem} In \cite{EQuantumI} only the root subrealization is studied. In this paper we introduce $V_q$ and its variant $V_z$ below. This slightly larger realization is significantly easier to deal with, both on a theoretical and practical level, see e.g. Remark \ref{rem:betterthanroot1} and Remark \ref{rmk:stairslnvsgln}. \end{rem}

The symmetries of the affine Dynkin diagram act compatibly on $V_{\boring}$. They do act on $V_{\de}$ and $V_q$ as well, but not by obvious, memorable formulas. We can alter both realizations to produce a more symmetric version.

\begin{defn} \label{defn:Vden} Let $V_{\de/n}$ be the free module $\Bbbk^n \oplus \Bbbk \cdot \de/n$, where $\Bbbk^n$ has basis $\{\gamma_i\}$. For all $i \in \Om$ we set
\begin{equation} \label{denrefrep} s_i(\gamma_i) = \gamma_{i+1} - \frac{\de}{n}, \qquad s_i(\gamma_{i+1}) = \gamma_i + \frac{\de}{n}, \qquad s_i(\gamma_j) = \gamma_j \text{ if } j \notin \{i,i+1\}. \end{equation}
The simple roots and coroots are
\begin{equation} \al_i = \gamma_i - \gamma_{i+1} + \frac{\de}{n}, \qquad \al_i^\vee = \gamma_i^* - \gamma_{i+1}^*. \end{equation}
\end{defn}

After base change to $\Bbbk' = \Bbbk[\frac{1}{n}]$, setting $\gamma_k = y_k + \frac{k}{n} \de$ for $1 \le k \le n$ (not considered modulo $n$ for purposes of this formula) gives an identification of $V_{\de/n}$ with $V_{\de}$. When working over $\Bbbk = \ZZ$, $V_{\de/n}$ is a larger $\ZZ$-lattice which properly contains $V_{\de}$.

\begin{defn} \label{defn:Vz} Let $z$ be a formal variable, and let $\AC_{z} := \Bbbk[z,z^{-1}]$. Let $V_z$ be the free $\AC_z$-module $\AC_z^{\oplus n}$ with basis $\{x_i\}_{i \in \Om}$. For all $i \in \Om$ define
\begin{equation} \label{zetarefrep} s_i(x_i) = z x_{i+1}, \quad s_i(x_{i+1}) = z^{-1} x_i, \quad s_i(x_j) = x_j \text{ if } j \notin \{i,i+1\}. \end{equation}
The simple roots and coroots are
\begin{equation} \al_i = x_i - z x_{i+1}, \qquad \al_i^\vee = x_i^* - z^{-1} x_{i+1}^*. \end{equation}
The Cartan matrix is
\begin{equation} \label{slnzeCartan} \left( \begin{array}{cccccc}
2 & -z & 0 & \cdots & 0 & -z^{-1} \\
-z^{-1} & 2 & -z & \cdots & 0 & 0 \\
0 & -z^{-1} & 2 & \cdots &  & 0 \\
\vdots & \vdots & \vdots & \ddots & -z & 0 \\
0 & 0 &   & -z^{-1} & 2 & -z \\
-z & 0 & 0 & \cdots & -z^{-1} & 2
\end{array} \right), \end{equation}
though in the special case $n=2$ the Cartan matrix is
\begin{equation} \label{sl2zeCartan} \left( \begin{array}{cc}
2 & -(z+z^{-1}) \\
-(z+z^{-1}) & 2 \end{array} \right). \end{equation}
The determinant is $2 - z^n - z^{-n}$.
\end{defn} 

After base change to $\Bbbk[p,p^{-1}]$, where $z = p^2$ and $q = p^{-n}$, setting $v_k = z^k x_k$ for $1 \le k \le n$ (not considered modulo $n$ for purposes of this formula) gives an identification of $V_z$ with $V_q$. This is Lemma \ref{lem:matchqz} below.

\begin{rem} Why do we identify $q^2 = z^{-n}$ rather than $q^2 = z^n$? In this paper there is no reason to prefer one convention over the other. Even in the quantum geometric Satake
equivalence, most scalars involve Laurent polynomials in $q$ which are invariant under $q \mapsto q^{-1}$. However, the braiding between quantum representations is not invariant under
$q \mapsto q^{-1}$. Preliminary computations seem to indicate that setting $z^n = q^{-2}$ will lead to less annoying formulas for the braiding when defined in terms of $z$, so we
suffer this annoyance now. \end{rem}

Let us now define the symmetries in question.

\begin{defn} \label{defn:symmetries} Let $\si$ denote the rotation operator on $\Om$, for which $\si(i) = i+1$. We let $\si$ act on $W_{\aff}$ by permuting indices, so $\si(s_i) =
s_{i+1}$, and similarly we let it act $\Bbbk$-linearly on $V_{\boring}$, $V_{\de/n}$, and $V_{z}$ by permuting indices. We also specify that $\si(\de/n) = \de/n$ and $\si(z) = z$.

The reflection operator requires more care, because simple reflections (which permute $i$ and $i+1$) are really centered around half-integers. Let $\tau$ be the automorphism of $W_{\aff}$ defined by $\tau(s_i) = s_{-i}$. Let $\tau$ be the $\Bbbk$-linear automorphism of $V_{\boring}$ defined by $\tau(y_i) = y_{1-i}$, and similarly define $\tau$ on $V_{\de/n}$ and $V_z$. We also specify that $\tau(\de/n) = -\de/n$ and $\tau(z) = z^{-1}$. \end{defn}

Note that, if one thinks of $z$ as a root of unity in $\CC$, then the action of $\tau$ on $V_z$ is conjugate-linear rather than linear.

\subsection{Basic properties of these realizations} \label{ssec:basicpropsofrealizations}

\begin{lem} \label{lem:easypermute} Inside $V_z$, the symmetric group $S_n \subset W_{\aff}$ permutes the vectors $z^k x_k$ for $1 \le k \le n$  according to their index. \end{lem}
	
\begin{proof} An easy computation. \end{proof}

\begin{lem} \label{lem:matchqz} Let $\Bbbk'$ be a $\Bbbk$-algebra which is a commutative domain, containing invertible elements $z$ and $q$ such that $q^2 = z^{-n}$. There is an isomorphism $V_q \ot_{\AC_q} \Bbbk' \to V_z \ot_{\AC_z} \Bbbk'$ of $W_{\aff}$ representations. This isomorphism sends
\begin{equation} v_k \mapsto z^k x_k \end{equation} for $1 \le k \le n$. The simple roots are rescaled by this isomorphism: we have
\begin{equation} \al_k \mapsto z^k \al_k \text{ for } 1 \le k \le n-1, \qquad \al_0 \mapsto q^{-1} \al_0. \end{equation}
The simple coroots are rescaled by the inverse scalars. \end{lem}

\begin{proof} An easy computation. \end{proof}

\begin{rem} Note that the identification of $V_q$ with $V_z$ is different from the identification of $V_{\de}$ with $V_{\de/n}$, as the roots are rescaled. This rescaling leads to a
changed Cartan matrix, and to a rescaled definition of Demazure operators. This accounts for the difference between our braid relations in $\NC(z,3)$ found in \eqref{zeR3}, and the braid relations found in \cite[Equations (3.9abc)]{EQuantumI}. \end{rem}

Let us pin down the action of the root lattice on $V_z$.

\begin{lem} \label{lem:easytranslate} Let $\La$ denote the subgroup of $\ZZ^n$ consisting of tuples of integers $(a_1, \ldots, a_n)$ such that $\sum a_i = 0$. There is an embedding $\La \to \End_{\AC_z}(V_z)$, where the action is given by
\begin{equation} x_i \mapsto z^{n a_i} x_i. \end{equation}
Then $\La \subset \End_{\AC_z}(V_z)$ is the isomorphic
 image of the action of $\La_{\rt} \subset W_{\aff}$.
\end{lem}

\begin{proof}
The translation operator $s_0 t_{\lon}$ acts by
\begin{equation}
\label{transop} s_0 t_{\lon}(x_1) = z^n x_1, \qquad s_0 t_{\lon}(x_n) = z^{-n} x^n, \qquad s_0 t_{\lon}(x_j) = x_j \text{ if } j \notin \{1,n\}. \end{equation}
The $S_n$-conjugates of this operator live inside $\La$, and generate it, just as for $\La_{\rt}$. \end{proof}	

\subsection{Dual realization} \label{ssec:dualrealization}

Let $V$ be a realization of a Coxeter group $W$. Recall that a simple reflection $s$ acts on a vector $v \in V$ by the formula
\begin{equation} s(v) = v - \langle \alpha_s^\vee, v \rangle \cdot \alpha_s. \end{equation}
Similarly, one can define an action of simple reflections on $\phi \in V^*$ by the formula
\begin{equation} s(\phi) = \phi - \langle \phi, \alpha_s \rangle \cdot \alpha_s^\vee. \end{equation}
The corresponding action of the Coxeter group $W$ on $V^*$ is the \emph{dual realization}, where one swaps the role of roots and coroots.

It it easy to verify from these formulas that $\langle s \phi, s v \rangle = \langle \phi, v \rangle$. Thus for any $w \in W$, $v \in V$, and $\phi \in V^*$ we have
\begin{equation} \label{eq:Wpreservespairing} \langle w \phi, wv \rangle = \langle \phi, v \rangle. \end{equation}
	
\begin{prop} \label{prop:antilineardual} A $\ZZ$-linear map $\psi$ between $\AC_z$-modules is called \emph{$z$-antilinear} if $\psi(zm) = z^{-1} \psi(m)$. The $z$-antilinear map $V_z \to V_z^*$ which sends
	\begin{equation} z \mapsto z^{-1}, \qquad v_i \mapsto v_i^*, \end{equation}
is a $W$-equivariant isomorphism between $V_z$ and $V_z^*$, sending roots to roots and coroots to coroots. \end{prop}

\begin{proof} This is evident from the formulas in Definition \ref{defn:Vz}. \end{proof}

We use the following corollary when studying roots and length functions. The proof is immediate from Proposition \ref{prop:antilineardual}.

\begin{cor} \label{cor:actiononcoroots} Let $w \in W_{\aff}$ and $d \in \ZZ$ and $j, k \in \Om$ be such that $w(\alpha_j) = z^d \alpha_k$. Then $w(\alpha_j^\vee) = z^{-d} \alpha_k^\vee$. \end{cor}

\begin{ex} The above corollary is consistent with the following computation, which uses \eqref{eq:Wpreservespairing}.
\begin{equation} 2 z^d = \langle \alpha_k^\vee, z^d \alpha_k \rangle =  \langle \alpha_k^\vee, w(\alpha_j) \rangle = \langle w^{-1}(\alpha_k^\vee), \alpha_j \rangle = 
	\langle z^d \alpha_j, \alpha_j \rangle = 2 z^d. \end{equation}
\end{ex}

\subsection{Embeddings of affine realizations} \label{ssec:embeddings}

Let us mention an embedding of realizations which is motivational, and we expect will be useful in future work. It is well-known classically that the affine reflection $s_0$ and the longest finite reflection $t_{\lon}$ generate an infinite dihedral group, i.e. a copy of the affine Weyl group of type $A_1$. This embedding of groups (though not an embedding of Coxeter systems) is consistent with an embedding of their reflection representations. Here we argue that this embedding is also consistent with their deformed reflection representations.

\begin{rem} In type $A_1$ there is little difference between $V_q$ and $V_z$, as $q^2 = z^{-2}$, the Cartan matrices agree, and the bases agree up to rescaling. For clarity we use $V_q$ below. \end{rem}
	
\begin{defn} In any realization of $W_{\aff}$, let $\alpha_{\lon}$ denote $s_{n-1} \cdots s_3 s_2 (\alpha_1)$. Define $\alpha_{\lon}^\vee$ similarly. \end{defn}
	
\begin{thm} \label{thm:A1embedding} Let $\Bbbk'$ be a $\Bbbk$-algebra which is a commutative domain, containing invertible elements $z$ and $q$ such that $q^2 = z^{-n}$. Let $W_{\aff, A_1}$ denote the affine Weyl group in type $A_1$, an infinite dihedral group generated by $t$ and $u$. Let $V_{q, A_1}$ denote the realization for $W_{\aff,
A_1}$ constructed in Definition \ref{defn:Vq}. There is an embedding of Coxeter groups $W_{\aff, A_1} \to W_{\aff}$ sending $t
\mapsto t_{\lon}$ and $u \mapsto s_0$. Correspondingly, there is an embedding of realizations $V_{q, A_1} \to V_z$ sending
\begin{equation} \label{eq:rootscorootsA1embedding} \alpha_t \mapsto \alpha_{\lon}, \qquad \alpha_u \mapsto
q z^{n-1} \alpha_0, \qquad \alpha_{\lon}^\vee \mapsto \alpha_t^\vee, \qquad q^{-1} z^{1-n} \alpha_0^\vee \mapsto \alpha_u^\vee. \end{equation} \end{thm}

Perhaps this theorem can be most easily summarized by considering the pairing matrix between $\{\alpha_{\lon}^\vee, \alpha_0^\vee\}$ and $\{\alpha_{\lon}, \alpha_0\}$ in $V_z$:
\begin{equation} \left( \begin{array}{cc} 2 & -z(1+z^{-n}) \\ -z^{-1}(1+z^n) & 2 \end{array} \right). \end{equation}
Conjugating by a diagonal matrix, we obtain the Cartan matrix \eqref{sl2qCartan}.

\begin{proof} Recall that $V_{q, A_1}$ has basis $v_1$ and $v_2$, with $t(v_1) = v_2$ and $u(v_1) = q^2 v_2$. We have
\begin{equation} \alpha_t = v_1 - v_2, \qquad \alpha_u = qv_2 - q^{-1} v_1, \qquad \alpha_t^\vee = v_1^* - v_2^*, \qquad \alpha_u^\vee = q^{-1} v_2^* - q v_1^*. \end{equation}
Meanwhile, $V_z$ has basis $\{x_i\}$ for $1 \le i \le n$, with $t_{\lon}(x_1) = z^{n-1} x_n$ and $s_0(x_1) = z^{-1} x_n$. We have
\begin{equation} \alpha_{\lon} = x_1 - z^{n-1} x_n, \qquad \alpha_0 = x_n - z x_1, \qquad \alpha_{\lon}^\vee = x_1^* - z^{1-n} x_n^*, \qquad \alpha_0^\vee = x_n^* - z^{-1} x_1^*. \end{equation}
It is straightforward to verify that the linear transformation
\begin{equation} v_1 \mapsto x_1, \qquad v_2 \mapsto z^{n-1} x_n, \end{equation}
will intertwine the action of $t$ with $t_{\lon}$, and the action of $u$ with $s_0$. Moreover, it will act on roots and coroots as in \eqref{eq:rootscorootsA1embedding}. \end{proof}

Because of this embedding, one expects that phenomena which occur for infinite dihedral groups when $q$ is specialized to a root of unity will carry over to $V_z$.

\begin{rem} Similarly, when $n = k \ell$ there is an embedding of the affine Weyl group of $S_{k}$ into the affine Weyl group of $S_n$, and of $V_{z^{\ell}, S_{k}}$ into $V_{z, S_n}$. \end{rem}
	
\subsection{Exponentiation} \label{ssec:exponentiation}

Let us note the relationship between $V_{\de/n}$ and $V_z$, due to \cite{EWKSBim} where it is studied in more detail. One can define $V_{\de/n}$ integrally, so that $W_{\aff}$ is acting
on a lattice. Then $W_{\aff}$ acts on the group algebra $R_K$ of this lattice, which is a Laurent polynomial ring on the variables $e^{\gamma_i}$ and $e^{\de/n}$. Note that $R_K$ has a
geometric interpretation as the equivariant $K$-theory of a point under the action of the torus of (the $GL_n$-analogue of) the affine Kac-Moody group. There is a $W_{\aff}$-equivariant
map $V_z \to R_K$ which sends $x_i \mapsto e^{\gamma_i}$ and $z \mapsto e^{\de/n}$. The reader should compare \eqref{zetarefrep} with \eqref{denrefrep}. Taking the Laurent polynomial
ring $\AC_z[x_1^{\pm}, \ldots, x_n^{\pm}]$ of $V_z$, we get a ring isomorphic to $R_K$.

Said briefly, the polynomial ring of $V_z$ is a subring of the group algebra of the lattice $V_{\de/n}$, intertwining the action of $W_{\aff}$. One can perform a similar operation to
relate $V_{\de}$ and $V_q$. It is worth emphasizing that $V_{\de/n}$ is rank $n+1$, while $V_z$ is rank $n$; the rank difference translates into a difference in base ring, as the
$W_{\aff}$-invariant span of $\de/n$ is absorbed into the scalars $\AC_z$ of $V_z$.

\begin{rem} \label{rem:betterthanroot1} One can define the ring $R_K$ analogously for any crystallographic Coxeter groups acting on its reflection representation. However, it will not
typically contain a $W$-invariant polynomial subring $\AC_z[x_1, \ldots, x_n]$ from which we could extract a reflection representation like $V_z$. In fact, even the root subrealization
of $V_{\de/n}$ spanned by roots will not admit a polynomial subring. What is special about type $A$ is the realization coming from the defining representation of $S_n$ rather than the
standard representation, i.e. acting on the $\gl_n$ torus rather than the $\sl_n$ torus. \end{rem}

\begin{rem} This warning is for the reader who might try to use the exponentiation map to prove facts about $V_z$ from the classical realization $V_{\de/n}$. The exponentiation map, though a $W_{\aff}$-equivariant map from $V_{\de/n}$ to $V_z$, does not send roots to roots. For example, $e^{\gamma_1 - \gamma_2 +
\de/n} = z x_1 x_2^{-1}$, which is different from the root $x_1 - z x_2 = e^{\gamma_1} - e^{\gamma_2 + \de/n}$. \end{rem}

\subsection{Roots} \label{ssec:roots}

For the remainder of this paper we work solely with $V_z$ and its specializations, having introduced the other realizations mostly for context. 

\begin{defn} A \emph{root} in $V_{z}$ is an element in the $W_{\aff}$-orbit of some simple root. We let $\Phi$ denote the set of roots. \end{defn}
	
The behavior of the roots when $n=2$ is quite different to the case when $n \ge 3$, and we focus on the latter here. For the $n=2$ case see Remark \ref{rmk:rootsn2}.

\begin{prop} Suppose $n \ge 3$. The roots in $V_{z}$ have the form
\begin{equation} \al = z^{k}(z^{i+ln}x_i - z^{j} x_j) \end{equation}
for $1 \le i \ne j \le n$ and $k,l \in \ZZ$. \end{prop}

\begin{proof} With Lemmas \ref{lem:easypermute} and \ref{lem:easytranslate}, it is easy to observe that both $S_n$ and $\La_{\rt}$ preserve this set of vectors, so we need only argue that every element in the set is a root.
	
Starting at $\al_i = x_i - z x_{i+1}$ and applying $s_i s_{i+1}$, we get $z(x_{i+1} - z x_{i+2})$. Continuing in this fashion, we can apply an element of $W_{\aff}$ to get to
$z^{k} \al_{i+k}$ for any $i$ and $k$. This implies that $\Phi$ is closed under multiplication by $z^{\pm 1}$. In particular, $z^i \al_i = z^i x_i - z^{i+1} x_{i+1}$ is a
root. The action of $S_n$ allows us to reach $z^{i}x_i - z^{j} x_j$ for any $i,j$. Finally, because $n \ge 3$ there is a simple root whose
translation operator in $\La_{\rt}$ will send $x_i \mapsto z^n x_i$ and $x_h \mapsto z^{-n} x_h$ for some $h \ne i, j$, and will fix $x_j$. Applying this operator $l$ times, we
get $z^{i+ln}x_i - z^{j} x_j$. \end{proof}

In ordinary Coxeter theory one has positive and negative roots. The length of an element in a Coxeter group is the number of positive roots that it sends to negative roots. Meanwhile,
as just seen, many roots in $V_z$ are colinear, since roots are closed under multiplication by $z^{\pm 1}$, as well as multiplication by $-1$. To find an analogue of positive roots,
should one choose a single root in each line, making a set $\Phi^1$, or should one partition the roots into two halves $\Phi^+$ and $\Phi^-$ related by a minus sign? These are two
different things, and both are useful, playing the role of the positive roots in different contexts: this idea was used repeatedly in \cite{BreMal1, BreMal2} when studying $G(m,1,n)$
and $G(m,m,n)$. They choose $\Phi^1 \subset \Phi^+$ and $\Phi^- = - \Phi^+$ appropriately, and define a length function on $G(m,1,n)$ and $G(m,m,n)$ so that the length is equal to the
number of roots in $\Phi^1$ which are sent to $\Phi^-$. We perform the same ritual for the affine Weyl group here.

Here is our preferred choice of one root in each line.

\begin{defn} For each ordered pair $(i,j) \in \Om$ with $i \ne j$, one can define their \emph{ordered distance} $|j-i|$ by choosing representatives in $\ZZ$ such that $1 \le j-i \le n-1$, and letting $|j-i| := j-i$ (the result is independent of choice). Let
\begin{equation} \al_{(i,j,l)} = x_i - z^{|j-i|+ln} x_{j} \end{equation}
for any pair $(i,j)$ as above, and for $l \in \ZZ$. Then let
\begin{equation} \label{eq:Phiplus} \Phi^1 = \{\al_{(i,j,l)} \mid (i,j) \in \Om \text{ with } i \ne j, l \ge 0\}. \end{equation}
In particular, the simple root $\al_{(i,i+1,0)} = \al_i$ is in $\Phi^1$.
\end{defn}

\begin{ex} Suppose $n=3$. Then $\Phi^1$ contains
\[ x_1 - z x_2, \; \; x_1 - z^4 x_2, \; \; x_1 - z^7 x_2, \ldots \]
for the ordered pair $(1,2)$, as well as
\[ x_2 - z^2 x_1, \; \; x_2 - z^5 x_1, \; \; x_2 - z^8 x_1, \ldots \]
for the ordered pair $(2,1)$. It also contains $x_1 - z^{2 + ln} x_3$ and $x_3 - z^{1 + ln} x_1$ and $x_2 - z^{1+ln} x_3$ and $x_3 - z^{2 + ln} x_2$, all for $l \ge 0$. \end{ex}

\begin{lem} \label{lem:posandneg} For each root in $\Phi$, there is a unique root in $\Phi^1$ which is colinear (over $\AC_z$). Let 
\begin{equation} \label{posandneg} \Phi^+ = \{ +z^k \Phi^1 \}_{k \in \ZZ}, \qquad \Phi^- = \{- z^k \Phi^1 \}_{k \in \ZZ}. \end{equation} 
Then there is a disjoint union $\Phi = \Phi^+ \sqcup \Phi^-$. \end{lem}

\begin{proof} Left to the reader. \end{proof}
	
Our notation reflects that $\Phi^1$ has one element of $\Phi$ from each line, while $\Phi^+$ and $\Phi^-$ are the intersection with positive and negative ``rays.''

\begin{prop} \label{onlyonerootgoesneg} For each $i \in \Om$, the simple root $\al_i$ is sent to $-\al_i$ by $s_i$, and every other element of $\Phi^1$ is sent to $\Phi^+$. More precisely, if $\al \in \Phi^1$ is not equal to $\al_i$, then it is sent to $z^\epsilon \be$ for some $\be \in \Phi^1$ and $\epsilon \in \{-1,0,1\}$. 
\end{prop}

\begin{proof} By symmetry, it suffices to check the case $i=n-1$. Clearly $s_i(\al_i) = - \al_i$. Any root $\al_{(j,k,l)} \in \Phi^1$ will be fixed by $s_i$ unless either $j$ or $k$ is in $\{n-1,n\}$. The remainder follows from the following computations.
\begin{subequations}
\begin{equation} s_{n-1}(\alpha_{(i,n-1,l)}) = \alpha_{(i,n,l)} \in \Phi^1, \qquad \text{for } 1 \le i < n-1. \end{equation}
\begin{equation} s_{n-1}(\alpha_{(n-1,i,l)}) = z \alpha_{(n,i,l)} \in z \Phi^1, \qquad \text{for } 1 \le i < n-1. \end{equation}
\begin{equation} s_{n-1}(\alpha_{(n-1,n,l)}) = z \alpha_{(n,n-1,l-1)} \in z \Phi^1, \qquad \text{for } l \ge 1. \end{equation}

\end{subequations}
Applying this calculation in reverse will show, for example, that $s_{n-1}(\alpha_{(n,i,l)}) \in z^{-1} \Phi^1$, and so forth.
\end{proof}

We do not use the following result in this paper. It follows the same standard ideas from \cite[Lemma 1.6 and Corollary 1.7]{HumpCox}, but one must show that these arguments still work up to powers of $z$.

\begin{prop} \label{prop:length} For all $w \in W_{\aff}$, let $n(w)$ be the number of roots in $\Phi^1$ sent to an element of $\Phi^-$ by $w$. Then $n(w) = \ell(w)$. \end{prop}

\begin{proof} First we prove the following two statements: for $i \in \Om$ we have
	\begin{equation} \label{eq:nincrease} w(\alpha_i) \in \Phi^+ \implies n(ws_i) = n(w) + 1, \end{equation}
	\begin{equation} w(\alpha_i) \in \Phi^- \implies n(ws_i) = n(w) - 1. \end{equation}
Consider $\alpha \in \Phi^1 \setminus \{\alpha_i\}$. By Proposition \ref{onlyonerootgoesneg}, $s_i(\alpha) = z^{\epsilon} \beta$ for some $\beta \in \Phi^1 \setminus \{\alpha_i\}$. So $w(\beta) \in \Phi^{\pm}$ if and only if $w s_i(\alpha) \in \Phi^{\pm}$. Hence the elements of $\Phi^1 \setminus \{\alpha_i\}$ sent to $\Phi^-$ by $w$ is in bijection with those sent to $\Phi^-$ by $ws_i$. The difference between $n(ws_i)$ and $n(w)$ is purely accounted for by whether $\alpha_i$ is sent to $\Phi^+$ or $\Phi^-$. Now it is easy to verify the two statements above.

By \eqref{eq:nincrease}, it is easy to see that $n(w) \le \ell(w)$, since both agree on the identity element.

Now suppose $n(w) < \ell(w) = r$, and let $w = s_{i_1} \cdots s_{i_r}$ be a reduced expression. There must be some $j \le r$ such that $s_{i_1} s_{i_2} \cdots s_{i_{j-1}} \alpha_{i_j} \in \Phi^-$; otherwise $n(w) = r$ by \eqref{eq:nincrease}. Applying simple reflections to $\alpha_{i_j}$ eventually sends it from positive to negative, and we keep track of the place where this happens. That is, there is some $1 \le k \le j-1$ such that $s_{i_{k+1}} \cdots s_{i_{j-1}} \alpha_{i_j} \in \Phi^+$ but $s_{i_k} s_{i_{k+1}} \cdots s_{i_{j-1}} \alpha_{i_j} \in \Phi^-$. Let $u = s_{i_{k+1}} \cdots s_{i_{j-1}} \in W_{\aff}$. The only elements of $\Phi^+$ which are sent to $\Phi^-$ by $s_{i_k}$ are powers of $z$ times $\alpha_{i_k}$. Thus
\begin{equation} \label{eq:alphajalphak} u \alpha_{i_j} = z^d \alpha_{i_k} \end{equation}
for some $d \in \ZZ$. 

We claim that \eqref{eq:alphajalphak} implies that $s_{i_k} u = u s_{i_j}$. For any $v \in V$ we have
\begin{equation} s_{i_k} u(v) = u(v) - \langle \alpha_{i_k}^\vee, u(v) \rangle \cdot \alpha_{i_k}, \end{equation}
\begin{equation} w s_{i_j}(v) = u(v - \langle \alpha_{i_j}^\vee, v \rangle \cdot \alpha_{i_j}) = u(v) - \langle \alpha_{i_j}^\vee, v \rangle \cdot u(\alpha_{i_j}) =
	u(v) - \langle \alpha_{i_j}^\vee, v \rangle \cdot z^k \alpha_{i_k}. \end{equation}
These agree for all $v$ if $\langle \alpha_{i_k}^\vee, u(v) \rangle = z^d \langle \alpha_{i_j}^\vee, v \rangle$. Using \eqref{eq:Wpreservespairing}, this is a consequence of Corollary
 \ref{cor:actiononcoroots}.
 
Now $s_{i_k} u s_{i_j}$ appears within the expression for $w$, and $s_{i_k} u s_{i_j} = u$. We can remove $s_{i_k}$ and $s_{i_j}$ from the expression for $w$, obtaining a shorter expression for $w$, which is a contradiction. \end{proof}

\begin{rem} \label{rmk:rootsn2}  For completeness, we now describe the roots when $n=2$. It turns out that no two roots are colinear, aside from multiplication by $\pm 1$. We set 
\begin{equation} \alpha_{(i,l)} = z^{-l}(x_i - z^{2 l + 1} x_{i+1}) \end{equation}
for $i \in \Om$ and $l \ge 0$. Then
\begin{equation} \Phi^1 = \Phi^+ = \{\alpha_{(i,l)} \}, \qquad \Phi^- = \{ - \alpha_{(i,l)}\}. \end{equation}
Then $\Phi = \Phi^+ \cup \Phi^-$ is the set of roots, and length is the number of positive roots sent to negative roots. This is easy to prove, because
\begin{equation} s_i(\alpha_{(i+1,l)}) = \alpha_{(i,l+1)}, \qquad s_i(\alpha_{(i,0)}) = -\alpha_{(i,0)}. \end{equation} \end{rem}


%

\subsection{All at a root of unity} \label{ssec:rou}

Now we set $\Bbbk = \CC$ and specialize $z$ to a complex number $\ze$. When $\ze = 1$, $V_z$ specializes to the boring representation $V_{\boring}$. In fact, when $\ze^n = 1$,
then $V_{\ze}$ is also isomorphic to $V_{\boring}$, as is evident from Lemmas \ref{lem:easypermute} and \ref{lem:easytranslate}.

\begin{defn} \label{defn:Vminbody} Fix $m \ge 2$. Suppose that $z$ is specialized to a primitive $(nm)$-th root of unity $\ze$ in $\CC$. Let $V_m := V_z \ot_{\CC[z, z^{-1}]} \CC$ be the specialization, a representation of $W_{\aff}$.  \end{defn}

By \eqref{transop}, we see that $(s_0 t_{\lon})^m$ acts trivially on $V_m$. By conjugation, this means that the $m$-th multiple of translation by any root acts trivially. Thus the
action of $W_{\aff}$ on $V_m$ factors through the quotient \begin{equation} W_m := W_{\aff} / m \La_{\rt} \cong S_n \ltimes (\ZZ/m\ZZ)^{n-1}. \end{equation} This quotient is also
known as the complex reflection group $G(m,m,n)$.

Let $\Phi_m$ be the image of the roots in $V_m$, or equivalently, the union of the orbits in $V_m$ of the simple roots under $W_m$. The image of $\Phi^1$ inside $V_m$ will not consist of non-colinear vectors. There is also no hope for a decomposition into positive and negative roots like in Lemma \ref{lem:posandneg}, as $\ze^{mn/2} = -1$ when $mn$ is even. Instead we just choose one vector from each line as follows.

\begin{defn} \label{defn:phiplusm} Let $\Phi^1_m \subset \Phi_m$ be defined as follows:
\begin{equation} \Phi^1_m = \{\al_{(i,j,l)} := x_i - \ze^{j-i+ln} x_j \mid 1 \le i < j \le n, 1 \le l \le m \}. \end{equation} \end{defn}

\begin{prop} \label{prop} The size of $\Phi^1_m$ is $m\binom{n}{2}$. The set $\Phi^1_m$ contains one exactly one root from each colinearity class in $\Phi_m$. Any element of $\Phi^1$ specializes to an element of $\Phi^1_m$. \end{prop}

\begin{proof} The proof is straightforward. The only comment to make is that the indices $i$ and $j$ are not treated symmetrically in \eqref{eq:Phiplus}, but the images of 
	\[ \{x_i - \ze^{|j-i|+ln} x_{j} \mid i \ne j \in \Om, l  \ge 0\} \quad \text{ and } \quad \{x_j - \ze^{|i-j|+ln} x_{i} \mid i \ne j \in \Om, l \ge 0\} \] in $V_m$ will agree. \end{proof}

\begin{ex} Suppose that $n=3$ and $m=2$. Note that $\zeta^3 = -1$. Here are the six roots in $\Phi^1_m$:
\begin{equation} x_1 - \ze x_2, \quad z_1 + \ze x_2, \quad x_2 - \ze x_3, \quad x_1 - \ze^2 x_3, \quad x_2 + \ze x_3, \quad x_1 + \ze^2 x_3. \end{equation}
The action of $s_1$ sends the first root to minus itself, fixes the second root, sends the third root to $\ze^{-1}$ times the fourth, and sends the fifth root to $\ze^{-1}$ times the sixth. \end{ex}

This collection of roots for $G(m,m,n)$ is common in the literature, see \cite{BreMal2, KirillovMaeno}. Let us state here the analogue of a well-known classical fact.

\begin{defn} \label{defn:delta}  Let $\De = \Pi_{\al \in \Phi^1_m} \al$, living in the symmetric algebra of $V_m$. \end{defn}

\begin{prop} \label{prop:deltaantiinvt} For any $k \in \Om$, $s_k(\Delta) = -\Delta$. \end{prop}

\begin{proof} Let $\hat{\Phi}^1_m = \{x_i - z^{j-i+ln} x_j \mid 1 \le i < j \le n, 1 \le l \le m \} \subset \Phi^1$ be a particular lift of $\Phi^1_m$ to $\Phi^1$ in $V_z$. By Proposition \ref{onlyonerootgoesneg}, for any $\al \in \hat{\Phi}^1_m$ with $\al \ne \al_k$, there exists $\be \in \Phi^1$ and $\epsilon \in \{-1, 0, 1\}$ with $s_k(\al) = z^{\epsilon} \be$. Let $\be' \in \hat{\Phi}^1_m$ have the same image in $V_m$ as $\be$. Inside $V_m$ (abusively using the same notation for roots in $V_z$ and their specializations in $V_m$) we have $s_k(\al) = \ze^{\epsilon} \be'$ and $s_k(\be') = \ze^{-\epsilon} \al$. Note that if $\be' = \al$ then $\epsilon = 0$, since the only eigenvalues of $s_k$ are $\pm 1$.

Now examining the product $\Delta$, it factors as a product of $\al$ with $s_k(\al) = \al$, a product of $\al \be'$ with $\al \ne \be'$ and $s_k(\al \be') = \al \be'$, and one copy of $\al_k$ with $s_k(\al_k) = -\al_k$. Thus $s_k(\Delta) = -\Delta$. \end{proof}

\begin{prop} \label{prop:deltasigmainvt} We have $\si(\Delta) = \Delta$ and $\tau(\Delta) = (-1)^{\binom{n}{2}} \ze^{-m\binom{n+1}{3}} \Delta$. \end{prop}

\begin{proof} Let us note that 
\begin{equation}  \ze^{n\binom{m+1}{2}} = \begin{cases} 1 & \text{ if $m$ is odd}, \\ -1 & \text{ if $m$ is even} \end{cases} = (-1)^{m-1}. \end{equation}
This is because $\ze$ is a primitive $nm$-th root of unity. If $m$ is odd then $m$ divides $\binom{m+1}{2}$. If $m$ is even then $\ze^{nm/2} = -1$, and $(-1)^{m+1} = -1$. Similarly, 
\begin{equation} \ze^{m \binom{n}{2}} = (-1)^{n-1}. \end{equation}
	
For $1 \le i < j \le n$ and $1 \le l \le m$, we have the enumeration $\al_{(i,j,l)} = x_i - \ze^{j-i+ln} x_j$ of $\Phi^1_m$. Whenever $j < n$ we have $\si(\al_{(i,j,l)}) = \al_{(i+1, j+1, l)})$. Meanwhile, 
\begin{equation} \si(\al_{(i,n,l)}) = -\ze^{n-i+ln} x_1 + x_{i+1} = -\ze^{n-i+ln}(x_1 - \ze^{i-n(l+1)} x_{i+1}). \end{equation}
As $l$ ranges from $1$ to $m$ within $\ZZ/m\ZZ$, $-(l+1)$ also ranges from $1$ to $m$ in $\ZZ/m\ZZ$, so every root $\al_{(1,i+1,l')}$ appears as $\si(\al_{(i,n,l)})$ exactly once, up to a scalar. Multiplying these scalars together, we compute that
\begin{equation} \frac{\si(\Delta)}{\Delta} = \prod_{1 \le i < n} \prod_{l \in \ZZ/m\ZZ} (-\ze^{n-i+ln}) = \prod_{1 \le i < n} \left[ (-1)^{m} \ze^{-im} \prod_{l \in \ZZ/m\ZZ} \ze^{n(l+1)} \right].  \end{equation}
The product over $l$ of $\ze^{n(l+1)}$ is $\ze^{n\binom{m+1}{2}}$, which is $(-1)^{m-1}$. This combines with the factor $(-1)^m$ to produce $-1$. So
\begin{equation} \frac{\si(\Delta)}{\Delta} = \prod_{1 \le i < n} -\ze^{-im} = (-1)^{n-1} \prod_{1 \le i < n} \ze^{-im}. \end{equation}
The product over $i$ of $\ze^{-im}$ is $\ze^{-m\binom{n}{2}}$, which is $(-1)^{n-1}$. So we conclude that
\begin{equation} \frac{\si(\Delta)}{\Delta} = 1. \end{equation}

The symmetry $\tau$ is defined so that $\tau(x_i) = x_{1-i}$ and $\tau(\ze) = \ze^{-1}$. Thus $\tau(x_1) = x_0 = x_n$, and $\tau(x_2) = x_{-1} = x_{n-1}$, and so forth. We have
\begin{equation} \tau(\al_{(i,j,l)}) = x_{n+1-i} - \ze^{i-j-ln} x_{n+1-j}
	= -\ze^{i-j-ln}\al_{(n+1-j,n+1-i,l)}. \end{equation}
So we need only compute the product, over all triples $(i,j,l)$, of $-\ze^{i-j-ln}$. As before, the product over $l$ of $-\ze^{-ln}$ is $(-1)^m \ze^{-n \binom{m+1}{2}} = -1$. The product over $l$ of $\ze^{i-j}$ is $\ze^{m(i-j)}$. So 
\begin{equation} \frac{\tau(\Delta)}{\Delta} = \prod_{1 \le i < j \le n} -\ze^{m(i-j)}. \end{equation}
It is a standard fact that
\begin{equation} \sum_{1 \le i < j \le n} (j - i) = \binom{n}{2} + \binom{n-1}{2} + \ldots + \binom{2}{2} = \binom{n+1}{3}, \end{equation}
from which the result follows immediately. \end{proof}

\begin{rem} \label{rmk:notpreservedbytaureally} The fact that $\tau(\Delta) = (-1)^a \ze^b \Delta$ (for some value of $a$ and $b$) does not contradict the fact that $\tau^2$ is the identity. Since $\tau$ is conjugate linear, we have
\begin{equation} \tau^2(\Delta) = \tau((-1)^a \ze^b \Delta) = (-1)^a \ze^{-b} \tau(\Delta) = (-1)^{2a} \ze^{-b+b} \Delta = \Delta. \end{equation}
Said another way, if $\ze$ is a complex number, then $\tau$ is only $\RR$-linear. The $\CC$-span of $\Delta$ is two-dimensional over $\RR$, and always affords the regular representation of the group of size 2 generated by $\tau$. When $n$ is relatively prime to $6$, then $n$ divides $\binom{n+1}{3}$, and the $\RR$-span of $\Delta$ is preserved by $\tau$, giving either the trivial or sign representation based on $\binom{n}{2}$. Otherwise, the $\RR$-span of $\Delta$ is not preserved by $\tau$. \end{rem}

\subsection{Classical viewpoints on $G(m,m,n)$} \label{ssec:otherviews}

While the representation $V_q$ of $W_{\aff}$ was new in \cite{EQuantumI}, and the more symmetric version $V_z$ first appears here, the specializations $V_m$ viewed as
representations of $G(m,m,n)$ are isomorphic to the standard reflection representation of $G(m,m,n)$ appearing in the literature. However, there are differences
between these viewpoints which we wish to emphasize.

Let us consider the less symmetric representation $V_q$ rather than $V_{z}$, because this more closely resembles the literature. The symmetric group $S_n \subset W_{\aff}$ acts merely
to permute the basis elements $\{v_1, \ldots, v_n\}$. Meanwhile, we have added a new simple reflection $s_0$ which swaps $v_1$ and $q^2 v_n$. We get a representation of $G(m,m,n)$ after
specializing $q^2$ to an $m$-th root of unity. Instead, the literature (c.f. \cite[Thm 3.25]{BMR} and \cite{Ramp}) frequently introduces a reflection $s_1'$ which swaps $v_1$ and $q^2
v_2$. For example, when $q^2 = -1$, this is exactly the usual simple reflection in type $D_n$. Obviously $s_1'$ is also a reflection in $W_{\aff}$, the conjugate of $s_0$ by the
transposition $(2n)$. Our representation $V_q \ot \CC$ (specializing $q$ to a primitive $2m$-th root of unity) is exactly the same as their reflection representation of $G(m,m,n)$,
described using different generators for the group. But $s_1'$ is not a simple reflection in $W_{\aff}$, and this leads to a different notion of simple roots, a different length
function, etcetera. It is a subtle difference which leads to many changes.


The ring $R_m$ and the invariant subring $R_m^{W_m}$ and the Frobenius trace $\pa_{W_m}$ (studied in the next chapter) will be the same in these two pictures, since these constructions
are intrinsic to the group and the representation. However, one obtains two different subalgebras of $\End_{R_m^{W_m}}(R_m)$, both containing the Demazure operators of $S_n$, and one
additionally generated by $\pa_{s_0}$, the other by $\pa_{s_1'}$. This latter subalgebra was studied by Rampetas \cite{Ramp}, while the former is our $\NC(m,m,n)$. We will discuss
Rampetas' results in more detail in \S\ref{ssec:rampetas}.

%% file: Polys.tex
\section{Symmetric and anti-symmetric polynomials} \label{sec:polys}

Let $z$ and $\AC_z$ and $V_z$ be as in Definition \ref{defn:Vz}. Let $R_z = \AC_z[x_1, \ldots, x_n]$ be the polynomial ring of $V_z$, equipped with its
action of the affine Weyl group $W_{\aff}$. We equip $R_z$ with a grading where $\deg x_i = 1$.

\begin{rem} One often equips $R_z$ with a grading where $\deg x_i = 2$ instead, to match the grading on the torus-equivariant cohomology of the point. \end{rem}

Let $\ze$ be a primitive $(mn)$-th root of unity in $\CC$. Let $R_m = \CC[x_1, \ldots, x_n]$ be the polynomial ring of $V_m$, equipped with its action of the affine Weyl group
$W_{\aff}$ from Definition \ref{defn:Vminbody}. The action factors through the quotient $W_m$. We equip $R_m$ with a grading where $\deg x_i = 1$.

The ring $R_m$ is commonly studied in the literature. As noted previously, many algebraic properties of $R_m$ do not depend on the choice of presentation of $W_m$. For sake of completeness we present our own proofs of some well-known results, e.g. the generating set of symmetric polynomials.

When we are not actively disambiguating between $R_m$ and $R_z$ or $W_m$ and $W_{\aff}$, we write $R = R_m$ and $W = W_m$.

\subsection{Symmetric polynomials}

\begin{prop} \label{prop:CST} The subring $R^{W}$ is a polynomial ring with generators in degrees $n, m, 2m, \ldots, (n-1)m$. The ring $R$ is free as a graded module over
$R^{W}$ with graded rank \begin{equation} \label{pim} \pi_m := \frac{(1-v^n)(1-v^m)(1-v^{2m}) \cdots (1-v^{(n-1)m})}{(1-v)^n} = (n)_v (m)_v (2m)_v \cdots ((n-1)m)_v, \end{equation}
where $(k)_v = 1 + v + \ldots + v^{k-1}$. \end{prop}

\begin{proof} Since $W$ is a finite group acting faithfully on the complex vector space $V_m$, and $W$ is generated by elements which act by reflections, then this is a consequence of the Shephard-Todd theorem \cite{ShephardTodd}. The degrees of $G(m,m,n)$ are well-known, and are also found in \cite[\S 6 (2)]{ShephardTodd}.  \end{proof}

\begin{lem} \label{lem:y} Let $y_k = \ze^{mk} x_k^m$. Then the $\CC$-span of $\yb_n = \{y_1, \ldots, y_n\}$ inside $R$ is a copy of the $W_{\aff}$-representation $V_{\boring}$. For $1 \le i \le n$, the elementary symmetric polynomials $e_i(\yb_n)$ are in $R^{W}$, and are algebraically independent. \end{lem}

\begin{proof} Clearly $s_k(y_j) = y_j$ for $j \notin \{k,k+1\}$. Also, \[ s_k(y_k) = \ze^{mk} (\ze x_{k+1})^m = \ze^{m(k+1)} x_{k+1}^m = y_{k+1}. \] This final computation works with indices modulo $n$, because $\ze^{mn} = 1$. Thus $W_{\aff}$ acts on the variables
$y_k$ just as it does in the representation $V_{\boring}$, where the action factors through $S_n$.

If $R' = \CC[y_1, \ldots, y_n]$ is the polynomial ring of $V_{\boring}$, then $(R')^{W_{\aff}} = (R')^{S_n}$ is the usual ring of symmetric polynomials (as the representation factors through
the symmetric group quotient of $W_{\aff}$). Since $(R')^{W_{\aff}} \subset R^{W}$ is a subalgebra, the elementary symmetric polynomials remain algebraically independent in $R_m$. \end{proof}

\begin{lem} Let $\zz := x_1 x_2 \cdots x_n$. As a ring, $R^{W}$ is a polynomial ring generated by $\BB := \{e_i(\yb_n)\}_{i=1}^{n-1} \cup \{\zz\}$. \end{lem}

\begin{proof} That $\zz$ is in $R^W$ is clear. Note that $e_n(\yb_n) = \ze^{m \binom{n}{2}} \prod_k x_k^m = \ze^{m \binom{n}{2}} \zz^m$. Thus $e_n(\yb_n)$ is not needed as a
generator of $R^W$. Since $\zz^m$ is not algebraic over the subring generated by $\{e_i(\yb_n)\}_{i=1}^{n-1}$, neither is $\zz$. Thus $\BB$ is algebraically independent. The
graded degrees of $\BB$ match those stated in Proposition \ref{prop:CST}, so they must generate all of $R^W$. \end{proof}

\begin{rem} Before specializing to a root of unity, the invariant polynomials on $V_z$ are generated by $\zz$. After all, any symmetric polynomial on $V_{\ze}$ descends to a symmetric polynomial on $V_m$ for all $m \ge 2$. For a given degree, when $m \gg 0$, the only invariant polynomials are powers of $\zz$.

Readers familiar with invariant rings for Coxeter groups will note the absence of a degree $2$ invariant polynomial. Typically one exists, corresponding to the invariant quadratic
form induced by the Cartan matrix. Note however that our Cartan matrix is not symmetric, and does not induce an invariant symmetric bilinear form. \end{rem}

\subsection{A basis for polynomials over symmetric polynomials}

\begin{prop} \label{prop:abasis} Choose a total order $i_1 < i_2< \ldots < i_n$ on the set $\{1, \ldots, n\}$. We let $\XX = \{x_{i_1}^{a_1} \cdots x_{i_n}^{a_n}\}$ denote the set of monomials satisfying the following properties. \begin{itemize}
\item $a_1 \le m(n-1)$.
\item $a_2 \le m(n-2)$ unless $a_1 = 0$, in which case $a_2 \le m(n-1)-1$.
\item $a_3 \le m(n-3)$ unless $a_1=0$ or $a_2 = 0$, in which case $a_3 \le m(n-2)-1$.
\item $\ldots$
\item $a_k \le m(n-k)$ unless $a_i = 0$ for some $i < k$, in which case $a_k \le m(n+1-k)-1$.
\end{itemize}
Then $\XX$ is a homogeneous basis of $R$ as a free $R^W$-module. \end{prop}

We prove Proposition \ref{prop:abasis} at the end of this section, after some lemmas. The definition of $\XX$ is a mouthful, but the formula \eqref{betterbounds} in the first
lemma gives a nicer restatement. The second lemma computes the graded count of $\XX$ (the sum over elements of $v$ raised to the degree of that element) to show it matches the
known graded rank from \eqref{pim}. In the third lemma, we show that $\XX$ spans $R$ over $R^W$. Thus $\XX$ must be a basis.

\begin{lem} Let $1 \le k \le n$. Inside $\XX$ there are $m \cdot \binom{n}{2}$ monomials for which $a_k = 0$ and $a_i \ne 0$ for $i < k$. Together, these enumerate $\XX$, which has size $n \cdot m \cdot \binom{n}{2}$. \end{lem}
	
\begin{proof} Choose a monomial in $\XX$. Either $a_n \le m(n-n) = 0$ or some previous exponent is zero. Therefore, some exponent is zero. There exists a unique $1 \le k \le n$ such that $a_k =0$ and $a_i \ne 0$ for $i < k$.
	
By adding one to all the exponents $a_i$ for $i > k$, we get a nicer version of the criteria defining $\XX$:
\begin{align} \label{betterbounds} \nonumber 1 \le a_1 \le m(n-1), \quad 1 \le a_2 \le m(n-2), \quad \ldots, \quad 1 \le a_{k-1} \le m(n-k+1),& \\
a_k = 0, &  \\
\nonumber 1 \le a_{k+1}+1 \le m(n-k), \quad  1 \le a_{k+2}+1 \le m(n-k-1),\quad \ldots, \quad 1 \le a_n+1 \le m.& \end{align}
The upper bound lowers by $m$ each time, skipping a beat when it passes over the zero exponent $a_k$.

It is now easy to count that the number of possibilities for the exponents is $m + 2m + \ldots + m(n-1) = m \cdot \binom{n}{2}$. \end{proof}

\begin{lem} The graded count of $\XX$ is equal to $\pi_m$, defined in \eqref{pim}. \end{lem}

\begin{proof} Consider the monomials in $\XX$ for which $a_k = 0$ and $a_i \ne 0$ for $i < k$. By \eqref{betterbounds}, if $k < n$ then $a_n$ can be any number between $0$ and $m-1$, the graded count of which is
\begin{subequations}
\begin{equation} 1 + v + \ldots + v^{m-1} = \frac{v^m - 1}{v-1} = (m)_v. \end{equation}
Similarly, for any $j$ with $1 \le j \le n-k$, the graded count of the choices of $a_{n+1-j}$ is 
\begin{equation} 1 + v + \ldots + v^{jm-1}  = (jm)_v. \end{equation}
However, for any $i$ with $1 \le i \le k-1$, we have $1 \le a_i \le m(n-i)$, the graded count of which is
\begin{equation} v + v^2 + \ldots + v^{m(n-i)} = v \cdot ((n-i)m)_v. \end{equation}
\end{subequations}
In conclusion, the overall graded count of monomials with $a_k = 0$ and $a_i \ne 0$ for $i < k$ is the product of these choices, 
\begin{equation} v^{k-1} (m)_v (2m)_v \ldots ((n-1)m)_v. \end{equation}

Taking the sum over all $1 \le k \le n$ we get that the graded count of $\XX$ is
\begin{align} \nonumber  (m)_v (2m)_v & \cdots ((n-1)m)_v \left[ 1 + v + \ldots + v^{n-1} \right] = \\
	 & (m)_v (2m)_v \cdots ((n-1)m)_v (n)_v = \pi_m. \qedhere \end{align}
\end{proof}

To prove that $\XX$ spans $R$ over $R^{W_m}$, it is more convenient to reformulate the question.

\begin{lem} Let $I$ denote the ideal inside $R$ generated by the positive degree elements of $R^{W_m}$, or equivalently, the ideal generated by $\BB$. The set $\XX$ descends to a
spanning set for the coinvariant ring $C := R/I$. \end{lem}

\begin{proof}
For ease of discussion we prove the result for the standard total order $1 < 2 < \ldots$, but the argument would work in the same way for any other order of the monomials. Recall that $y_k = \ze^{mk} x_k^m$. We let $\yb_k$ denote the family of variables $\{y_1, \ldots, y_k\}$, and $\yb^{n-k} = \{y_{k+1}, \ldots, y_n\}$, for the purpose of considering symmetric polynomials like $e_i(\yb_k)$.

Inside $C$ we already have the relations
\begin{subequations} \label{relationsmodRWplus}
\begin{equation} \label{ekreln} h_k(\yb_n) = 0, \quad k > 0. \end{equation}
Let us recall the standard argument which deduces the relations
\begin{equation} \label{hkreln} h_{n-k+1}(\yb_k) = 0, \qquad k > 0. \end{equation}
A well-known property of symmetric polynomials is that
\begin{equation} h_c(\yb_k) = \sum_{a+b=c} (-1)^b h_a(\yb_n) e_b(\yb^{n-k}). \end{equation}
In the quotient where \eqref{ekreln} holds, only the $a=0$ term survives, giving $h_c(\yb_k) \equiv \pm e_c(\yb^{n-k})$. But the latter vanishes when $c > n-k$, because there aren't enough variables.

Inside $C$ we also have the relation
\begin{equation} \label{zreln} x_1 x_2 \cdots x_n = 0. \end{equation}
We claim that we also have the relation
\begin{equation} \label{hkplusreln} x_1 \cdots x_k h_{n-k}(\yb_{k}) = 0 \end{equation}
\end{subequations}
for any $1 \le k \le n$. Let us prove this result by descending induction on $k$. The $k=n$ case is precisely \eqref{zreln}.  Assume that \eqref{hkplusreln} holds for $k+1$. Note that
\begin{equation} h_i(\yb_k) = h_i(\yb_{k+1}) - y_{k+1} h_{i-1}(\yb_{k+1}). \end{equation}
Thus
\begin{equation} \label{foobar1} x_1 \cdots x_k h_{n-k}(\yb_{k}) = x_1 \cdots x_k h_{n-k}(\yb_{k+1}) - x_1 \cdots x_k y_{k+1} h_{n-k-1}(\yb_{k+1}). \end{equation}
The first term on the RHS of \eqref{foobar1} vanishes by \eqref{hkreln}. Meanwhile, $x_{k+1}$ divides $y_{k+1}$ so the second term is zero by the $k+1$ case of \eqref{hkplusreln}. Thus the LHS of \eqref{foobar1} is zero, as desired. In particular, when $k=1$, we see that $x_1 h_{n-1}(y_1) = x_1^{(n-1)m+1} = 0$.

We will use the relations in \eqref{relationsmodRWplus} to replace the lexicographically-maximal monomial with a linear combination of other monomials. The lexicographically maximal
monomial in \eqref{hkreln} is $x_k^{(n+1-k)m}$ (which up to a scalar is $y_k^{n+1-k}$). The lexicographically-maximal monomial in \eqref{hkplusreln} is $x_1 x_2 \cdots x_{k-1}
x_k^{(n-k)m+1}$. Any monomial not in $\XX$ will be divided by one of these lexicographically-maximal monomials, so it lies in the span of monomials which come earlier in the
lexicographic order. By induction on the lexicographic order, monomials in $\XX$ will span all monomials. \end{proof}

\begin{proof}[Proof of Proposition \ref{prop:abasis}] Because $R$ is free over $R^{W_m}$, $\XX \subset R$ will be a basis of $R$ over $R^{W_m}$ if and only if it descends to a
basis of $C$ over $\CC$. Spanning was proven in the preceding lemma, and linear independence follows because the graded count of $\XX$ matches the graded dimension of $C$. \end{proof}

\begin{ex} \label{ex:dualbasesn3} Suppose that $n=3$. Each of the following sets is a basis for $R$ over $R^W$:
\begin{align} \nonumber \XX_{1<2<3} & = \{1\} \cup \{x_1^a \mid 1 \le a \le 2m \} \cup \{x_2^b \mid 1 \le b \le 2m-1\} \cup \{x_3^c \mid 1 \le c \le m-1\} \\ \nonumber & \cup \{x_1^a x_2^b \mid 1 \le a \le 2m,\; 1 \le b \le m\} \cup \{x_1^a x_3^c \mid 1 \le a \le 2m, 1 \le c \le m-1 \} \\ & \cup \{x_2^b x_3^c \mid 1 \le b \le 2m-1, 1 \le c \le m-1\}. \end{align}
\begin{align} \nonumber \XX_{3<2<1} =& \{1\} \cup \{x_3^c \mid 1 \le c \le 2m \} \cup \{x_2^b \mid 1 \le b \le 2m-1\} \cup \{x_1^a \mid 1 \le a \le m-1\} \\ \nonumber & \cup \{x_3^c x_2^b \mid 1 \le c \le 2m,\; 1 \le b \le m\} \cup \{x_3^c x_1^a \mid 1 \le c \le 2m, 1 \le a \le m-1 \}\\ & \cup \{x_2^b x_1^a \mid 1 \le b \le 2m-1, 1 \le a \le m-1\}. \end{align}
\end{ex}

\subsection{Antisymmetric polynomials and antisymmetrization}

\begin{defn} The \emph{antiinvariant polynomials} $R^{-W}$ are those polynomials $f \in R$ for which $wf = (-1)^{\ell(w)} f$ for any $w \in W_{\aff}$.\end{defn}

This definition makes sense for any representation of $W_{\aff}$, not necessarily factoring through a finite quotient $W_m$. However, normally only the zero polynomial is antiinvariant, the proof being similar to the proof of Proposition \ref{prop:rootsdivideanti} below. When the representation factors through $W_m$, the antiinvariants become nontrivial, and we can produce an antisymmetrization operator.

\begin{lem} The sign representation of $W_{\aff}$ descends to $W_m$. \end{lem}

\begin{proof} Recall that $\La_{\rt}$ is generated by conjugates of $s_0 t_{\lon}$. All these conjugates are even length elements of $W_{\aff}$. So $\La_{\rt}$ already acts
trivially on the sign representation of $W_{\aff}$. \end{proof}

For $w \in W_m$ we let $\ell(w)$ be the length of the shortest representative of $w$ in $W_{\aff}$. Then $w$ acts on the sign representation by $(-1)^{\ell(w)}$.

\begin{defn} The \emph{antisymmetrization operator} $A \co R \to R^{-W}$ is 
\begin{equation} A \co f \mapsto \sum_{w \in W_m} (-1)^{\ell(w)} w(f). \end{equation}
Viewing $S_n$ as a subgroup of $W_m$ generated by $\{s_1, \ldots, s_{n-1}\}$, the \emph{finite antisymmetrization operator} $A' \co R \to R^{-S_n}$ is
\begin{equation} A' \co f \mapsto \sum_{w \in S_n} (-1)^{\ell(w)} w(f). \end{equation}
\end{defn}

Let us consider what happens when $A$ is applied to a monomial.

\begin{prop} \label{prop:antisymmonomial} Let $b = x_1^{a_1} \cdots x_n^{a_n}$ be a monomial in $R$. If $m$ divides $a_i-a_j$ for all $1 \le i,j \le n$, then $A(b) = m^{n-1} A'(b)$. If there are some $i,j$ such that $m$ does not divide $a_i-a_j$, then $A(b) = 0$. In particular, $A(b) = 0$ if $a_i = a_j$ for some $i \ne j$.  \end{prop}
	
\begin{proof} Let $a$ be the minimum exponent among the $a_i$, and let $c = x_1^a \cdots x_n^a$. Then $c$ is $W$-invariant, and $b = cb'$. Thus $A(b) = c A(b')$. If we can prove these results for $b'$ we deduce them for $b$. Thus we may assume $a_l = 0$ for some index $l$, and we choose such an $l$.

Let $j$ be some index for which $m$ does not divide $a_j$. Necessarily $j \ne l$. Inside $\La_{\rt}$ is an element $\la$, a conjugate of $s_0 t_{\lon}$, which acts by the following formula:
\begin{equation} \la(x_k) = x_k \text{ for } k \ne j, l, \qquad \la(x_j) = \ze^n x_j, \qquad \la(x_l) = \ze^{-n} x_l. \end{equation}
Moreover, $\la$ has order $m$ inside $\La_{\rt}/m\La_{\rt}$, and has even length. Let $H$ be the subgroup of $W_m$ generated by $\la$, and let $[G/H]$ denote any set of coset representatives. Then
\begin{equation} A(b) = \sum_{w \in [G/H]} (-1)^{\ell(w)} w \cdot \left( \sum_{k=1}^{m} \la^k(b) \right), \end{equation}
and
\begin{equation} \sum_{k=1}^{m} \la^k(b) = (\sum_{k=1}^m \ze^{ka_jn}) b = 0. \end{equation}
The equality with zero follows because $\ze^{a_j n}$ is a non-trivial $m$-th root of unity.

Now suppose that $m$ divides $a_i$ for all $i$. Then $\La_{\rt}$ acts trivially on $b$. There are $m^{n-1}$ elements of $W_m$ in the image of $\La_{\rt}$, and they all have even length. Thus
\begin{equation} A(b) = \sum_{w \in S_n} (-1)^{\ell(w)} w \cdot \left( \sum_{\la \in \La_{\rt}} \la(b) \right) = m^{n-1} \sum_{w \in S_n} (-1)^{\ell(w)} w(b) = m^{n-1} A'(b). \end{equation}
\end{proof}

\begin{prop} \label{prop:rootsdivideanti} Recall the element $\De$ from Definition \ref{defn:delta}.
The antiinvariants $R^{-W_m}$ are a free module of rank $1$ over the invariants $R^{W_m}$, generated by $\De$. \end{prop}

This is a fairly standard argument, adapted from the case of finite Coxeter groups.

\begin{proof} Let $f \in R^{-W}$, and $i \in \Om$. We think about $f$ as a polynomial function on $V_m^*$, and use the Nulstellensatz. The fact that $s_i f = -f$ implies that $f$ vanishes on any point fixed by $s_i$. The fixed point set of $s_i$ is the hyperplane cut out by $\al_i = 0$. Thus $\al_i$ divides $f$.
	
Because one root $\al$ divides $f$, then all roots divide $f$. This is because $w(\al)$ divides $w(f) = \pm f$, and all roots are in the same $W_{\aff}$ orbit up to a power of
$\ze$. So any $f \in R^{-W}$ has the form $g \De$ for some $g \in R$.

By Proposition \ref{prop:deltaantiinvt}, $\De \in R^{-W}$. If $g \De \in R^{-W}$ then for any $i \in \Om$ we have
\begin{equation} g \De = - s_i(g \De) = - s_i(g) s_i(\de) = s_i(g) \De, \end{equation} whence $g = s_i(g)$. Consequently $g \in R^{W}$. \end{proof}

\begin{defn} \label{defn:J} Let $J \co R \to R^{W_m}$ denote the operator
\begin{equation} J \co f \mapsto \frac{A(f)}{m^{n-1} \De}. \end{equation}
\end{defn}

\begin{lem} The operator $J$ is a well-defined, $R^{W}$-linear map of degree $- m \binom{n}{2}$. \end{lem}

\begin{proof} The operator $J$ is well-defined (rather than mapping to the fraction field) by Proposition \ref{prop:rootsdivideanti}. It is $R^{W}$-linear by construction.  \end{proof}

Note that $m^{n-1}$ divides the numerator of $J$ by Proposition \ref{prop:antisymmonomial}, so even if we work integrally $J$ is still well-defined.

\begin{lem} \label{lem:stair1} Define the \emph{$m$-staircase polynomial} as
\begin{equation} \stair := x_1^{(n-1)m} x_2^{(n-2)m} \cdots x_{n-1}^m. \end{equation}
Then $J(\stair) = 1$. \end{lem}

\begin{proof} By Proposition \ref{prop:antisymmonomial}, $J(\stair) = A'(\stair)/\De$. For degree reasons we know that $A'(\stair)$ is a scalar multiple of $\De$, and we need only
determine the scalar. Since $S_n$ has no stabilizer when acting on $\stair$, we need only compute the coefficient of $\stair$ inside $\De$. There are exactly $(n-1)m$ roots of the
form $x_1 - \ze^k x_i$ inside $\Phi_m^+$, and these are the only roots with $x_1$, so to get exponent $x_1^{(n-1)m}$ we need to choose the $x_1$ factor from each. Similarly, there
are $(n-2)m$ remaining roots of the form $x_2 - \ze^k x_i$ (for $i > 2$), and we must choose the $x_2$ factor from each. Continuing, there is a unique way to obtain the monomial
$\stair$ inside the product $\De$, and the coefficient is $1$. \end{proof}

\begin{lem} \label{lem:uniqueuptoscalar} Let $\pa$ be any homogeneous $R^{W}$-linear map $R \to R^{W}$ of degree $-m\binom{n}{2}$. Then $\pa$ is a scalar multiple of $J$. More precisely, $\pa = \pa(\stair) \cdot J$. \end{lem}

\begin{proof} An $R^{W}$-linear map is determined by what it does to the basis $\XX$. For degree reasons, $\pa$ kills any basis element of degree $<m \binom{n}{2}$. The basis $\XX$ has a unique element of top degree $m\binom{n}{2}$, which is $\stair$, and $\stair$ must be sent to an element of degree zero, i.e. a multiple of the identity. So the space of such maps is one-dimensional, identified by where it sends $\stair$, and $J$ corresponds to $1$ by the previous lemma. \end{proof}

\begin{cor} \label{cor:sitauonJ} The action of $\si$ on the span of $J$ is the trivial representation. \end{cor}

\begin{proof} The operator $\si$ intertwines with the antisymmetrization operator $A$, since Dynkin diagram automorphisms preserve the sign representation. By Proposition \ref{prop:deltasigmainvt}, $\si$ fixes $\Delta$, and hence fixes $J$. \end{proof}

\begin{rem} The operator $\tau$ also intertwines with $A$, since it preserves the sign representation. So it acts on the complex span of $J$, c.f. Remark \ref{rmk:notpreservedbytaureally}. \end{rem}

\subsection{Frobenius trace via antisymmetrization} \label{ssec:frobtrace}


\begin{thm} \label{thm:JFrob} The map $J \co R \to R^W$ is a Frobenius trace map. \end{thm}

First let us note a presumably well-known fact.

\begin{lem} \label{lem:finitedualpair} Consider the standard action of $S_n$ on a polynomial ring $R_1 = \CC[y_1, \ldots, y_n]$ which permutes the variables. Let $\stair_1 = y_1^{n-1} y_2^{n-2} \cdots y_{n-1}$, and let $\XX_1$ denote the set of monomials dividing $\stair_1$. Let $\stair_2 = y_2 y_3^2 \cdots y_n^{n-1}$, and let $\XX_2$ denote the set of monomials dividing $\stair_2$. Let $J_1$ denote the usual Frobenius trace $R_1 \to R_1^{S_n}$. Then $\XX_1$ and $\XX_2$ pair nondegenerately under $J_1$.
\end{lem}

\begin{proof} The Schubert polynomials (obtained by applying Demazure operators to $\stair_1$) span the same subspace as $\XX_1$, and are related to $\XX_1$ by an invertible integral change of basis matrix. We learned this fact from \cite[Proposition 3.4]{LauSL2}.
Similarly, the dual Schubert polynomials (obtained by applying Demazure operators to $\stair_2$) are related to $\XX_2$ in the same way. Since Schubert and dual Schubert bases are, indeed, dual under $J_1$, the pairing of $\XX_1$ and $\XX_2$ must be nondegenerate. \end{proof}

\begin{proof}[Proof of Theorem \ref{thm:JFrob}] We claim that it suffices to prove that $J$ descends to a Frobenius trace map from $C$ to $\CC$. Here is the standard argument. Suppose one has homogeneous bases for $R$ over $R^W$ which descend to dual bases for $C$ over $\CC$. Order these bases by degree. Then the pairing matrix in $R$ must be upper-triangular for degree reasons, with $1$s on the diagonal. Hence it is non-degenerate.
	
To prove that $J$ is a Frobenius trace map on $C$, we need bases $\XX = \{b_i\}$ and $\XX'
= \{c_i\}$ for $C$ such that the matrix $J(b_i c_j)$ of complex numbers is non-degenerate. We claim that we can choose two bases of monomials constructed as in Proposition \ref{prop:abasis}: the
basis $\XX$ for the order $1 < 2 < \ldots < n$, and the basis $\XX'$ for the order $n < \ldots < 2 < 1$. See Example \ref{ex:dualbasesn3} for the case $n=3$.

We claim that the pairing matrix of $\XX$ against $\XX'$ is block upper-triangular. Moreover, each block on the diagonal is $M$, the pairing matrix between $\XX_1$ and $\XX_2$ from Lemma \ref{lem:finitedualpair}. This proves the desired nondegeneracy.

Suppose that $p = \prod x_i^{a_i} \in \XX$ and $p' = \prod x_i^{b_i} \in \XX'$ are monomials such that $\deg pp' = m\binom{n}{2}$. If $z = x_1 \ldots x_n$ divides $pp'$ then $J(pp') = z J(\frac{pp'}{z}) = 0$ for degree reasons. Thus we can assume that some variable $x_i$ has zero exponent in both $p$ and $p$. By Proposition \ref{prop:antisymmonomial}, if some exponent $a_i + b_i$ in $pp'$ is not a multiple of $m$ then $J(pp')=0$. Moreover, the exponents $a_i + b_i$ must all be distinct, or $J(pp')=0$. As a consequence, the only possibility for $J(pp') \ne 0$ in this degree is if there is an equality of sets
\begin{equation} \{a_i + b_i\}_{i=1}^n = \{m(n-1), \ldots, 2m, m, 0\}. \end{equation}
In this case $J(pp') = (-1)^{\ell(w)}$ where $w$ is the permutation which matches these two ordered sets. This follows by Lemma \ref{lem:stair1} since $A'(wf) = (-1)^{\ell(w)} f$ for $w \in S_n$.

Let $i$ be the smallest index such that $a_i = 0$, and let $j$ be the largest index such that $b_j= 0$. If $i > j$ then there is no common index $k$ for which $a_k = b_k = 0$, and
hence $J(pp')=0$. This is our first upper-triangularity result: the pairing matrix for $J$ on $\XX$ and $\XX'$ is a block upper-triangular matrix, where the diagonal blocks correspond to places where $i=j$. We need only
verify that $J$ is nondegenerate on each block.

For the rest of this proof we assume that $a_k = 0$, $a_i \ne 0$ for $i < k$, $b_k = 0$, $b_j \ne 0$ for $j > k$. By the constructions of $\XX$ and $\XX'$ we see that
\begin{itemize} \item $1 \le a_1 \le (n-1)m$ and $0 \le b_1 \le m-1$,
	\item $1 \le a_2 \le (n-2)m$ and $0 \le b_2 \le 2m-1$,
	\item $\ldots$
	\item $1 \le a_{k-1} \le (n-k+1)m$ and $0 \le b_{k-1} \le (k-1)m-1$
	\item $a_k = 0$ and $b_k = 0$
	\item $0 \le a_{k+1} \le (n-k)m - 1$ and $1 \le b_{k+1} \le km$
	\item $\ldots$
	\item $0 \le a_{n-1} \le 2m-1$ and $1 \le b_{n-1} \le (n-2)m$
	\item $0 \le a_n \le m-1$ and $1 \le b_n \le (n-1)m$.
\end{itemize}
Let us define integers $c_i$ and $d_i$ for $1 \le i \le n-1$ as follows. We have $c_i = a_i -1$ for $i < k$ and $c_i = a_{i+1}$ for $i \ge k$. We have $d_i = b_i$ for $i < k$, and $d_i = b_{i+1}-1$ for $i \ge k$. Then we see that
\begin{itemize} \item $0 \le c_1 \le (n-1)m-1$ and $0 \le d_1 \le m-1$,
	\item $0 \le c_2 \le (n-2)m-1$ and $0 \le d_2 \le 2m-1$,
	\item $\ldots$
	\item $0 \le c_{k-1} \le (n-k+1)m-1$ and $0 \le d_{k-1} \le (k-1)m-1$
	\item $0 \le c_k \le (n-k)m - 1$ and $0 \le d_k \le km-1$
	\item $\ldots$
	\item $0 \le c_{n-2} \le 2m-1$ and $0 \le d_{n-2} \le (n-2)m-1$
	\item $0 \le c_{n-1} \le m-1$ and $0 \le d_{n-1} \le (n-1)m-1$.
\end{itemize}
The pairing is zero unless $c_i + d_i \equiv -1$ modulo $m$, in which case $J(pp')$ is the sign of the permutation which puts
\[ \{c_i + d_i\} \to \{(n-1)m-1, \ldots, m-1\} \]
in bijection, or zero if the sets are not equal. Note that this description is independent of $k$! Thus all diagonal blocks of the pairing matrix are the same.

For each $1 \le i \le n-1$, pick some number $0 \le g_i \le m-1$. Consider the basis elements for which $c_i \equiv g_i$ modulo $m$. These can only pair against basis elements for which $d_i \equiv -1-g_i$ modulo $m$, and vice versa. Thus each diagonal block of the pairing matrix is itself a block diagonal matrix, with one block for each vector $g = (g_1, \ldots, g_{n-1})$.

Fix $g$. We define integers $e_i$ and $f_i$ such that $e_i = c_i - g_i / m$ and $f_i = (d_i - (m-1-g_i))/m$ for each $1 \le i \le n-1$. Then we have
\begin{itemize} \item $0 \le e_1 \le (n-2)$ and $0 \le f_1 \le 0$,
	\item $0 \le e_2 \le (n-3)$ and $0 \le f_2 \le 1$,
	\item $\ldots$
	\item $0 \le e_k \le (n-k-1)$ and $0 \le f_k \le (k-1)$
	\item $\ldots$
	\item $0 \le e_{n-1} \le 0$ and $0 \le f_{n-1} \le (n-2)$.
\end{itemize}
The pairing $J(pp')$ is the sign of the permutation which puts
\[ \{e_i + f_i\} \to \{n-2,\ldots,1,0\} \]
in bijection, or zero if the sets are not equal.

But the possible integers $e_i$ and $f_i$ and their pairings exactly match the situation of $\XX_1$ and $\XX_2$ via $J_1$.	
\end{proof}

\begin{cor} \label{cor:nonzeroonstairgoodenough} An $R^{W}$-linear map of degree $-m\binom{n}{2}$ is a Frobenius trace if and only if it sends $\stair$ to an invertible scalar. \end{cor}

\begin{proof} This follows by Lemma \ref{lem:uniqueuptoscalar}, and the fact that a scalar multiple of a Frobenius trace is a Frobenius trace if and only if the scalar is
invertible. \end{proof}

\subsection{On the proof that $J$ is a Frobenius trace} \label{ssec:proofcomparison}

In \S\ref{ssec:frobtrace} we proved that $J$ is a Frobenius trace, by explicitly evaluating its action on products of monomials in a basis for $R$ over $R^{W_m}$. This ad hoc solution appears to contrast with Demazure's lovelier proof, so we wish to explain why Demazure's proof no longer works in our context, and what interesting questions this raises.

Let us explain Demazure's proof of the analogous theorem for Coxeter groups. Let $(W,S)$ be a finite Coxeter system, acting on a realization with positive roots $\Phi^+$. Let $A$ be
the anti-symmetrization map, and let \[ J(f) = \frac{A(f)}{\prod_{\al \in \Phi^+} \al}. \]
Demazure wishes to determine when $J$ is a Frobenius trace even when working over $\ZZ$ (he only treats Weyl groups), so he assumes\footnote{This is related to the index of torsion defined in \cite[\S 5]{Demazure} and the assumption that this index is $1$ in \cite[Theorem 2]{Demazure}.} the existence of a polynomial $\stair$ for which $J(\stair) = 1$. He proves this (necessary) condition is also sufficient.

Consider the element $\pa_{w_0} := \pa_{s_{i_1}} \pa_{s_{i_2}} \ldots \pa_{s_{i_d}}$ for a reduced expression of the longest element $w_0$ of a Coxeter group. This agrees
with $J$ up to scalar for degree reasons, and Demazure proves \cite[Lemmas 3 and 4]{Demazure} that the scalar is $1$ (an argument involving leading terms when expanding both operators over $Q[W]$, where $Q$ is the fraction field of $R$).

Demazure proves that $\pa_{w_0}$ is a Frobenius trace as follows, c.f. \cite[Proposition 4, Corollary afterwards, Theorem 2]{Demazure}. Recall the twisted Leibniz rule
\[ \pa_s(fg) = \pa_s(f) g + s(f) \pa_s(g). \]
By the nil-quadratic relation, $\pa_{w_0} \pa_s = 0$. Also, $\pa_{w_0}$ kills anything in $R^s$, so $\pa_{w_0}(sf) = -\pa_{w_0}(f)$.
Thus 
\begin{equation}\label{leibniz} \pa_{w_0}(\pa_s(f) g) = - \pa_{w_0}(s(f) \pa_s(g)). \end{equation}
Now consider the pairing of $\pa_w(\stair)$ with $\pa_x(\stair)$. We have
\begin{equation} \pa_{w_0}(\pa_w(\stair) \pa_x(\stair)) = (-1)^{\ell(w)} \pa_{w_0}(w(\stair) \pa_{w^{-1}} \pa_x(\stair)). \end{equation}
If $\ell(w^{-1} x) < \ell(w^{-1}) + \ell(x)$ then $\pa_{w^{-1}} \pa_x = 0$. The only way the result will be a nonzero scalar (i.e. degree zero) is if $\ell(w^{-1}) + \ell(x) = \ell(w_0)$ and $w^{-1} x = w_0$, or equivalently, if $x = w w_0$. If $x = w w_0$ then we get
\begin{equation} \pa_{w_0}(\pa_w(\stair) \pa_{w w_0}(\stair)) = (-1)^{\ell(w)} \pa_{w_0}(w(\stair) \pa_{w_0}(\stair)) = \pa_{w_0}(\stair) = 1. \end{equation}
Thus $\{\pa_w(\stair)\}$ and $\{\pa_{w w_0}(\stair)\}$ are \emph{almost dual bases}: their pairing matrix is upper-triangular, with $1$s on the diagonal, and all non-zero elements above the diagonal being of positive degree. The existence of almost dual bases implies the existence of dual bases (though it does not make it easy to compute them explicitly).

Suppose we try to imitate this proof for $G(m,m,n)$. The first major issue is that the Poincar\'{e} polynomial of $\NC(m,m,n)$ is much larger than that of $R$ (as an $R^W$-module), so
one should not expect a basis of $R$ by taking a basis of $\NC(m,m,n)$ and applying it to one polynomial like $\stair$. Instead one could potentially use a subbasis of $\NC(m,m,n)$ of the appropriate size. One might hope that this subbasis comes from a combinatorially-defined subset of $W_{\aff}$, and similarly for an (almost) dual (sub)basis.

The second key fact used by Demazure is that for each $w$ there is a unique element $w'$ of complementary length such that $w^{-1} w' = w_0$, and for all $x \ne w'$ of the same
length we have $\pa_{w^{-1}} \pa_x = 0$. In $\NC(m,m,n)$ this statement is patently false. For example, when $m=2$ (see \S\ref{ssec:223intro}) $sts$ can be extended to $stsuts$ or
$stsust$, both of which serve as the longest element in $\NC(m,m,n)$. As a consequence, $\pa_{sts}(\stair)$ will pair nontrivially against both $\pa_{uts}(\stair)$ and
$\pa_{ust}(\stair)$. In fact, computations for small values of $m$ indicate that finding two subsets of $W$ which pair perfectly against each other (for each element in one subset,
there is a unique element in the other subset for which the product is a longest element) is impossible. Thus finding almost dual bases combinatorially presents extra challenges,
which we leave open to future explorers.

\begin{rem} \label{rmk:stairslnvsgln} While on the topic of Demazure's proof, we discuss the difference between the $\sl_n$ and $\gl_n$ realizations. Since any Frobenius trace is surjective, $J$ can not be a Frobenius trace unless there is a polynomial $\stair$ satisfying $J(\stair) = 1$. Demazure proves his result over arbitrary domains using that assumption. For the $\gl_n$ realization (the usual permutation action of $S_n$ on variables $x_i$ for $1 \le i \le n$), we can set $\stair = \prod_{i=1}^{n} x_i^{n-i}$, and this works over $\ZZ$. For the $\sl_n$-realization, where one uses polynomials in the roots $x_i - x_j$, it is not possible to find a polynomial $\stair$ over $\ZZ$, so $J$ fails to be a Frobenius trace in some characteristics. In our affine setting, this is another reason to study the realization $V_m$ rather than its root subrealization. \end{rem}

%% file: Exotic.tex
\section{The exotic nilCoxeter algebra} \label{sec:nilcox}

In this chapter we define the exotic nilCoxeter algebra, and prove the roundabout relations. We also provide background on previous approaches to Demazure operators for $G(m,d,n)$.

\subsection{Demazure operators and braid relations} \label{ssec:nilcoxbasics}

For $V_z$ or $V_m$, as for any realization, we can define Demazure operators for simple reflections.

\begin{defn} Let $R_z$ and $R = R_m$ be as in \S\ref{sec:polys}. For each $i \in \Om$ we have a \emph{Demazure operator} $\pa_i$ which sends $R_z \to R_z^i := R_z^{s_i}$. It is
defined be \begin{equation} \label{paidef} \pa_i(f) = \frac{f - s_i f}{\al_i} = \frac{f - s_i f}{x_i - z x_{i+1}}. \end{equation} We define an operator $\pa_i$ on $R_m$ by the
same formula, replacing $z$ with $\ze$. \end{defn}

\begin{ex} We have $\pa_i(x_i) = 1$ and $\pa_i(x_{i+1}) = - z$ and $\pa_i(x_j) = 0$ for $j \ne i, i+1$. We have $\pa_i(x_i x_{i+1}) = 0$. We have
	\begin{equation} \pa_i(x_i^3) = x_i^2 + z x_i x_{i+1} + z^2 x_{i+1}^2, \end{equation}
	\begin{equation} \pa_i(x_{i+1}^3) = -z^{-1} (x_{i+1}^2 + z^{-1} x_i x_{i+1} + z^{-2} x_i^2) = -z^{-3} \pa_i(x_i^3). \end{equation}
More generally, $\pa_i(x_i^k)$ is a geometric series in $x_i$ and $z x_{i+1}$, while $\pa_i(x_{i+1}^k) = - z^{-k} \pa_i(x_i^k)$.
\end{ex}

\begin{defn} Let $\NC(z,n)$ denote the subalgebra of $\End_{R_z^{W_{\aff}}}(R_z)$ generated by $\pa_i$, the \emph{deformed affine nilCoxeter algebra}.  Let $\NH(z,n)$, the \emph{deformed affine nilHecke algebra}, denote the subalgebra generated by $\NC(z,n)$ and multiplication by $x_i$.  Let $\NC(m,m,n)$ denote the subalgebra of $\End_{R_m^{W_m}}(R_m)$ generated by $\pa_i$. Let $\NH(m,m,n)$ denote the subalgebra generated by $\NC(m,m,n)$ and multiplication by $x_i$. \end{defn}

Now we state some standard properties of Demazure operators, which generalize to arbitrary realizations. These relations hold in both $\NC(z,n)$ and $\NC(m,m,n)$.

\begin{lem} Let $i \in \Om$. We have $\pa_i^2 = 0$, the \emph{nil-quadratic relation}. The Demazure operators satisfy the \emph{twisted Leibniz rule}
\begin{equation} \label{eq:Leibniz} \pa_i(fg) = \pa_i(f)g + s_i(f) \pa_i(g). \end{equation}
 \end{lem}

\begin{proof} This is true for any realization. Since the image of $\pa_i$ is contained in $R^i$, and $\pa_i$ kills $R^i$, we have $\pa_i^2 = 0$. The Leibniz rule follows by computation directly from \eqref{paidef}. \end{proof}

\begin{lem} Let $i \ne j \in \Om$. If $j \ne i \pm 1$ then $\pa_i \pa_j = \pa_j \pa_i$. \end{lem}

\begin{proof} Since $s_i(\al_j) = \al_j$ and vice versa, and $s_i s_j = s_j s_i$, this follows immediately by computation from the formula \eqref{paidef}. \end{proof}

The next braid relation may be slightly unfamiliar, because our Cartan matrix is not ``balanced,'' see \cite[Appendix A]{ECathedral}.

\begin{lem} Suppose $n \ge 3$, and let $i \in \Om$. In $\NC(z,n)$ we have
\begin{equation} \label{zeR3} z \pa_i \pa_{i+1} \pa_i = \pa_{i+1} \pa_i \pa_{i+1}. \end{equation}
In $\NC(m,m,n)$ one should replace $z$ with $\ze$.
\end{lem}

\begin{proof} This was proven for general (unbalanced) realizations in \cite[Claim A.7]{ECathedral}. One can also use \eqref{paidef} to explicitly write out $\pa_i \pa_{i+1} \pa_i(f)$ and $\pa_{i+1} \pa_i \pa_{i+1}(f)$. Both can be described as a sum $(-1)^{\ell(w)} w(f)$, for $w$ in the parabolic subgroup generated by $s_i$ and $s_{i+1}$, divided by a product of three roots. One product is $\al_i \al_{i+1} s_i(\al_{i+1})$, and the other product is $\al_i \al_{i+1} s_{i+1}(\al_i)$. The multiplicative factor $z$ precisely matches the difference between $s_i(\al_{i+1})$ and $s_{i+1}(\al_i)$. \end{proof}

\begin{ex} We have
\begin{equation} \pa_1 \pa_2 \pa_1(x_1^2 x_2) = \pa_1 \pa_2(x_1 x_2) = \pa_1(x_1) = 1, \quad \pa_2 \pa_1 \pa_2(x_1^2 x_2) = \pa_2 \pa_1(x_1^2) = \pa_2(x_1 + z x_2) = z. \end{equation}
This matches the fact that $z \pa_1 \pa_2 \pa_1 = \pa_2 \pa_1 \pa_2$. \end{ex}

\begin{lem} Recall the symmetries $\si$ and $\tau$ from Definition \ref{defn:symmetries}. One has
	\begin{equation} \label{siondem} \si(\pa_i(f)) = \pa_{i+1}(\si(f)), \end{equation}
	\begin{equation} \label{tauondem} \tau(\pa_i(f)) = (-z) \pa_{-i}(\tau(f)). \end{equation}
\end{lem}

\begin{proof} The statement about $\si$ is straightforward, so we perform the computation with $\tau$. We have
	\begin{equation} \tau(\pa_i(f)) = \frac{\tau(f) - \tau(s_{i} f) }{x_{1-i} - z^{-1} x_{-i}} = (-z) \frac{\tau(f) - s_{-i} \tau(f)}{x_{-i} - z x_{-i+1}} = (-z) \pa_{-i}(\tau(f)). \qedhere \end{equation} \end{proof}

We believe the relations above give a presentation of $\NC(z,n)$.

\begin{conj} \label{conj:NCzpresentation} The algebra $N$ abstractly generated by symbols $\pa_i$ for $i \in \Om$, modulo the quadratic and braid relations above, is isomorphic to
$\NC(z,n)$, and has the same size as the affine Weyl group $W_{\aff}$ (i.e. graded dimension equals Poincar\'{e} polynomial). \end{conj}

\begin{lem} Pick one reduced expression $\un{w}$ for each element $w \in W_{\aff}$ (e.g. one could use the parametrization $\un{w}(a,b,i)$ from Definition \ref{defn:abiintro}), and use this
expression to define the operator $\pa_{\un{w}} \in N$. Then $\{\pa_{\un{w}}\}_{w \in W_{\aff}}$ is a spanning set for $N$. \end{lem}

\begin{proof} Using the quadratic and braid relations in $W_{\aff}$, one can transform any expression $\un{x}$ into one of the chosen reduced
expressions. Applying the analogous relations in $N$ to the corresponding composition $\pa_{\un{x}}$, one deduces that $\pa_{\un{x}}$ is zero if $\un{x}$ is not a reduced expression, and a scalar multiple of $\pa_{\un{w}}$ if $\un{x}$ is a reduced expression for $w$. \end{proof}

Conjecture \ref{conj:NCzpresentation} is beyond the scope of this paper. The scalar appearing in \eqref{zeR3} makes this algebra not fall into the standard framework of
generalized Hecke algebras for which the result is well-known, see \cite[\S 7]{HumpCox}. The analogous presentation holds for the usual nilCoxeter algebra of $W_{\aff}$, before $z$-deforming. As noted previously, one does not obtain the undeformed setting by specializing $z = 1$, so one can not obviously deduce this conjecture from the usual case by specialization.


\begin{rem} It is of independent interest to generalize the results of \cite{EDiamond} to apply to affine Hecke-type algebras. Then Conjecture \ref{conj:NCzpresentation} would follow as
a consequence. \end{rem}


\subsection{Cyclic words} \label{ssec:cyclic}

\begin{notation} To a word $\un{w} = (i_1, \ldots, i_d)$ in $\Om$, i.e. an expression in the Coxeter system $(W_{\aff}, S)$. To this word we may associate an element
\begin{equation} s_{\un{w}} = s_{i_1} \cdots s_{i_d} \end{equation} in $W_{\aff}$ or $W_m$, and we may associate a \emph{Demazure operator} \begin{equation}
\pa_{\un{w}} := \pa_{i_1} \circ \cdots \circ \pa_{i_d} \end{equation} in $\NC(z,n)$ or $\NC(m,m,n)$.
A non-reduced expression will yield the zero operator. Because of \eqref{zeR3}, two different
reduced expressions for the same element of $W_{\aff}$ will give the same operator only up to an invertible scalar. For this reason we prefer to index Demazure operators by expressions rather than elements of
$W_{\aff}$. \end{notation}

\begin{ex} We write $\pa_{1231}$ or $\pa_{(1,2,3,1)}$ for the composition $\pa_1 \circ \pa_2 \circ \pa_3 \circ \pa_1$. \end{ex}

\begin{defn} A \emph{cyclic word} is a word in $\Om$ of the form $(i, i+1, i+2, \ldots, i+(d-1))$ (clockwise) or of the form $(i,i-1, i-2, \ldots, i-(d-1))$ (widdershins) for some $i \in \Om$ and $d \ge 1$. We write $(i, i+1, \ldots, i+(d-1) = j)$ as $\cw_{i_L, d}$ when we wish to emphasize that it starts with $i$, and as $\cw_{j_R, d}$ when we wish to emphasize that it ends in $j$. Similarly, we write $(i,i-1, i-2, \ldots, i-(d-1) = j)$ as $\ws_{i_L, d}$ or $\ws_{j_R, d}$. \end{defn}

\begin{ex} When $n=3$, $\cw_{2_L,5} = \cw_{3_R,5} = (2,3,1,2,3)$ is a clockwise cyclic word of length $5$. \end{ex}

Below we frequently use the phrase: let $i \in \Om$, and suppose $\cw_{i_L,d} = \cw_{j_R,d}$. This is a more useful way of saying  $j = i+(d-1)$. All the results we prove below for clockwise cyclic words have analogs for widdershins cyclic words, obtained by applying $\tau$.
	
\begin{prop} \label{prop:omdefined} Inside $W_m$ the element $s_{\cw_{i_R, m(n-1)}}$ is independent of $i$, and we denote it by $\om$. Moreover, in the action of $W_m$ on $V_m$ one has
\begin{equation} \om(x_j) = \ze^m x_{j+m} \end{equation}
for all $j \in \Om$. \end{prop}

\begin{ex} When $n=3$ and $m=3$ we are claiming that $s_1 s_2 s_3 s_1 s_2 s_3 = s_2 s_3 s_1 s_2 s_3 s_1 = s_3 s_1 s_2 s_3 s_1 s_2$ in $W_{\aff}/3 \La_{\rt}$. \end{ex}

\begin{proof} Using the fact that $V_m$ is a faithful representation of $W_m$, we can deduce this equality by evaluating each element of $W_m$ on the standard basis of $V_m$.

Consider $s_{\cw_{1_L, n-1}} = s_{\cw_{(n-1)_R, n-1}} = s_1 s_2 \ldots s_{n-1}$. We claim that 
\begin{equation} s_{\cw_{1, n-1}}(x_n) = \ze^{-(n-1)} x_1, \qquad s_{\cw_{1, n-1}}(x_j) = \ze x_{j+1} \text{ for } j \ne n. \end{equation}
This is a straightforward computation. Similarly, 
\begin{equation} \label{onestep} s_{\cw_{i_R, n-1}}(x_{i+1}) = \ze^{-(n-1)} x_{i+2}, \qquad s_{\cw_{i_R, n-1}}(x_j) = \ze x_{j+1} \text{ for } j \ne i+1. \end{equation}

Now we decompose
\begin{equation} s_{\cw_{1_R, m(n-1)}} = s_{\cw_{m_R, n-1}} \cdots s_{\cw_{3_R, n-1}} s_{\cw_{2_R, n-1}} s_{\cw_{1_R, n-1}}. \end{equation}
Applying \eqref{onestep} $m$ times, one deduces that 
\begin{equation} s_{\cw_{1_R, m(n-1)}}(x_2) = \ze^{-m(n-1)} x_{2+m}, \qquad s_{\cw_{1_R, m(n-1)}}(x_j) = \ze^m x_{j+m} \text{ for } j \ne 2. \end{equation}
However $\ze^{-m(n-1)} = \ze^m$, so we get the simpler formula
\begin{equation} s_{\cw_{1_R, m(n-1)}}(x_j) = \ze^m x_{j+m} \text{ for all } j \in \Om. \end{equation}
Applying the symmetry $\si$, we deduce that $s_{\cw_{i_R, m(n-1)}}$ obeys the same formula for any $i$, proving the proposition. \end{proof}

Now we digress to discuss how composing cyclic words will or will not produce reduced expressions.

\begin{lem} \label{lem:repeatnotreduced} Let $1 \le d \le n-1$. Then the concatenation $\cw_{i_L, d} \cw_{i_L, d}$ is not reduced. \end{lem}

\begin{proof} This computation can be performed in some finite symmetric group containing all the indices which appear. Here is the prototypical case: $s_1 s_2 \cdots s_d s_1 s_2 \cdots s_d$ is not reduced in $S_{d+1}$, because strands $d+1$ and $d$ cross twice. \end{proof}
	
\begin{lem} \label{lem:shiftednotreduced} The concatenation $\cw_{j_R, n-1} \cw_{i_L, n-1}$ is not reduced unless $i = j+1$, in which case one obtains a clockwise cyclic word of length $2(n-1)$. \end{lem}

\begin{ex} Suppose $n=5$ and $j=4$. Then $(1234)(5123) = \cw_{1,8}$. However, $(1234)(1234)$ is not reduced by Lemma \ref{lem:repeatnotreduced}. Similarly, $(1234)(2345)$ contains $(234)(234)$ as a subword, which is also not reduced by Lemma \ref{lem:repeatnotreduced}. Similarly, $(1234)(3451)$ contains $(34)(34)$ as a subword, and $(1234)(4512)$ contains $(4)(4)$. \end{ex}

\begin{proof} If $i = j+1 + k$ for $k>0$, and $d = n-k$, then $\cw_{i_L, d} = \cw_{j_R, d}$. Thus the composition $\cw_{j_R, n-1} \cw_{i_L, n-1}$ contains the subword $\cw_{i_L, d} \cw_{i_L,d}$ for $d = n-k$. This is not reduced by Lemma \ref{lem:repeatnotreduced}. \end{proof}

Now we use Lemma \ref{lem:shiftednotreduced} to discuss what happens when we remove an element from a cyclic word. Let $\un{w}$ be a word of length $d$, and let $1 \le \ell \le d$. We write $\hat{\ell}$ or $(\un{w}, \hat{\ell})$ for the word obtained by omitting the $\ell$-th term from $\un{w}$. 

\begin{lem} \label{lem:insidenotreduced} The expression $(\cw_{i_R, d}, \hat{\ell})$ is not reduced for any $n \le \ell \le d-(n-1)$.
	
In particular, for any $k \ge 3$, $(\cw_{i_R, k(n-1)}, \hat{\ell})$ is not reduced for any $n \le \ell \le (k-1)(n-1)$. \end{lem}
	
In other words, if we want a reduced expression, we can only remove an element from the first or last $n-1$ elements in $\cw_{i_R, d}$.
	
\begin{ex} Let $n=5$ and $k=3$. Consider $\cw_{2_R, 3(n-1)} = (1234)(5123)(4512)$. Removing $5$ in the middle third, we have a subword $(1234)(1234)$, which is not reduced.
Removing $1$ in the middle third, we have a subword $(2345)(2345)$, which is not reduced. Similarly, removing any element in the middle third will produce a non-reduced
subexpression. \end{ex}

\begin{proof} Suppose the $\ell$-th index of $\un{w}$ is $s_j$. The expression $\hat{\ell}$ has subword
\begin{equation} \ldots \cw_{(j-1)_R, n-1} \hat{s}_j \cw_{(j+1)_L, n-1}. \end{equation}
This is not reduced, by Lemma \ref{lem:shiftednotreduced}. \end{proof}

\subsection{Rotational Demazure operators}

In this section we explore some interesting relations in the nilHecke algebra $\NH(z,n)$. In the next section, we explore their consequences for $\NC(m,m,n)$.

\begin{defn} Let $i \in \Om$. Suppose that $\cw_{i_L,n-1} = \cw_{j_R, n-1}$. Define the degree $-(n-1)$ operator $\Theta_{i_L} = \Theta_{j_R} \in \NC(z,n)$ by the formula
\begin{equation} \label{eq:Theta} \Theta_{i_L} := \pa_{\cw_{i_L, n-1}} + z^{-1} \pa_{\cw_{(i-1)_L, n-1}} + z^{-2} \pa_{\cw_{(i-2)_L, n-1}} + \ldots + z^{-(n-1)} \pa_{\cw_{(i+1)_L, n-1}}. \end{equation} \end{defn}

\begin{ex} Let $n=3$. Then
\begin{subequations}
\begin{equation} \Theta_{1_L} = \Theta_{2_R} = \pa_1 \pa_2 + z^{-1} \pa_3 \pa_1 + z^{-2} \pa_2 \pa_3, \end{equation}
\begin{equation} \Theta_{2_L} = \Theta_{3_R} = \pa_2 \pa_3 + z^{-1} \pa_1 \pa_2 + z^{-2} \pa_3 \pa_1, \end{equation}
\begin{equation} \Theta_{3_L} = \Theta_{1_R} = \pa_3 \pa_1 + z^{-1} \pa_2 \pa_3 + z^{-2} \pa_1 \pa_2. \end{equation}
\end{subequations}
\end{ex}

The following theorem is crucial.

\begin{thm} \label{thm:thetarotates} Suppose $\cw_{i_L,n-1} = \cw_{j_R, n-1}$. For any $f \in R$, we have
\begin{equation} \Theta_{j_R}(x_{j+1} f) = x_i \Theta_{(j-1)_R}(f). \end{equation}
\end{thm}

Note that $i = j+2$. Let us prove the result first when $n=3$. The general proof is identical with more annoying notation.

\begin{ex} \label{ex:theta1n3x3f} Applying the Leibniz rule twice in each equation, we have
\begin{subequations} \label{subeq:theta1x3f}
\begin{align} \nonumber \pa_1 \pa_2(x_3 f) & = \pa_1(\pa_2(x_3) \cdot f + s_2(x_3) \pa_2(f)) = \pa_2(x_3) \cdot \pa_1(f) + \pa_1(s_2(x_3)) \pa_2(f) + s_1 s_2(x_3) \pa_1 \pa_2(f)\\ &= -z^{-1} \pa_1(f) - z^{-2} \pa_2(f) + z^{-2} x_1 \pa_1 \pa_2(f). \end{align}
\begin{equation} \pa_3 \pa_1(x_3 f) = \pa_3(x_3 \pa_1(f)) = \pa_3(x_3) \cdot \pa_1(f) + s_3(x_3) \pa_3 \pa_1(f) = \pa_1(f) + z x_1 \pa_3 \pa_1(f). \end{equation}
\begin{equation} \pa_2 \pa_3(x_3 f) = \pa_2(\pa_3(x_3) \cdot f + s_3(x_3) \pa_3(f)) = \pa_2(f + z x_1 \pa_3(f)) = \pa_2(f) + z x_1 \pa_2 \pa_3(f). \end{equation}
\end{subequations}
In the first equation we tacitly use the fact that $\pa_2(x_3)$ is a constant, and hence is $W$-invariant. The other two equations have fewer terms, because $\pa_1(x_3) = 0$ and $\pa_2(x_1) = 0$.

Adding it all up, most of the terms cancel, and we get
\begin{equation} \Theta_{2_R}(x_3 f) = z^{-2} x_1 \pa_1 \pa_2(f) + x_1 \pa_3 \pa_1(f) + z^{-1} x_1 \pa_2 \pa_3(f) = x_1 \Theta_{1_R}(f). \end{equation}
\end{ex}

\begin{rem} Using \eqref{subeq:theta1x3f} to compute $\Theta_{3_R}(x_3 f)$ and $\Theta_{1_R}(x_3 f)$ instead, there is no cancellation of terms, and we end up with a big mess. \end{rem}

\begin{proof} Consider a clockwise word of length $n-1$; exactly one index in $\Om$ does not appear, which we call $p$. The word in question is $\cw_{(p+1)_L,n-1} = \cw_{(p-1)_R,n-1}$, which we simply shorten to $\cw_{\hat{p}}$. Now consider a subword of length $n-2$; now exactly two indices in $\Om$ do not appear, $p$ and $q$. Let $\un{v}_{p,q}$ denote the word obtained from $\cw_{\hat{p}}$ by removing the index $q$ from the one place it appears, for $q \ne p$. Note that 
\begin{equation} \un{v}_{p,q} = \cw_{(p+1)_L,a} \circ \cw_{(q+1)_L,b} = \cw_{(q-1)_R,a} \circ \cw_{(p-1)_R,b} \end{equation}
for some lengths $a$ and $b$ which are determined by $p$ and $q$. Each index in the first clockwise subword commutes with each index in the second clockwise subword. Thus, while $\un{v}_{p,q}$ and $\un{v}_{q,p}$ are not the same word, they agree up to commutation of distant indices. So we have
\begin{equation} \pa_{\un{v}_{p,q}} = \pa_{\un{v}_{q,p}}. \end{equation}
We will let $\cw_{\hat{p}\hat{q}}$ denote either word, because we will only use this notation to index a Demazure operator.

Suppose we wish to compute $\pa_{\cw_{\hat{p}}}(x_{j+1} f)$ by iteratively applying the twisted Leibniz rule. We claim that
\begin{equation} \pa_{\cw_{\hat{p}}}(x_{j+1} f) = s_{\cw_{\hat{p}}}(x_{j+1}) \pa_{\cw_{\hat{p}}}(f) + \sum_{q \ne p} \kappa_{p,q} \pa_{\cw_{\hat{p}\hat{q}}}(f). \end{equation}
In the first term,
each Demazure operator is applied to $f$, and passes over $x_{j+1}$, twisting it. Meanwhile, the sum consists of what we call \emph{lower terms}. Suppose that $q$ is the $k$-th index from the right in $\cw_{\hat{p}}$, for $1 \le k \le n-1$. In the term of the sum indexed by $q$, the first $k-1$ Demazure operators are applied to $f$, twisting $x_{j+1}$. Then $\pa_q$ is applied to the twist of $x_{j+1}$, yielding a scalar $\kappa_{p,q}$ (which is therefore $W$-invariant). The remaining Demazure operators act on $f$. The scalar $\kappa_{p,q}$ is therefore equal to
\begin{equation} \kappa_{p,q} = \pa_q(w_{p,q}(x_{j+1})), \qquad w_{p,q} = s_{\cw_{(q+1)_L,b}} = s_{\cw_{(p-1)_R,b}}. \end{equation}

Our goal is to argue that, in the sum $\Theta_{j_R}(x_{j+1} f)$, each lower term $\pa_{\cw_{\hat{p}\hat{q}}}(f)$ will appear exactly twice (up to scalar multiples), once coming from
$\pa_{\cw_{\hat{p}}}$ and once from $\pa_{\cw_{\hat{q}}}$. The coefficients will cancel out, and all lower terms will vanish. Let us verify this. We consider ``proof by example'' to be more illustrative than keeping track of the indices carefully en route, though we will state the general result when we're through.

We suppose $n=11$ and compute $\kappa_{9,4}$. We have $w_{9,4} = s_5 s_6 s_7 s_8$. There are four cases to treat, depending on $j$.
\begin{itemize} \item If $5 \le j+1 \le 8$ then $s_5 s_6 s_7 s_8(x_{j+1}) = x_{j+2}$, and $\pa_4(x_{j+2}) = 0$.
	\item If $10 \le j+1 \le 3$ then $s_5 s_6 s_7 s_8(x_{j+1}) = x_{j+1}$ and $\pa_4(x_{j+1}) = 0$.
	\item If $j+1 = 4$ then $s_5 s_6 s_7 s_8(x_4) = x_4$ and $\pa_4(x_4) = 1$.
	\item If $j+1 = 9$ then $s_5 s_6 s_7 s_8(x_9) = z^{-4} x_5$ and $\pa_4(x_5) = -z^{-1}$, so $\pa_4(s_5 s_6 s_7 s_8(x_9)) = - z^{-5}$.
\end{itemize}
More generally,
\begin{equation} \kappa_{p,q} = \begin{cases} 1 & \text{ if } j+1 = q, \\ -z^{q - p} & \text{ if } j+1 = p, \\ 0 & \text{ else.} \end{cases} \end{equation}
In this context, we choose representatives of $\Om$ such that $1 \le p-q \le n-1$.

Suppose $\cw_{i_L,n-1} = \cw_{j_R, n-1}$, so that $i = j+2$. In $\Theta_{j_R}$ the operator $\pa_{\cw_{9_L,n-1}}$ appears with coefficient $z^{9-i}$ and the operator $\pa_{\cw_{4_L,n-1}}$ appears with coefficient $z^{4-i}$, again with the interpretation that $0 \le i-9, i-4 \le n-1$. In our computations below we may add or subtract $n$ from an exponent; this is because we may need to choose the representative $i+n$ rather than $i$.

The overall contribution to the coefficient of $\pa_{\cw_{\hat{4},\hat{9}}}(f)$ is zero if $j+1 \ne 4,9$. If $j+1 = 4$ so that $i=5$, we get
\begin{equation} z^{9-i} \cdot \kappa_{9,4} + z^{4-i} \kappa_{4,9} = z^{9-5-n} \cdot 1 + z^{4-5} \cdot (-z^{9-4-n}) = z^{9-5-n} - z^{9-5-n} = 0. \end{equation}
Similarly, one gets zero if $j+1 = 9$.

We've proven the cancellation of all lower terms, so we need only examine the linear combination of terms $s_{\cw_{\hat{p}}}(x_{j+1}) \pa_{\cw_{\hat{p}}}(f)$. Note that 
\begin{equation} s_{\cw_{\hat{p}}}(x_{j+1}) = \begin{cases} z x_{j+2} & \text{ if } p \ne j+1, \\ z^{1-n} x_{j+2} & \text{ if } p = j+1. \end{cases} \end{equation}
From here, it is easy to match up the coefficient of $x_i \pa_{\cw_{\hat{p}}}(f)$ in $\Theta_{j_R}(x_{j+1} f)$ and in $x_i \Theta_{(j-1)_R}(f)$. \end{proof}

Now let us discuss what happens when we multiply $\Theta_{j_R}$ by $\Theta_{i_L}$.

\begin{ex} \label{ex:blahblahblah} Let $n=3$. The various $\Theta_{i_L}$ are built out of the three terms $\pa_1 \pa_2$, $\pa_3 \pa_1$ and $\pa_2 \pa_3$. Note that
\begin{equation} \pa_1 \pa_2 \circ \pa_2 \pa_3 = 0, \qquad \pa_1 \pa_2 \circ \pa_1 \pa_2 = 0. \end{equation}
So the only term that survives in the composition $\pa_1 \pa_2 \circ \Theta$ is the term of the form $\pa_1 \pa_2 \pa_3 \pa_1$. More generally, the only terms that survive when composing two different $\Theta$ operators are the terms associated to cyclic words. 
\end{ex}

\begin{defn} Let $i \in \Om$ and $k \ge 1$. Suppose that $\cw_{i_L,k(n-1)} = \cw_{j_R, k(n-1)}$. Define the degree $-k(n-1)$ operator $\Theta^{(k)}_{i_L} = \Theta^{(k)}_{j_R} \in \NC(z,n)$ by the formula
\begin{equation} \label{eq:Thetak} \Theta^{(k)}_{i_L} := \pa_{\cw_{i_L, k(n-1)}} + z^{-k} \pa_{\cw_{(i-1)_L, k(n-1)}} + z^{-2k} \pa_{\cw_{(i-2)_L, k(n-1)}} + \ldots + z^{-k(n-1)} \pa_{\cw_{(i+1)_L, k(n-1)}}. \end{equation}
Thus $\Theta_{i_L} = \Theta_{i_L}^{(1)}$.
\end{defn}

\begin{ex} Let $n=3$. Then
\begin{equation} \Theta_{1_L}^{(2)} = \Theta_{1_R}^{(2)} = \pa_{1231} + z^{-2} \pa_{3123} + z^{-4} \pa_{2312}. \end{equation}
\begin{equation} \Theta_{1_L}^{(3)} = \Theta_{3_R}^{(3)} = \pa_{123123} + z^{-3} \pa_{312312} + z^{-6} \pa_{231231}. \end{equation}	
Continuing Example \ref{ex:blahblahblah}, one can compute that $\Theta^{(1)}_{2_R} \circ \Theta^{(1)}_{3_L} = \Theta^{(2)}_{1_L}$.
\end{ex}

\begin{thm} \label{thm:composetheta} Let $i \in \Om$ and $k, \ell \ge 1$. Suppose that $\cw_{i_L,k(n-1)} = \cw_{j_R, k(n-1)}$. We have
\begin{equation} \Theta^{(k)}_{i_L} \Theta^{(\ell)}_{(j+1)_L} = \Theta^{(k)}_{j_R} \Theta^{(\ell)}_{(j+1)_L} = \Theta^{(k+\ell)}_{i_L}. \end{equation}
\end{thm}

\begin{proof} There are $(n-1)$ terms in $\Theta^{(k)}_{j_R}$ and $(n-1)$ terms in $\Theta^{(\ell)}_{(j+1)_L}$, leading to $(n-1)^2$ terms in the composition. Composing the $d$-th term in each sum we obtain
\begin{equation} z^{-dk} \pa_{\cw_{(j-d)_R, k(n-1)}} \circ z^{-d \ell} \pa_{\cw_{(j+1-d)_L, \ell(n-1)}} = z^{-d(k+\ell)} \pa_{\cw_{(i-d)_L, (k+\ell)(n-1)}}, \end{equation}
which is the $d$-th term in $\Theta^{(k+\ell)}_{i_L}$. Meanwhile, any other composition is a scalar multiple of
\begin{equation} \pa_{\cw_{(j-d)_R, k(n-1)}} \circ  \pa_{\cw_{(j+1-d')_L, \ell(n-1)}} \end{equation}
for some $d' \ne d$. By Lemma \ref{lem:shiftednotreduced}, this composition is associated to a non-reduced word, and vanishes. \end{proof}

\begin{thm} \label{thm:thetakrotates} Suppose $\cw_{i_L,k(n-1)} = \cw_{j_R, k(n-1)}$. For any $f \in R$, we have
\begin{equation} \Theta^{(k)}_{j_R}(x_{j+1} f) = x_i \Theta^{(k)}_{(j-1)_R}(f). \end{equation}
\end{thm}

\begin{proof} By Theorem \ref{thm:composetheta}, $\Theta^{(k)}$ is an iterated composition of $\Theta^{(1)}$ for various indices. Applying Theorem \ref{thm:thetarotates} $k$ times, we deduce the desired result. \end{proof}
	
\begin{rem} All these results hold in $\NH(m,m,n)$ by specialization. \end{rem}

\subsection{The roundabout relation}

The discussion of $\Theta^{(k)}$ in the previous section applied for generic $z$, but now we explore what happens when $z$ is specialized to a root of unity.

\begin{lem} \label{lem:thetamindep} Inside $\NC(m,m,n)$, we have
\begin{equation} \ze^{-m} \Theta^{(m)}_{i_L} = \Theta^{(m)}_{(i+1)_L}. \end{equation} \end{lem}

\begin{proof} We have
\begin{equation} \ze^{-m} \Theta^{(m)}_{i_L} = \ze^{-m} \pa_{\cw_{i_L, m(n-1)}} + \ze^{-2m} \pa_{\cw_{(i-1)_L, m(n-1)}} + \ze^{-3m} \pa_{\cw_{(i-2)_L, m(n-1)}} + \ldots + \ze^{-nm} \pa_{\cw_{(i+1)_L, m(n-1)}}. \end{equation}
Since $\ze^{-nm} = 1$, this exactly agrees with $\Theta^{(m)}_{(i+1)_L}$. \end{proof}

\begin{thm} \label{thm:roundabout} Inside $\NC(m,m,n)$, the operator $\Theta^{(m)}_{i_L}$ vanishes for any $i \in \Om$. \end{thm}
	
By the previous lemma, the relation $\Theta^{(m)}_{i_L} = 0$ for any given $i$ implies the same relation for any other $i$. We refer to the equality
\begin{equation} \Theta^{(m)}_{1_L} = 0 \end{equation}
as the \emph{(clockwise) roundabout relation}, with the understanding that it implies $\Theta^{(m)}_{i_L} = 0$ for any $i \in \Om$.

\begin{proof} Suppose $\cw_{i_L, m(n-1)} = \cw_{j_R, m(n-1)}$, or in other words $j = i + m(n-1) - 1$. By Theorem \ref{thm:thetakrotates}, $\Theta^{(m)}_{j_R}(x_{j+1} f) = x_i \Theta^{(m)}_{(j-1)_R}(f)$. Multiplying this by $\ze^{-m}$, and using Lemma \ref{lem:thetamindep}, we deduce that $\Theta^{(m)}_{(j+1)_R}(x_{j+1} f) = x_i \Theta^{(m)}_{j_R}(f)$ as well. Continuing to rotate in this fashion, and reindexing, we deduce that
\begin{equation} \label{eq:whatthetamdoes} \Theta^{(m)}_{j_R}(x_k f) = x_{k-m(n-1)} \Theta^{(m)}_{(j-1)_R}(f) \end{equation}
for all $k \in \Om$. Note that $k - m(n-1) = k + m$ modulo $n$. Another way to read this relation using Proposition \ref{prop:omdefined} is that
\begin{equation} \label{eq:whatthetamdoes2} \Theta^{(m)}_{j_R}(x_k f) = \ze^{-m} \om(x_{k}) \Theta^{(m)}_{(j-1)_R}(f) = \om(x_k) \Theta^{(m)}_{j_R}(f). \end{equation}

So if $f$ is any monomial, one can repeatedly use \eqref{eq:whatthetamdoes2} to deduce that 
\begin{equation} \label{eq:whatthetamdoes3} \Theta^{(m)}_{j_R}(f) = \om(f) \Theta^{(m)}_{j_R}(1). \end{equation}
Since $\Theta^{(m)}$ has degree $-m(n-1)$, it kills $1$. Thus $\Theta^{(m)}_{j_R}$ kills any monomial, and by linearity, any element of $R_m$. \end{proof}

By definition (after specialization $z \mapsto \ze$) we have
\begin{equation} \Theta^{(m)}_{1_L} = \pa_{\cw_{1_L, m(n-1)}} + \ze^{-m} \pa_{\cw_{n_L, m(n-1)}} + \ldots + \ze^{-m(n-1)} \pa_{\cw_{2_L, m(n-1)}}. \end{equation}
Reordering the terms in this relation and using that $\ze^{mn}=1$ we get
\begin{equation}\label{thetamoreconvenient} \Theta^{(m)}_{1_L} = \pa_{\cw_{1_L, m(n-1)}} + \ze^{m} \pa_{\cw_{2_L, m(n-1)}} + \ldots + \ze^{m(n-1)} \pa_{\cw_{n_L, m(n-1)}}. \end{equation}
This tends to be the version of the relation we use in practice. It also more closely matches the widdershins case below.

By applying the operator $\tau$, we obtain results about operators akin to $\Theta$ but using widdershins cyclic words instead.

\begin{defn} Let $i \in \Om$ and $k \ge 1$. Suppose that $\ws_{i_L,k(n-1)} = \ws_{j_R, k(n-1)}$. Define the degree $-k(n-1)$ operator $\bTheta^{(k)}_{i_L} = \bTheta^{(k)}_{j_R} \in \NC(z,n)$ by the formula
\begin{equation} \label{eq:bThetak} \bTheta^{(k)}_{i_L} := \pa_{\ws_{i_L, k(n-1)}} + \ze^{k} \pa_{\ws_{(i+1)_L, k(n-1)}} + \ze^{2k} \pa_{\ws_{(i+2)_L, k(n-1)}} + \ldots + \ze^{k(n-1)} \pa_{\ws_{(i-1)_L, k(n-1)}}. \end{equation}
\end{defn}

\begin{ex} Let $n=3$. Then
\begin{equation} \bTheta_{1_L}^{(2)} = \bTheta_{1_R}^{(2)} = \pa_{1321} + \ze^{2} \pa_{2132} + \ze^{4} \pa_{3213}. \end{equation}
\begin{equation} \bTheta_{1_L}^{(3)} = \bTheta_{2_R}^{(3)} = \pa_{132132} + \ze^{3} \pa_{213213} + \ze^{6} \pa_{321321}. \end{equation}	
\end{ex}

\begin{lem} We have $\tau(\Theta^{(k)}_{i_L}) = (-\zeta)^{k(n-1)} \bTheta^{(k)}_{-i_L}$. \end{lem}
	
\begin{proof} This follows quickly from \eqref{tauondem}. \end{proof}

\begin{cor} \label{cor:otherroundabout} Inside $\NC(m,m,n)$, we have $\bTheta^{(m)}_{i_L} = 0$ for any $i \in \Om$. \end{cor}

\begin{proof} This follows from Theorem \ref{thm:roundabout} by applying $\tau$. \end{proof}

Again, we refer to the equality
\begin{equation} \bTheta^{(m)}_{1_L} = 0 \end{equation}
as the \emph{(widdershins) roundabout relation}, with the understanding that it implies $\bTheta^{(m)}_{i_L} = 0$ for any $i \in \Om$.

While the clockwise and widdershins roundabout relations imply each other in the presence of $\tau$, they do not imply each other without it. When defining $\NC(m,m,n)$ by generators and relations, we must include both the clockwise and widdershins roundabout relations, giving two relations in degree $m(n-1)$.

Another consequence of the roundabout relations is the vanishing of certain cyclic Demazure operators.

\begin{cor} \label{cor:longcyclicvanishes} If $k \ge (m+1)(n-1)$, then 
\begin{equation} \label{eq:longcyclicvanishes} \pa_{\cw_{i_L, k}} = 0 = \pa_{\ws_{i_L, k}}\end{equation}
for any $i \in \Om$. \end{cor}
	
\begin{ex} \label{ex:longcyclicvanishesm2} When $m=2$ and $n=3$, the clockwise roundabout relation states that
\begin{equation} 0 = \pa_{1231} + \ze^{-2} \pa_{3123} + \ze^{-4} \pa_{2312}. \end{equation}
Left multiplying by $\pa_3$ we get
\begin{equation} 0 = \pa_{31231} + \ze^{-2} \pa_{33123} + \ze^{-4} \pa_{32312} = \pa_{31231} + \ze^{-4} \pa_{32312}. \end{equation}
Left multiplying again by $\pa_2$ we get
\begin{equation} 0 = \pa_{231231} + \ze^{-4} \pa_{232312} = \pa_{231231}. \end{equation}
Thus $\pa_{231231} = 0$, as is any other cyclic word (clockwise or widdershins) of length $6 = (m+1)(n-1)$. \end{ex}

\begin{proof} Let $\cw_{i_L, n-1} = \cw_{j_R, n-1}$. By the same argument we used to prove Theorem \ref{thm:composetheta}, one can prove that
	\begin{equation} \pa_{\cw_{j_R, n-1}} \circ \Theta^{(m)}_{(j+1)_L} = \pa_{\cw_{i_L, (m+1)(n-1)}}. \end{equation}
	But $\Theta^{(m)}_{(j+1)_L} = 0$, implying the result. \end{proof}

Because we find it to be edifying, let us sketch another proof of Theorem \ref{thm:roundabout}.

\begin{proof}[Sketch of alternate proof of Theorem \ref{thm:roundabout}] Let us attempt to directly prove the equality
\begin{equation} \Theta^{(m)}_{j_R}(x_k f) = \om(x_k) \Theta^{(m)}_{j_R}(f) \end{equation}
for all $j$ and $k$, c.f. \eqref{eq:whatthetamdoes2}. As in the proof of Theorem \ref{thm:roundabout}, this suffices to prove the roundabout relations. We will follow the style of computation used in the proof of Theorem \ref{thm:thetarotates}.

For a word $\un{w} = (i_1, i_2, \ldots, i_d)$ of length $d$, and for $1 \le \ell \le d$, let $\un{w}_{> \ell}$ denote the subword $(i_{\ell+1}, i_{\ell+2}, \ldots, i_d)$.  For any degree $1$ homogeneous polynomial $g$, one can apply the Leibniz rule iteratively to prove that
\begin{equation}\label{iteratedLeibniz} \pa_{\un{w}}(fg) = s_{\un{w}}(g) \pa_{\un{w}}(f) + \sum_{\ell=1}^d \pa_{i_\ell}(s_{\un{w}_{>\ell}}(g)) \pa_{(\un{w}, \hat{\ell})}(f). \end{equation}
The key point in proving this formula is that $s_{\un{w}_{>\ell}}(g)$ has degree $1$, so $\pa_{i_\ell}(s_{\un{w}_{>\ell}}(g))$ has degree zero. Being a scalar, it is $W$-invariant and pulls freely out of the remaining Demazure operators.

When applied to $\pa_{\cw_{j_L, m(n-1)}}$, all the terms with $n \le \ell \le (m-1)(n-1)$ will die by Lemma \ref{lem:insidenotreduced}.  We are left with $n-1$ terms at the start, and $n-1$ terms at the end. By Proposition \ref{prop:omdefined}, $s_{\cw_{j_L, m(n-1)}} = \om$ is independent of the choice of $j$. So we can write
\begin{equation} \pa_{\cw_{j_L, m(n-1)}}(fg) = \om(g) \pa_{\cw_{j_L, m(n-1)}}(f) + \sum_{\ell} \kappa_{j,\ell} \pa_{(\cw_{j_L, m(n-1)}, \hat{\ell})}(f). \end{equation}
The scalar $ \kappa_{j,\ell}$ is equal to the factor $\pa_{i_\ell}(s_{\un{w}_{>\ell}}(g))$ in \eqref{iteratedLeibniz}, applying a single Demazure operator to a twist of $g$ (though with apologies, we do not use the same indexing conventions for $\kappa$ as when studying $\cw_{\hat{p}\hat{q}}$). We think of the sum as \emph{lower terms}.

Now summing $n$ of these formulas together, we deduce that
\begin{equation} \Theta^{(m)}_{i_L}(fg) = \om(g) \Theta^{(m)}_{i_L}(f) + \sum_{j, \ell} \kappa_{j,\ell} \pa_{(\cw_{j_L, m(n-1)}, \hat{\ell})}(f). \end{equation}
We prove the desired statement if we can show that the sum of the lower terms is zero. The magic of the roundabout relation is the cancellation of these lower terms.

There are a total of $2n(n-1)$ nonzero terms, for $n$ choices of $j$ and $2(n-1)$ choices of $\ell$. Once again they cancel in pairs. Each reduced expression
$(\cw_{j_L, m(n-1)}, \hat{\ell})$ expresses an element of $W_{\aff}$ coming from precisely one other choice. The Demazure operators for the these paired expressions are not equal, not even up to commutation,
but they are equal up to a power of $\ze$ coming from the braid relations \eqref{zeR3}. Keeping track of the scalars, one can show these two paired terms cancel when $\ze^{mn} = 1$. Note the three
sources of scalars: powers of $\ze$ in the formula for $\Theta^{(m)}_{i_L}$, powers of $\ze$ which come from \eqref{zeR3}, and the scalars coming from applying a single Demazure
operator to a twist of $g$. This bookkeeping, though quite annoying, is familiar from the proof of Theorem \ref{thm:thetarotates}. \end{proof}

\begin{ex} When $m=2$ and $n=3$, we have
\begin{equation} \Theta := \Theta^{(2)}_{1_L} = \pa_{1231} + \ze^{-2} \pa_{3123} + \ze^{-4} \pa_{2312}. \end{equation}
When computing $\Theta(fg) = \om(g) \Theta(f) = \sum X(j,\ell)$, there are $12$ lower terms, coming from the four colength $1$ subexpressions of $1231$, $3123$, and $2312$.

Both $1231$ and $3123$ have a subexpression $123$, and no other subexpression expresses $s_1 s_2 s_3$. The terms involving $\pa_{123}(f)$ in this expression for $\Theta(fg)$
are \begin{equation} \label{123coeff} \pa_1(g) \pa_{123}(f) + \ze^{-2} \pa_3(s_1 s_2 s_3(g))\pa_{123}(f). \end{equation} In fact, \eqref{123coeff} vanishes for any $g$ of degree $1$. For example, if $g = x_1$, then $1 + \ze^{-2} \pa_3(\ze^{-3} x_1) = 0$, since $\ze^6 = 1$. One can also check vanishing directly for $x_2$ and $x_3$.

Similarly, both $1231$ and $2312$ have a subexpression for $s_1 s_2 s_1$, and no other subexpression expresses this element. The terms involving $\pa_{121}(f)$ in this expression for $\Theta(fg)$
are \begin{equation} \label{121coeff} \pa_3(s_1(g)) \pa_{121}(f) + \ze^{-4} \pa_3(s_1 s_2(g))\pa_{212}(f). \end{equation}
Recalling that $\pa_{212} = \ze \pa_{121}$, we seek to show that
$\pa_3(s_1(g)) + \ze^{-3} \pa_3(s_1 s_2(g)) = 0$ for any $g$ of degree $1$. Again, we encourage the reader to check this directly for $g = x_1, x_2, x_3$.
\end{ex}

\subsection{Detailed summary of literature on Demazure operators} \label{ssec:rampetas}

We try to summarize previous work relating to Demazure operators and complex reflection groups $G(m,d,n)$. In this section $W$ will represent the complex reflection group being discussed, either $G(m,1,n)$ or $G(m,m,n)$.

Previously, several authors had studied roots and reflections for complex reflection groups. In \cite{BreMal1}, Bremke and Malle introduced for $G(m,1,n)$ a partition of the roots
into positive and negative roots, and a special subset of the positive roots. The number of special positive roots sent to negative roots defines a length function on $G(m,1,n)$,
agreeing with the usual length function on the subgroup $S_n$. The Poincar\'{e} polynomial of this length function matches the graded rank of $R$ over $R^W$, so there is a unique
longest element $w_0$. Applications of this length function to the study of reduced words in $G(m,1,n)$ were begun in \cite{BreMal2} and continued in \cite{RamShoIandII}.

In \cite{BreMal2}, Bremke and Malle extended provided more concrete details on $G(m,1,n)$. They defined in \cite[Lemma 1.10 and
preceding]{BreMal2} certain special elements $w(a,i)$ for $1 \le a \le m$ and $1 \le i \le n$. For a sequence $\ab =(a_1, \ldots, a_n)$, the length of the element
\[ x_{\ab} := w(a_1,1) w(a_2,2) \cdots w(a_n,n) \]
is equal to the sum of the lengths of its parts. Letting $N$ denote the set of such $x_{\ab}$, they proved \cite[Corollary 1.16]{BreMal2} that $N$ is a set of distinguished coset
representatives for left cosets $S_n \backslash W$, for which $\ell(w x) = \ell(w) + \ell(x)$ for $w \in S_n$ and $x \in N$. In some sense, this gives an unusual notion of a distinguished reduced expression for any element of $W$, matching their length function. Note that the elements of $N$ are not
typically elements of the translation lattice.

Bremke and Malle in \cite{BreMal2} also extended their results to $G(m,m,n)$. By \cite[proof of Proposition 2.6]{BreMal2}, the same construction of $N$ works for $G(m,m,n)$ under
the assumption that $\sum a_i \equiv 0$ modulo $m$. Shoji adapted these ideas to the general case of $G(m,d,n)$ in \cite{ShojiGrpn}.

In part II\footnote{To clarify, \cite{RamShoIandII} is a compilation of two papers, part I and part II, and all references are from part II unless stated otherwise.} of
\cite{RamShoIandII}, Shoji and Rampetas defined Demazure operators associated to the reflection representation of $G(m,1,n)$. They defined in \cite[(2.3.1)]{RamShoIandII} a
Demazure operator $\De(a,i)$ for each element $w(a,i)$ of (negative) degree matching its length. By taking a product over these distinguished reduced expressions, they obtain an
operator $\De_w$ for each $w \in W$. They proved \cite[Proposition 2.12]{RamShoIandII} that these operators are linearly independent, and act perfectly on the coinvariant algebra
(making these operators isomorphic to the dual space to the coinvariant algebra). They prove \cite[Proposition 2.21]{RamShoIandII} that the coinvariant algebra is a Frobenius algebra, with Frobenius trace $\De_{w_0}$. They study the subspace of $\End_{R^W}(R)$
spanned by $\{\De_w\}_{w \in W}$, but they do not suggest it is a subalgebra (it is not). They do not investigate the algebra it generates. They do investigate the subalgebra
generated both by $\{\De_w\}$ and by multiplication by elements of $R$, which we might call the Rampetas-Shoji nilHecke algebra. They prove that $\{\De_w\}$ is a basis for their
nilHecke algebra over the subring $R$. However, the nilHecke algebra is quite large (it contains the group algebra of $W$). Passing to the nilHecke algebra obfuscates any
special features of the choice of simple reflections.

In \cite{Ramp}, Rampetas extended this work to $G(m,m,n)$, and there are considerable extra subtleties. Earlier in \cite{RamShoIandII}, the Demazure operator $\De_t$ associated to
the generating rotation element $t$ (of order $m$) was easy to construct, and other Demazure operators $\De(a,i)$ are built from this. For Rampetas there is no $\De_t$, and the
construction of $\De(a,i)$ is more complicated.

Let us recall from \S\ref{ssec:otherviews} the typical way to construct the group $G(m,m,n)$ acting on its reflection representation. Inside the representation $V_q$
one can define reflections $s_{i'}$ which swap $v_i$ with $q^2 v_{i+1}$. Most previous work views $G(m,m,n)$ as being generated by the symmetric group $S_n$ and the reflection
$s_{1'}$ (whereas we use the reflection $s_0$). Let $\pa_i$ and $\pa_{i'}$ be the ordinary divided difference operators associated to $s_i$ and $s_{i'}$. Rampetas \cite[(3.1.3), \S
3.5]{Ramp} alternates between $\pa_i$ and $\pa_{i'}$ to obtain degree $-a$ operators $\De_i^{(a)}$ and $\De_{i'}^{(a)}$ (the prime determines which of $\pa_i$ or $\pa_{i'}$ is
applied first), which he uses to build $\De(a,i)$ and then $\De_w$.

Let us note that the degree $-1$ Demazure operators $\De_w$ are precisely the ordinary Demazure operators $\pa_i$ associated to $S_n$, and the operator $\pa_{1'}$. Significantly,
$\pa_{1'}$ is not a linear combination of our degree $-1$ Demazure operators $\pa_i$ and $\pa_0$, which can be seen easily by evaluating these operators on the polynomial $v_1^2$.
In particular, Rampetas' operators $\De_w$ are not contained inside $\NC(m,m,n)$. Again, Rampetas does not suggest that the subspace spanned by $\{\De_w\}$ is a subalgebra of $\End_{R^W}(R)$, nor does he investigate the algebra generated by his operators, beyond investigating the nilHecke algebra. It would be interesting to present the algebra generated by Rampetas' operators by generators and relations.

Rampetas then states a non-degeneracy conjecture \cite[(3.12.2)]{Ramp}. To our knowledge, this conjecture has yet not been proven. Assuming this conjecture, he proves that $\{\De_w\}_{w
\in W}$ is linearly independent (Lemma 3.20), acts perfectly\footnote{This is effectively the same statement as the conjectural assumption (3.12.2)!} on the coinvariant algebra (Theorem
3.25), and further remarks that $\De_{w_0}$ agrees with $J$ up to nonzero scalar (Proposition 3.18). Unlike \cite{RamShoIandII}, there is no result stated that $\De_{w_0}$ is a
Frobenius trace (assuming the conjecture).

It would be interesting to try to construct analogs of the above work for the presentation of $G(m,m,n)$ as a quotient of $W_{\aff}$, with
simple reflection $s_0$ instead of $s_{1'}$. Is there a length function and a set of special positive roots\footnote{In \cite[part I]{RamShoIandII} they pin down the Bremke-Malle
length function as the unique function satisfying certain properties. One desires an analogous statement for this potential alternate length function.} which is compatible not just
with restriction to the standard copy of $S_n$, but also with restriction to the many different copies of $S_n$ within $W_{\aff}$? This seems like a tall order, but what is the
appropriate weakening? Can one find canonical coset representatives $N$ for each $S_n$, and a notion of length respecting this length function? Can one construct special elements
$\pa_w \in \NC(m,m,n)$ which are linearly independent and act perfectly on the coinvariant algebra?

%% file: FormulaROUpaperone.tex
\section{Computing the Frobenius trace} \label{sec:closedformularou}

Throughout this chapter, $n=3$, $z$ is a formal variable, $\ze$ is a primitive $3m$-th root of unity, and $R = R_m$. Computations take place either within $\NH(z,3)$ or $\NH(m,m,3)$. The conventions we establish here will also be used in the sequel \cite{EJY2}. We refer the reader to \S\ref{ssec:computations} for a preliminary discussion of the contents in this chapter.

\subsection{Choice of reduced expression, and setting expectations}

\begin{defn} Fix $a, b \ge 0$, and $i \in \Om$. Suppose that $\ws_{i_R,b} = \ws_{j_L,b}$, or if $b = 0$ set $j+1 = i$. Then let
	\begin{equation} \un{w}(a,b,i) = \cw_{j_R, a} \circ (j+1) \circ \ws_{i_R, b}. \end{equation}
This word has length $\ell = a+b+1$, and ends in $i$. \end{defn}

\begin{ex} We have $\un{w}(3, 5, 2) = (1,2,3,1,3,2,1,3,2)$. One can view this as a clockwise word of length $4 = a+1$ and a widdershins word of length $6 = b+1$ which overlap in the middle index $1 = j+1$. The last index is $i = 2$. \end{ex}

\begin{ex} Here are some edge cases. We have $\un{w}(0,0,i) = (i)$. we have $\un{w}(0,b,i) = \ws_{i_R, b+1}$ and $\un{w}(a,0,i) = \cw_{i_R, a+1}$. \end{ex}

\begin{lem} \label{lem:abiparametrize} Every non-identity element of $W_{\aff}$ has a reduced expression $\un{w}(a,b,i)$ for a unique triple $(a,b,i)$. We write this element as $w(a,b,i)$. \end{lem}

\begin{proof} We write this proof mostly to help the reader familiarize themselves with $W_{\aff}$ if necessary. We encourage the reader to use their favorite visualization tool for $W_{\aff}$ or its Coxeter complex (one which indicates the two-sided cells) when following this discussion. We recommend
\url{https://www.jgibson.id.au/lievis/affine_weyl/#.cells:twosided,labels:rex,simples:one-}, written by Joel Gibson, and we have taken the following picture from this website with permission.

\[ \ig{.15}{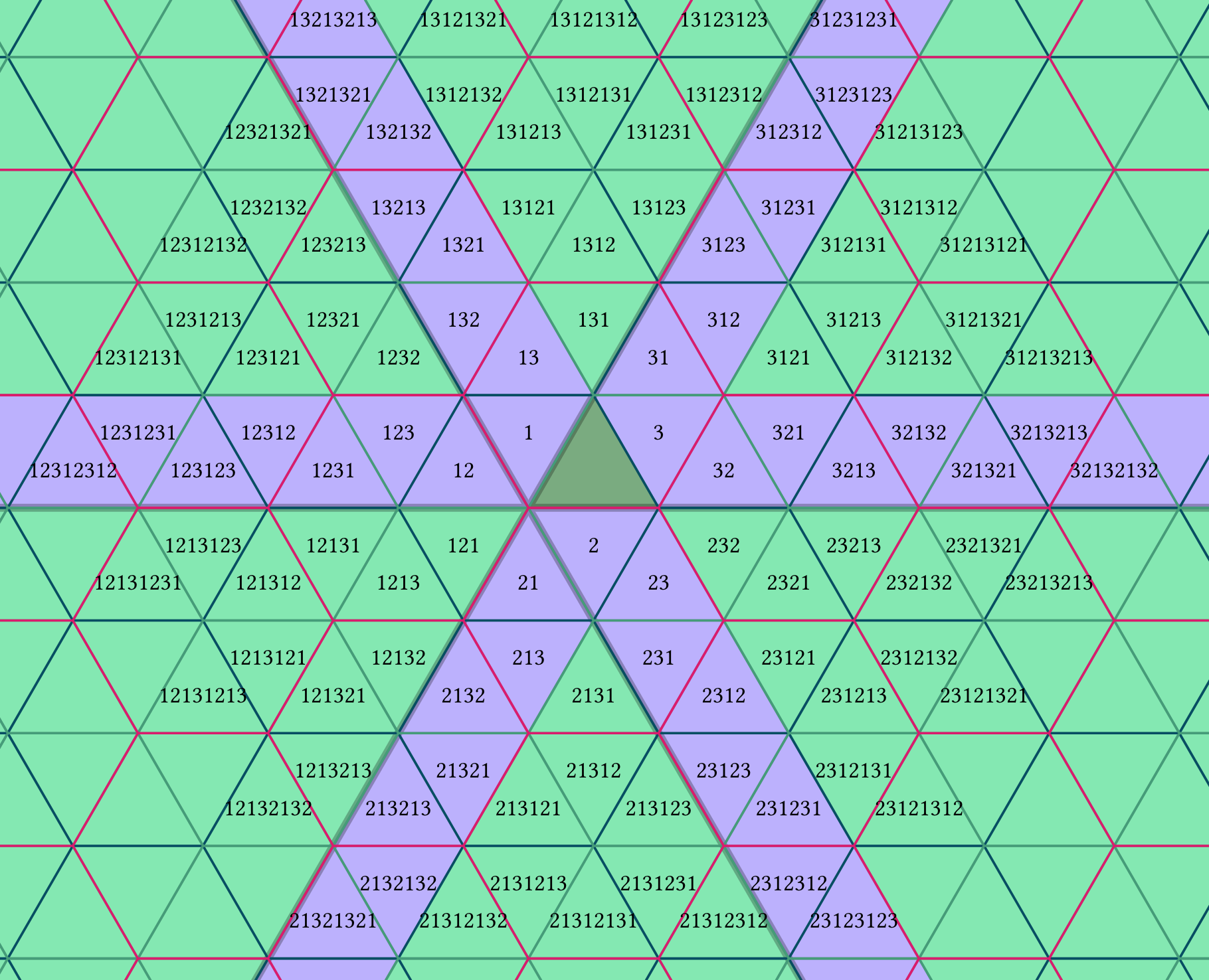} \]

Outside of the identity element, $W_{\aff}$ contains a two-sided cell (purple above) consisting of elements with a unique reduced expression, which form six rays emanating out from
the identity element. One ray is $\un{w}(a,0,1)$ as $a$ varies, which we denote as $\un{w}(-,0,1)$; this is the ray going due left of the identity element. The six rays are $\un{w}(-,0,i)$ and $\un{w}(0,-,i)$ for $i = 1, 2, 3$. Removing the six rays, what remains is the top two-sided cell (green above), which is split
into six connected regions we call \emph{hextants}.

We claim that the elements $\un{w}(a,b,1)$ for $a,b > 0$ comprise precisely two of these hextants, the ones which abut the ray $\un{w}(-,0,1)$. Similar statements hold for $\un{w}(a,b,i)$ for $i = 2,3$; we focus on $i=1$ for an example. The alcove associated to $\un{w}(a,0,1)$ is a triangle which shares an edge with the lower left hextant when $a$ is odd (and thus the length of the word is even), or with the upper left hextant when $a$ is even. The continuation $\un{w}(a,1,1)$ will step into the adjacent hextant across this edge. Fixing $a$ and letting $b \ge 0$ vary, the elements $\un{w}(a,b,1)$ form a ray (which we denote by $\un{w}(a,-,1)$) traveling through this hextant, parallel to the other bounding ray of the hextant. The rays $\un{w}(a,-,1)$ for $a=1, 2, 3, 4$ are depicted in yellow below, as spurs off the ray $\un{w}(-,0,1)$ which is red; when $a$ is odd the rays go down and left, and when $a$ is even the rays go up and left.

\[ \ig{.15}{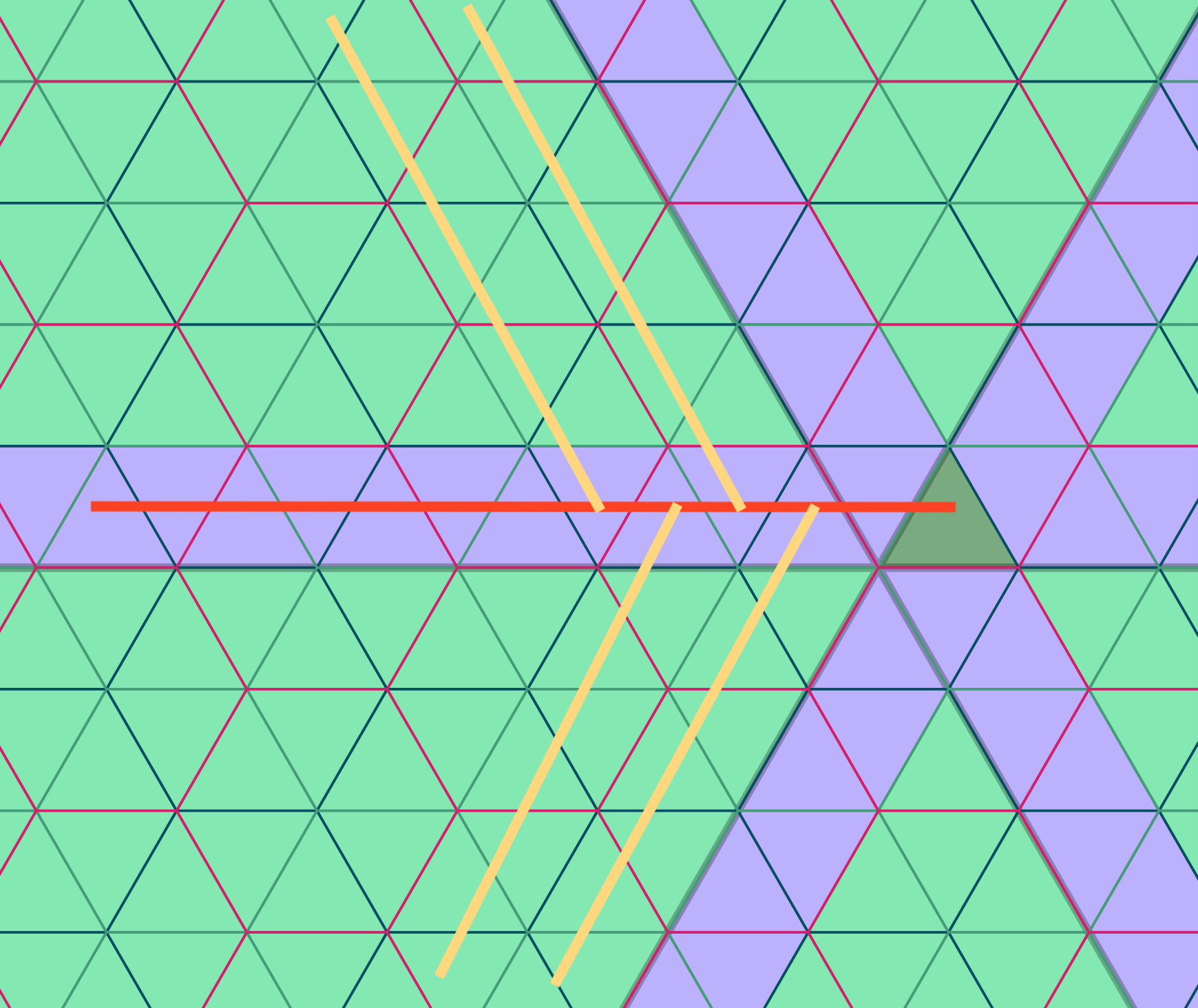} \]

Rotate this picture by 120 degrees to visualize $\un{w}(a,b,2)$ or $\un{w}(a,b,3)$. It is clear that, collectively, these rays bijectively exhaust $W_{\aff}$. \end{proof}

\begin{rem} \label{rmk:maximalcyclicsubwords} One useful feature of this parametrization is the following: no reduced expression of $w(a,b,i)$ has a clockwise subword of length $> a+1$ or a widdershins subword of length $> b+1$. Thus $\un{w}(a,b,i)$ maximizes the length of cyclic subwords. We do not use this result, but a sketch of the proof goes along the following lines. If one starts at any element $w \in W_{\aff}$ any composes on the right with a cyclic word of growing length, the result will be a ray pointing in one of six directions. Three of these directions are reserved for clockwise words, and three for widdershins, just as for the original six rays emanating from the origin in the (purple) unique reduced expression cell. Let $(x, y) \in \NN \times \NN$, and consider a usual square grid. When following grid lines from the origin to the point $(x,y)$ in the plane, when taking a minimal length path one can go at most $x$ steps to the right and at most $y$ steps up. One obtains the result by translating this idea to a triangular grid. \end{rem}

Let us describe some features of the element $w(a,b,i)$: \begin{enumerate}
\item If $a = 0$ or $b=0$, it has a unique reduced expression. For the remaining points we assume $a, b > 0$.
\item If $b$ is odd, then $w(a,b,i)$ has two elements in its right descent set. Its alcove is oriented the same way as its surrounding hextant, with one point aiming towards the identity.
\item If $b$ is even, then $w(a,b,i)$ has one element in its right descent set. Its alcove is oriented the opposite way to its surrounding hextant, with one point aiming away from the identity.
\item If $a$ is odd, then $w(a,b,i)$ has two elements in its left descent set. Such elements live in an \emph{inner hextant}, a hextant with a minimal length element of length $3$.
\item If $a$ is even, then $w(a,b,i)$ has one element in its left descent set. Such elements live in an \emph{outer hextant}, a hextant with a minimal length element of length $4$.
\end{enumerate}
Because of these differences, one should not be surprised to see our formulas for Demazure operators split into cases based on the parity of $a$ and $b$, with additional subcases when either $a$ or $b$ is zero.

\begin{lem} \label{lem:killx123} Let $\un{w}$ be a word of length $\ell$ and $f \in R_z$ be a monomial of degree $\ell$. If $x_1 x_2 x_3$ divides $f$ then $\pa_{\un{w}}(f) = 0$.
\end{lem}

\begin{proof} Suppose $f = g \cdot x_1 x_2 x_3$. Since $x_1 x_2 x_3$ is $W_{\aff}$-invariant, we have $\pa_{\un{w}}(f) = x_1 x_2 x_3 \pa_{\un{w}}(g)$. But for degree reasons,
$\pa_{\un{w}}(g) = 0$. \end{proof}

Because of Lemma \ref{lem:killx123}, we need only examine $\pa_{\un{w}}(f)$ when $f = x_1^k x_2^{\ell - k}$ or $x_2^k x_3^{\ell - k}$ or $x_3^k x_1^{\ell - k}$. But using the symmetry $\si$ and
\eqref{siondem}, we can reduce to the case when $f = x_1^k x_2^{\ell - k}$. We can also eliminate the case $k = 0$ if desired, since $\si$ will reduce it to the case of $k = \ell$.

The scalars we aim to compute are below.

\begin{notation} Suppose $w = w(a,b,i) \in W_{\aff}$, and $f$ is a polynomial of degree $\ell(w) = a+b+1$. Then we set
\begin{equation} \Xi(w,f) = \Xi(a,b,i,f) := \pa_{\un{w}(a,b,i)}(f). \end{equation}
For an integer $0 \le k \le a+b+1$ we set
\begin{equation} \Xi(w,k) = \Xi(a,b,i,k) := \Xi(w, x_1^k x_2^{a+b+1-k}). \end{equation}
These scalars are computed in $\NH(z,3)$, before specializing to a root of unity. To emphasize the fact that they live in $\ZZ[z,z^{-1}]$, we may write $\Xi(w,f)(z)$. When we specialize this scalar by setting $z \mapsto \ze$, we write $\Xi(w,f)(\ze)$ for the result. We set
\begin{equation} \Xi_m(a,i) = \Xi(a,3m-1-a,i,2m)(\ze). \end{equation}
This restricts our attention to elements of length $3m$ acting on the staircase mononomial $\stair = x_1^{2m} x_2^m$. In this context, we always assume $b := 3m-1-a$.
\end{notation}

One should expect the formulas for $\Xi(a,b,i,k)$ to depend on the parity of $a$ and $b$, as already noted. One should also expect that
$k = a+b+1$ or $k=0$ is an edge case with its own set of formulas. After all, the monomial $x_1^k$ possesses additional symmetry which other monomials lack. Finally, one should
expect three different formulas for $i=1$, $i=2$, and $i=3$; we have broken the symmetry by fixing our attention on $x_1^k x_2^{\ell - k}$. There are an immense number of cases. Thankfully, all the formulas with $a, b > 0$ and $0 < k < a+b+1$ have a great deal in common, with only minor annoying issues (a sign, a power of $z$, an offset in indexing) which depend on the specific
case. We say that quadruples $(a,b,i,k)$ with $a, b > 0$ and $0 < k < a+b+1$ live in the \emph{standard regime}.

We decided that it was not worth reprinting the formula in this paper, even only the formula in the standard regime, but it is discussed thoroughly in the sequel \cite[Section 5]{EJY2}. Instead we focus on $\Xi_m(a,i)$.

\begin{rem} \label{alreadyzero}
By \eqref{eq:longcyclicvanishes}, $\Xi_m(a,i) = 0$ if either $a \ge 2m+1$ or $b \ge 2m+1$, because then $\un{w}(a,b,i)$ contains a cyclic word of length at least $2m+2$. Since $a+b= 3m-1$, $\Xi_m(a,i)$ vanishes if either $a \ge 2m+1$ or $a \le m-2$. In particular, only the standard regime is relevant for the computation of $\Xi_m(a,i)$.
\end{rem}

\subsection{Extending scalars}

The formula for $\Xi(a,b,i,k)$ and $\Xi_m(a,i)$ will involve quantum numbers. Up to an invertible scalar, these are quantum numbers in $q$, not in $z$! Remember that $z^3 = q^{-2}$.
Thus the scalar $1 + z^3$ agrees with the unbalanced quantum number $(2)_{q^{-2}} = 1 + q^{-2}$, and $1 + z^3 + z^6$ agrees with $(3)_{q^{-2}} = 1 + q^{-2} + q^{-4}$, etcetera.
Unbalanced quantum numbers live within the ring $\ZZ[z^3]$.

To prove our formulas we will need various manipulations of quantum numbers, as well as their relationship to trigonometry. These are easier to state for balanced quantum numbers like $[2]_q = q + q^{-1}$
and $[3]_q = q^{-2} + 1 + q^2$. These can only be expressed using half-powers of $z$, e.g. $[2]_q = z^{\frac{3}{2}} - z^{-\frac{3}{2}}$. To this end, we introduce a square root of $z$, which we call $p$. Then 
\begin{equation} z = p^2, \qquad q = p^{-3}. \end{equation}
For example, $p^3[2]_q = 1 + z^3$. We give our formulas as expressions involving powers of $p$, $q$, and $z$, whatever is most convenient, but the result ultimately lives inside the ring
$\ZZ[z,z^{-1}] \subset \ZZ[p,p^{-1}]$.

In our formulas we will often need both ordinary binomial coefficients and quantum binomial coefficients. Ordinary binomial coefficients, which will be denoted $\binom{k}{c}$ and live in $\ZZ$, often appear within exponents. Balanced quantum binomial coefficients with respect to the variable $q$ will be denoted ${k \brack c}$.

\subsection{Formula at a root of unity} \label{ssec:evalatrou}

Here is a closed formula for $\Xi_m(a,i) = \Xi(a,3m-1-a, 2m, i)(\ze)$.

\begin{notation} \label{notation:bottom} Suppose $a + b + 1 = 3m$. Let $b = 2 \beta+1$ or $2 \beta + 2$, depending on parity. Let $m = 2d$ or $m = 2d+1$ depending on parity. Unless $m$ and $a$ and $b$ are odd, set $\bottom = d-1$. If $m$ and $a$ and $b$ are odd, set $\bottom = d$. \end{notation}
	
\begin{rem} Let $a = 2 \alpha+1$ or $2 \alpha+2$, depending on parity. The index $\bottom$ is designed so that $\bottom \le \alpha$ if and only if $\beta \le m-1$ if and only if $b \le 2m$. This is needed for $\Xi_m$ to be nonzero, see Remark \ref{alreadyzero}. Similarly, $\bottom \le \beta$ if and only if $\alpha \le m-1$ if and only if $a \le 2m$.
	
One also has $(\alpha - \bottom) + (\beta - \bottom) = m-1-\bottom$. Thus 
\[ {m-1-\bottom \brack \alpha - \bottom} = {m-1-\bottom \brack \beta - \bottom}.\] \end{rem}

\begin{thm} \label{thm:whatisw0} Let $m = 2d$ or $m = 2d+1$ depending on parity. Fix $a$ and $i$, let $b = 3m-1-a$, and let $\beta, d, \bottom$ be as in Notation \ref{notation:bottom}. If $\bottom \le \beta \le m-1$ then
\begin{equation} \Xi_m(a,i) = (-1)^d m^2 {m-1-\bottom \brack \beta - \bottom} \cdot \blahblah \end{equation}
where
\begin{equation} \blahblah = p^{2\binom{\beta}{2} - 7d} \cdot \begin{cases} p^{- 3 \beta d + \beta - 1} & \text{ if $m$ is even and $a$ is even,} \\ 
	-p^{ - 3 \beta d + 5\beta + 3} & \text{ if $m$ is even and $a$ is odd,} \\
	(-1)^{\beta+1} p^{-9 \beta d - \beta - 2} & \text{ if $m$ is odd and $a$ is even,} \\ 
	(-1)^{\beta+1} p^{-9\beta d - 2 \beta - 3} & \text{ if $m$ is odd and $a$ is odd.} \end{cases}
\end{equation}
Otherwise, $\Xi_m(a,i) = 0$.
\end{thm}

\begin{proof} This is \cite[Theorem 6.7]{EJY2}. \end{proof}

Note that this formula for $\Xi_m(a,i)$ is independent of $i$! This is essential given Corollary \ref{cor:sitauonJ}. This theorem serves to classify which elements of $W_{\aff}$ give rise to a Frobenius trace.

\begin{cor} \label{cor:actuallymaintheorem} Let $\un{w} \in W_{\aff}$ be any reduced expression of length $3m$. If it represents the same element as $\un{w}(a,b,i)$, then $\pa_{\un{w}}$ is a Frobenius trace if and only if $a, b \le 2m$. \end{cor}

\begin{proof} It is a restatement of Lemma \ref{lem:uniqueuptoscalar} that $\pa_{\un{w}(a,b,i)} = \Xi_m(a,i) \cdot J$, and Theorem \ref{thm:JFrob} states that $J$ is a Frobenius trace
map. Thus $\pa_{\un{w}(a,b,i)}$ is a Frobenius trace map if and only if $\Xi_m(a,i)$ is nonzero (see also Corollary \ref{cor:nonzeroonstairgoodenough}). Any other reduced expression for
the same element differs from $\un{w}(a,b,i)$ by a sequence of braid relations, the effect of which is merely to rescale $\pa_{\un{w}}$ by a unit. \end{proof}

%% file: Experimental.tex
\section{Examples and Experiments} \label{sec:exex}

\subsection{The example of $G(2,2,3)$} \label{ssec:223intro}

Let us look at the simplest interesting example: $n=3$ and $m=2$. Let $\ze$ be a primitive $6$-th root of unity. As for all examples with $n=3$, we let $s = s_1$, $t = s_2$, $u = s_0$.

As a quotient of $W$, $G(2,2,3)$ has the presentation
\begin{equation} \label{G223present} G(2,2,3) = \langle s,t,u \mid s^2 = t^2 = u^2 = 1, sts=tst, tut=utu, usu=sus, stsu=usts \rangle. \end{equation}
In addition to the usual Coxeter relations for $W$, the final relation states that $t_{\lon} s_0 = s_0 t_{\lon}$, or that $(s_0 t_{\lon})^m = 1$.

Let $\CC^4 = \CC \langle z_1, z_2, z_3, z_4\rangle$ be the defining representation of $S_4$, where $S_4$ acts to permute the variables. The standard representation of $S_4$ can be viewed as a subspace of elements $\sum a_i z_i$ where $\sum a_i = 0$.

\begin{thm} There is a group isomorphism
\begin{equation} \label{G223isomS4} \phi \co G(2,2,3) \stackrel{\sim}{\rightarrow} S_4, \qquad s \mapsto (12), \quad t \mapsto (13), \quad u \mapsto (14). \end{equation}
The following map $V_m \into \CC^4$ induces an equivariant (relative to $\phi$) isomorphism from $V_m$ to the standard representation:
\begin{equation} x_1 \mapsto z_1 - z_2 + z_3 - z_4, \quad x_2 \mapsto \ze^{-1}(-z_1 + z_2 + z_3 - z_4), \quad x_3 \mapsto \ze^{-2}(z_1 + z_2 - z_3 - z_4). \end{equation}
\end{thm}

\begin{proof} This is a straightforward computation. \end{proof}
	
The rings $R = \CC[x_1, x_2, x_3]$ and $\CC[z_1, z_2 ,z_3, z_4]/(e_1)$ are isomorphic as $G(2,2,3)$-modules (i.e. as $S_4$-modules), with the same invariant subrings. However, the simple reflections $\{s,t,u\}$ of $G(2,2,3)$ do not act the same way as the simple reflections $\{(12), (23), (34)\}$ of $S_4$, and the associated Demazure operators are different.

Let us equip $G(2,2,3)$ with a length function, sending a group element to the length of the minimal expression via the presentation \eqref{G223present}. The Poincar\'{e}
polynomial of this length function is 
\begin{equation} \xi = 1 + 3v + 6v^2 + 9v^3 + 5v^4.\end{equation}
This is in contrast with the Poincar\'{e} polynomial for the usual length function on $S_4$, which is 
\begin{equation} \label{eq:pi223} \pi = 1 + 3v + 5v^2 + 6v^3 + 5 v^4 + 3 v^5 + v^6.\end{equation}
This $\pi$ is also the Poincar\'{e} polynomial of the usual nilCoxeter algebra of $S_4$ (where $v$ represents degree $-1$ for this purpose), which is 24-dimensional.

The exotic nilCoxeter algebra $\NC(2,2,3)$ has its nil-quadratic relations
\begin{equation} \label{quad} \pa_s^2 = \pa_t^2 = \pa_u^2 = 0 \end{equation}
and its braid relations
\begin{equation} \label{braidish} \ze \pa_s \pa_t \pa_s = \pa_t \pa_s \pa_t, \quad \ze \pa_t \pa_u \pa_t = \pa_u \pa_t \pa_u, \quad \ze \pa_u \pa_s \pa_u = \pa_s \pa_u \pa_s, \end{equation}
see \S\ref{ssec:nilcoxbasics}. These relations held even for the nilCoxeter algebra $\NC(z,3)$ of $V_z$, before specialization to a root of unity, but the remaining relations hold only at a root of unity. The roundabout relations are
\begin{subequations}
\begin{equation} \label{roundaboutm2n3}\pa_s \pa_t \pa_u \pa_s + \ze^{2} \pa_t \pa_u \pa_s \pa_t + \ze^{4} \pa_u \pa_s \pa_t \pa_u = 0, \end{equation}
\begin{equation} \pa_s \pa_u \pa_t \pa_s + \ze^{2} \pa_t \pa_s \pa_u \pa_t + \ze^{4} \pa_u \pa_t \pa_s \pa_u = 0. \end{equation}
\end{subequations}
The final relations needed to present $\NC(2,2,3)$ are
\begin{subequations}
\begin{equation} \pa_s \pa_t \pa_u \pa_t \pa_s + \pa_t \pa_s \pa_u \pa_s \pa_t = \pa_s \pa_t \pa_u \pa_s \pa_t + \pa_t \pa_s \pa_u \pa_t \pa_s, \end{equation}
\begin{equation} \pa_s \pa_u \pa_t \pa_u \pa_s + \pa_u \pa_s \pa_t \pa_s \pa_u = \pa_s \pa_u \pa_t \pa_s \pa_u + \pa_u \pa_s \pa_t \pa_u \pa_s, \end{equation}
\begin{equation} \pa_u \pa_t \pa_s \pa_t \pa_u + \pa_t \pa_u \pa_s \pa_u \pa_t = \pa_u \pa_t \pa_s \pa_u \pa_t + \pa_t \pa_u \pa_s \pa_t \pa_u. \end{equation}
\end{subequations}
Note that each roundabout relation is preserved (up to scalar) by $\si$ and sent to the other roundabout relation by $\tau$, see Definition \ref{defn:symmetries}. The other three relations form an orbit under the action of $\si$ and $\tau$. So, up to the symmetries $\si$ and $\tau$, only two relations are needed to present $\NC(2,2,3)$ (beyond the generic relations of $\NC(z,3)$).

\begin{rem} Reversing the order of a word is an antiinvolution of $\NC(2,2,3)$, because it preserves all the relations. \end{rem}

The graded algebra $\NC(2,2,3)$ is 36-dimensional. For ease of discussion, we pretend it is positively graded rather than negatively graded.  It has Poincar\'{e} polynomial 
\begin{equation} \Xi = 1 + 3v + 6v^2 + 9 v^3 + 10v^4 + 6 v^5 + v^6.\end{equation}
It is generated in degree $1$, so it is impossible to find a subalgebra with the graded dimension $\pi$ or $\xi$. Nor is there a quotient algebra of graded dimension $\pi$, see the next remark.

However, $\NC(2,2,3)$ does have a left module with graded dimension $\pi$, because that is the dimension of the coinvariant algebra $C$. Recall that $C := R / (R^W_+)$, where
$R^W_+$ denotes the ideal in $R$ generated by the positive degree elements of $R^W$. The action of $\NC(2,2,3)$ on $R$ preserves $(R^W_+)$ since Demazure operators are $R^W$-linear,
so it descends to $C$. The action on $C$ is faithful; any operator killing $C$ will kill all of $R$ by $R^W$-linearity. It is easy to verify in this small example that $C$ is
cyclic, generated by the staircase polynomial $x_1^4 x_2^2$ as a module over $\NC(2,2,3)$. We now identify the corresponding left ideal.

The algebra $\NC(2,2,3)$ has a special element in degree $2$: \begin{equation} \gamma = \pa_{ts} - \ze \pa_{ut} + \ze^2 \pa_{su} - \ze^2 \pa_{us} + \ze \pa_{tu} - \pa_{st}.
\end{equation} It is killed by multiplication on the left (resp. on the right) by any element of degree $4$ in $\NC(2,2,3)$, and this property picks $\gamma$ out uniquely up to
scalar. It is an eigenvector for the symmetries $\si$ and $\tau$. It acts by zero on any degree $2$ polynomial, though it is nonzero in higher degree.

\begin{lem} The coinvariant algebra $C$ is isomorphic as a left $\NC(2,2,3)$ module to the quotient by the left ideal generated by $\gamma$. \end{lem}

\begin{proof} Define a surjective left module map $\psi \colon \NC(2,2,3) \to C$ by acting on $x_1^4 x_2^2$. One can verify that
\begin{align} \nonumber \frac{1}{2} \gamma(x_1^4 x_2^2) = & x_1 x_2^2 x_3 - \ze x_1 x_2 x_3^2 + \ze x_2^3 x_3 - x_2 x_3^3 \\ = & \ze^{-1} x_2 x_3 (x_1^2 + \ze^2 x_2^2 + \ze^4 x_3^2) + 2 (-\ze^{-1} x_1 + x_2 - \ze x_3) x_1 x_2 x_3. \end{align}
Recall that $R^W$ is generated by $x_1 x_2 x_3$ and $x_1^2 + \ze^2 x_2^2 + \ze^4 x_3^2$, so that $\gamma(x_1^4 x_2^2) \in (R^W_+)$. So $\psi(\gamma) = 0$. However, a computer calculation (found at \cite[\texttt{ActionOfGamma.m}]{EJYcode}) shows that the quotient of $\NC(2,2,3)$ by the left ideal of $\gamma$ has graded dimension $\pi$, so $\psi$ must descend to an isomorphism modulo this ideal. \end{proof}

\begin{rem} The quotient by the right ideal generated by $\gamma$ is a right module which also has dimension $\pi$. However, the two-sided ideal generated by $\gamma$ is too large, containing everything in degree $\ge 4$, so there is no algebra quotient of graded dimension $\pi$. \end{rem}

An important observation is that both the ordinary and exotic nilCoxeter algebras are one-dimensional in top degree (degree $6$). The graded dimension of $R$ over $R^W$ is
$\pi$, so (up to scalar) there is one $R^W$-linear map $R \to R$ of degree $-6$, and it is the Frobenius trace (with image lying in $R^W$). Thankfully, this one-dimensional space is
achieved within $\NC(2,2,3)$, so there is at least one word of length $6$ in the generators $\{\pa_s, \pa_t, \pa_u\}$ which is equal to the Frobenius trace map (up to scalar).

There are 18 elements of length $6$ in the affine Weyl group $W$. By pre- and
post-composing \eqref{roundaboutm2n3} with $\pa_s$ and $\pa_u$, and applying the nil-quadratic relation, one can deduce that $\pa_{stustu} = 0$. Similarly, the
following elements of $W$ give zero in degree $-6$ inside $\NC(2,2,3)$: \begin{equation} stustu, sutsut, tsutsu, tustus, ustust, utsuts. \end{equation} The other twelve elements
are all obtained from $stsust$ by applying the symmetries $\si$ and $\tau$ and by reversing the word. These twelve all serve as the
``longest element'' in $G(2,2,3)$, in that $\pa_w$ is a Frobenius trace. In fact, all twelve nonzero operators $\pa_w$ are equal up to a power of $\zeta$.

\begin{rem} Theorem \ref{thm:whatisw0} states precisely what these powers of $\zeta$ should be. Note that the quantum binomial coefficients ${m-1-\bottom \brack \alpha - \bottom}$ 
appearing in Theorem \ref{thm:whatisw0}, when $m=2$, are equal to either ${1 \brack 0}$ or ${1 \brack 1}$ or zero. \end{rem}

\subsection{The example of $G(3,3,3)$} \label{ssec:333intro}

Let $\ze$ be a $9$-th root of unity. The group $G(3,3,3)$ has size $54$, having presentation
\begin{equation} G(3,3,3) = \langle s,t,u \mid s^2 = t^2 = u^2 = 1, sts=tst, tut=utu, usu=sus, stsust=ustsus \rangle. \end{equation}

The exotic nilCoxeter algebra has a presentation with the nil-quadratic relations \eqref{quad}, the braid relations \eqref{braidish}, and two more relations up to symmetry. First we have the roundabout relation
\begin{equation} \label{roundaboutm3n3} \pa_{stustu} + \ze^{3} \pa_{tustus} + \ze^{6} \pa_{ustust} = 0. \end{equation}
Applying $\tau$, we obtain another roundabout relation in degree $6$. By pre- and post-composing \eqref{roundaboutm3n3} we can deduce that any \emph{cyclic} word $w$ of length $8$ in $W$ (see \S\ref{ssec:cyclic}), such as $stustust$ or $tsutsuts$, the
operator $\pa_w$ is zero. This implies that certain elements of length $9$ (like $ws$ and $sw$) induce the zero operator, and so forth for longer elements.

We find the following picture of the affine Weyl group helpful.
\begin{equation} \label{m3n3length8} \ig{.15}{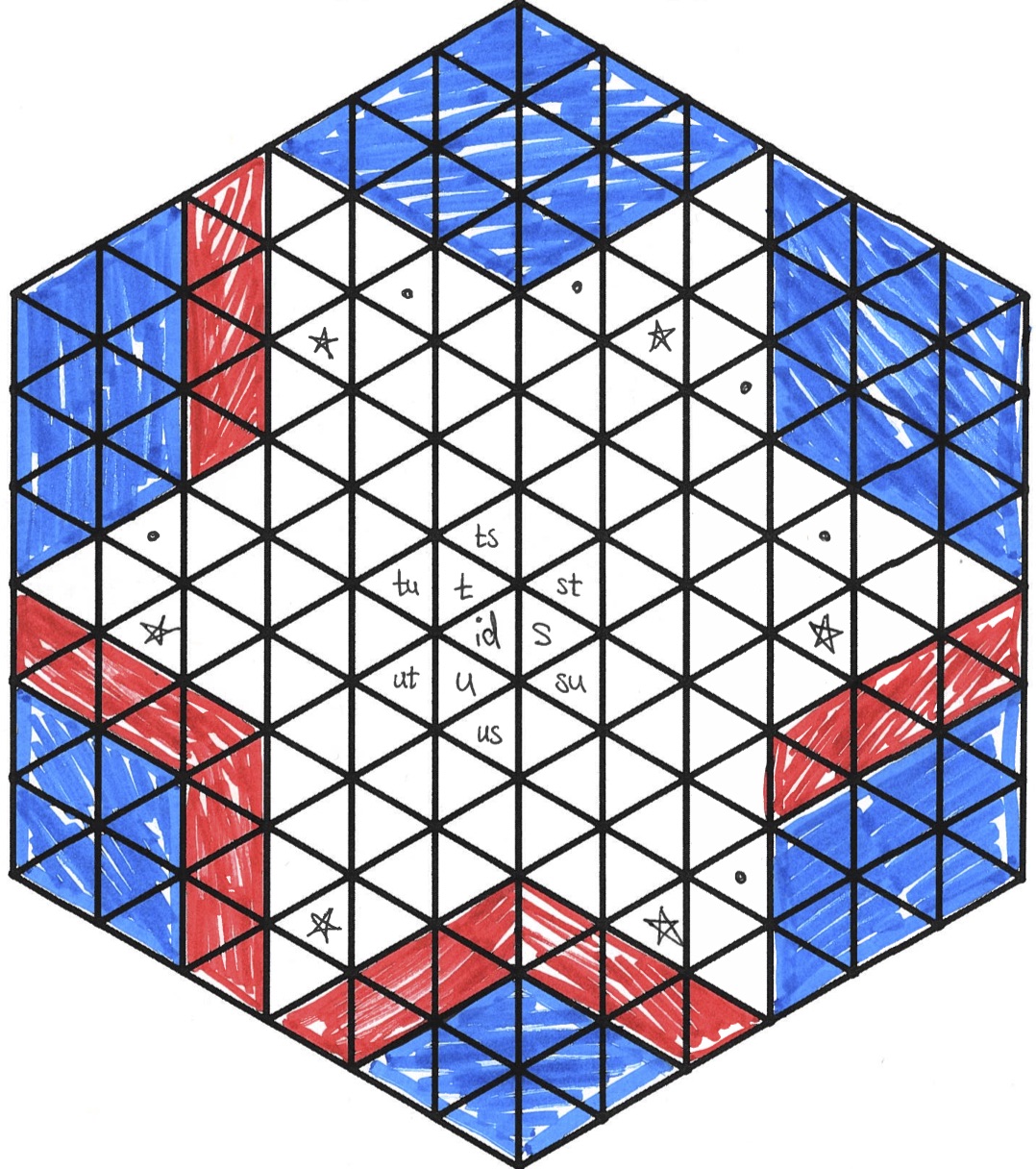} \end{equation}
Shaded in blue are the operators which are zero by the roundabout relation, as in the previous
paragraph. Shaded in red are non-zero operators which are not needed in a basis, because of the roundabout relations. For example, using \eqref{roundaboutm3n3} we need not use any
operator which begins with $\pa_{ustust}$ or (via the $\tau$ conjugate) $\pa_{utsuts}$. Precomposing \eqref{roundaboutm3n3} with $\pa_s$ or $\pa_t$, we can also rewrite
$\pa_{tustust}$ and $\pa_{sutsuts}$ in terms of other operators. The star-shape that remains after removing the red and blue shaded regions is a basis of the algebra obtained only
by imposing the roundabout relations but not \eqref{m3n3relation}, and its dimension matches Conjecture \ref{conj:justtheroundabout}. This behavior is typical as $m$ increases.

In addition to the roundabout relations, we need six relations in degree $8$. One of them is
\begin{equation} \label{m3n3relation} \pa_{stsutsut} = \pa_{utusutus} - \ze^3 \pa_{tusutusu}. \end{equation}
The orbit of this relation under $\si$ and $\tau$ is six-dimensional (rather than $3$-dimensional as happened for $m=2$). One way to think of these relations is as follows. The starred alcoves in the picture above are elements of the form $(sts)u(sts)u$ and $u(sts)u(sts)$, and their orbit under $\sigma$. They give rise to a basis for $\NC(3,3,3)$ in degree $8$, namely
\begin{equation} \label{m3n3basis8} \{ \pa_{stsustsu}, \pa_{sustsust}, \pa_{tutstuts}, \pa_{sutusutu}, \pa_{ustsusts}, \pa_{tsustsus} \}. \end{equation}
The six dotted alcoves need to be rewritten in terms of this basis. The relation \eqref{m3n3relation} and its conjugates more obviously rewrite the $\si$ and $\tau$ orbit of $stsutsut$ in terms of this basis. Some elements in this orbit live in the red shaded region; applying roundabout relations one reaches the remaining
dotted alcoves.

\begin{rem} Once again, reversing the order of the word gives an anti-involution of $\NC(3,3,3)$. \end{rem}


The Poincar\'{e} polynomial of $\NC_{(3,3,3)}$ is
\begin{equation} \label{poinpoly333} 1 + 3v + 6v^2 + 9v^3 + 12 v^4 + 15v^5 + 16v^6 + 15v^7 + 6 v^8 + v^9.\end{equation} It has dimension $84$. The
one-dimensional space in degree $-9$ is the minimal degree possible, and serves as a Frobenius trace map. Indeed, every element of length $9$ which is not in a blue-shaded region gives a nonzero Demazure operator. Choosing a particular normalization for the Frobenius trace $\pa_{W_m}$, some of these operators are equal to $\pa_{W_m}$, and some of these operators are $(1 + \ze^3) \pa_{W_m}$ (all up to a power of $\ze$). Note that $1 + \ze^3 = 1 + q^{-2}$ is (up to a power of $q$) the quantum number $[2]$. See \eqref{m3n3binomials}.

\subsection{Experimental data and interpretation} \label{ssec:results}

Let us now explain our experimental results, focusing on the case when $n=3$. Every statement in this section has been verified for $m \le \maxcalc$, and should be considered as conjectural in general.

The exotic nilCoxeter algebra $\NC(m,m,3)$ possesses the nil-quadratic relations \eqref{quad} and the braid relations \eqref{braidish}. Imposing only these two relations yields an infinite-dimensional algebra with graded dimension equal to the Poincar\'{e} polynomial of the affine Weyl group
\begin{equation} \label{affineWeylpoincare} 1 + 3v + 6v^2 + 9v^3 + 12v^4 + \ldots + (3k)v^k + \ldots. \end{equation}

The algebra $\NC(m,m,3)$ also has two roundabout relations in degree $2m$, as proven in Theorem \ref{thm:roundabout} and Corollary \ref{cor:otherroundabout}.

\begin{conj} \label{conj:justtheroundabout} Let $A$ be the algebra abstractly generated by $\pa_i$ for $i \in \Om$, modulo the quadratic and braid relations above, and the roundabout relations in degree $m(n-1)$. Then $A$ is a finite-dimensional negatively-graded (non-commutative) algebra with symmetric Poincar\'{e} polynomial, possibly a symmetric algebra or a Frobenius algebra. When $n=3$ it has Poincar\'{e} polynomial
\begin{align} \label{balancedpoincare} 1 + & 3v + 6v^2 + \ldots + (6m-6) v^{2m-2} + (6m-3) v^{2m-1} + (6m-2) v^{2m} + \\ 
	 & (6m-3) v^{2m+1} + (6m-6) v^{2m+2} + \ldots + 6v^{4m-2} + 3v^{4m-1} + v^{4m}. \nonumber \end{align}
\end{conj}

\begin{ex} When $n=3$ and $m=3$ the graded dimension of $A$ is
\begin{equation} 1 + 3v + 6v^2 + 9v^3 + 12v^4 + 15v^5 + 16v^6 + 15v^7 + 12 v^8 + 9 v^9 + 6 v^{10} + 3v^{11} + v^{12}. \end{equation}
The two relations in degree $2m = 6$, which cut down the dimension from $18$ to $16$ in this example.
\end{ex}

We have verified Conjecture \ref{conj:justtheroundabout} for $n=3$ and $m \le 20$ by computer.


However, the nilCoxeter algebra is much smaller than $A$. When $n=3$ it has minimal degree $-3m$ rather than $-4m$. The additional relations are very mysterious. We have written code to compute the relations for $m \le \maxcalc$, see \cite[\texttt{ExoticNilcoxeterPresentation.m}]{EJYcode}. The remaining relations occur in degree $3m-k$ for a relatively small number $k$, so that \eqref{balancedpoincare} controls all but the tail of the graded dimension of $\NC(m,m,3)$. We summarize the graded dimensions and the number of additional relations in this table, which focuses on the tail. 

\begin{table}[h!]
\label{table:gradeddimensions}
\centering
\makebox[\textwidth][c]{\input{gradeddimensions.tex}}
\caption{The graded dimensions of $\NC(m,m,3)$.}
\end{table}

The red numbers in this table are places where the two roundabout relations appear. Above and below the red numbers the table follows the predictable pattern of \eqref{balancedpoincare}. Numbers which follow this pattern are in black or red. However, at some point additional relations are required, as indicated in differing shades of green. The number of these relations is recorded in the ``rel'' column. After these relations are enforced, the remainder of the graded dimension seems to be a stable tail, see \eqref{eq:tail} below. Degrees which are in the tail but for which no relations are required appear in brown.

For $2 \le m \le 5$, we need $3(m-1)$ relations in degree $3m-1$ to ensure that the dimension of $\NC(m,m,3)$ in that degree is always $6$. No longer is there only one orbit of
relations under $\si$ and $\tau$. For $6 \le m \le 13$, we instead have $3(m-5)$ relations in degree $3m-2$ to ensure that the dimension in that degree is always $21$, and this suffices
to imply that the dimension in degree $3m-1$ is $6$. For $14 \le m \le 16$ we enter a period of ``transition metals,'' where we need some relations in degree $3m-3$ and $3m-2$, until
eventually for $17 \le m \le \maxcalc$ we have only relations in degree $3m-3$, exactly the amount needed to ensure that the dimension in that degree is $52$. We have only computed up
to $m=\maxcalc$.

There is a mysterious transition involving an unknown number of relations to what seems to be a stable tail, having the form
\begin{equation} \label{eq:tail} \ldots + 52 v^{3m-3} + 21 v^{3m-2} + 6 v^{3m-1} + v^{3m}. \end{equation}
One would need to compute to approximately $m=30$ to determine the next number in this stable tail, and distinguish between several sequences on the OEIS \cite{OEIS} which might fit. If the next coefficient is $105$ then we might have the tantalizing sequence $[3]!_k v^{3m-k}$, \cite[sequence A069778]{OEIS}. Or perhaps the next coefficient is $103$, and we get \cite[sequence A135454]{OEIS}.

To put this tail in context, just what is the graded dimension of $\End_{R^W}(R)$? By the Shephard-Todd theorem, the ``degrees'' of $G(m,m,3)$ are $3, m, 2m$, meaning that these are the degrees of the generators of $R^W$. More precisely, we have
\begin{equation} R^W = \CC[x_1 x_2 x_3, x_1^m + \ze^m x_2^m + \ze^{2m} x_3^m, x_1^m x_2^m + \ze^m x_1^m x_3^m + \ze^{2m} x_2^m x_3^m]. \end{equation}
By dividing their Poincar\'{e} polynomials, we see that $R$ is a free module over $R^W$ with graded rank (RHS correct for $m > 3$)
\begin{equation} \frac{(1-v^3)(1-v^m)(1-v^{2m})}{(1-v)^3} = 1 + 
\ldots + 9 v^{3m-3} + 6 v^{3m-2} + 3 v^{3m-1} + v^{3m} = \sum \pi_i v^i. \end{equation}
Meanwhile, $R$ has graded rank
\[ \sum \ga_i v^i = \frac{1}{(1-v)^3} = 1 + 3v + 6v^2 + 10v^3 + \ldots. \]
Then the dimension of $\End_{R^W}(R)$ in degree $-k$ is equal to
\begin{equation} d_k = \sum_{i \ge k} \pi_i \ga_{i-k}, \end{equation}
which is an upper bound on the possible dimension of $\NC(m,m,3)$. For small values of $k$ this upper bound is much larger than $\NC(m,m,3)$.

Here are the last nonzero values of $d$ (correct for $3 < m$):
\begin{equation} d_{3m} = 1, d_{3m-1} = 6, d_{3m-2} = 21, d_{3m-3} = 55. \end{equation}
Thus in degrees $3m$, $3m-1$, and $3m-2$ the subalgebra $\NC(m,m,n)$ is stably generic, in the sense that every possible $R^W$-module map is realized by Demazure operators. However, in degree $3m-3$ the subalgebra $\NC(m,m,n)$ is not stably generic, and some unknown constraints are being placed upon this subalgebra.
 
We conclude with one more conjecture.

\begin{conj} There is a $\ze$-linear anti-involution of $\NC(m,m,n)$ which sends $\pa_i$ to $\pa_i$ for all $i \in \Om$. \end{conj}

In order words, reversing the order of a word is a conjectural symmetry of $\NC(m,m,n)$. Clearly word reversal preserves the quadratic, braid, and roundabout relations. We do not know an a priori reason this symmetry should exist.


%
%

%% file: gradeddimensions.tex
\begin{tabular}{r | r r r r r r || r ||| r | r r r r r r r || r |||}
$m$ & 2 & 3 & 4 & 5 & 6 & 7 & & & 12 & 13 & 14 & 15 & 16 & 17 & 18 & \\
$\deg$ & &&&&&& rel & $\deg$ & &&&&&&& rel \\
\hline
0 & 1 & 1 & 1 & 1 & 1 & 1 & &											32 & 48 & 60 & 72 & 84 & {\color{red} 94} & 96 & 96 & \\
1 & 3 & 3 & 3 & 3 & 3 & 3 & &											33 & 45 & 57 & 69 & 81 & 93 & 99 & 99 & \\
2 & 6 & 6 & 6 & 6 & 6 & 6 & &											34 & {\color{Green} 21} & 54 & 66 & 78 & 90 & {\color{red} 100} & 102 		& 21  \\
3 & 9 & 9 & 9 & 9 & 9 & 9 & &											35 & {\color{Mahogany} 6} & 51 & 63 & 75 & 87 & 99 & 105 & \\
4 & {\color{red} 10} & 12 & 12 & 12 & 12 & 12 & &						36 & {\color{Mahogany} 1} & 48 & 60 & 72 & 84 & 96 & {\color{red} 106} & \\
5 & {\color{LimeGreen} 6} & 15 & 15 & 15 & 15 & 15 & 	3 &			37 & & {\color{Green} 21} & 57 & 69 & 81 & 93 & 105 						& 24 \\
6 & {\color{Mahogany} 1} & {\color{red} 16} & 18 & 18 & 18 & 18 & &		38 & & {\color{Mahogany} 6} & 54 & 66 & 78 & 90 & 102 & \\
7 & & 15 & 21 & 21 & 21 & 21 & & 										39 & & {\color{Mahogany} 1} & {\color{JungleGreen} 50} & 63 & 75 & 87 & 99 	& 1 \\
8 & & {\color{LimeGreen} 6} & {\color{red} 22} & 24 & 24 & 24&	6 & 40 & & & {\color{JungleGreen} 21} & 60 & 72 & 84 & 96 						& 21 \\
9 & & {\color{Mahogany} 1} & 21 & 27 & 27 & 27 & & 						41 & & & {\color{Mahogany} 6} & 57 & 69 & 81 & 93 & \\
10 & & & 18 & {\color{red} 28} & 30 & 30 & &							42 & & & {\color{Mahogany} 1} & {\color{JungleGreen} 52} & 66 & 78 & 90 		& 2 \\
11 & & & {\color{LimeGreen} 6} & 27 & 33 & 33 & 			9 & 	43 & & & & {\color{JungleGreen} 21} & 63 & 75 & 87						 	& 18 \\						
12 & & & {\color{Mahogany} 1} & 24 & {\color{red} 34} & 36 & & 			44 & & & & {\color{Mahogany} 6} & 60 & 72 & 84 & \\
13 & & & & 21 & 33 & 39 & & 											45 & & & & {\color{Mahogany} 1} & {\color{JungleGreen} 52} & 69 & 81 			& 5 \\
14 & & & & {\color{LimeGreen} 6} & 30 & {\color{red} 40} & 	12 & 	46 & & & & & {\color{JungleGreen} 21} & 66 & 78 								& 3 \\
15 & & & & {\color{Mahogany} 1} & 27 & 39 & & 							47 & & & & & {\color{Mahogany} 6} & 63 & 75 & \\
16 & & & & & {\color{Green} 21} & 36 & 						3 & 	48 & & & & & {\color{Mahogany} 1} & {\color{PineGreen} 52} & 72 					& 8 \\
17 & & & & & {\color{Mahogany} 6} & 33 & & 								49 & & & & & & {\color{Mahogany} 21} & 69 & \\
18 & & & & & {\color{Mahogany} 1} & 30 & &								50 & & & & & & {\color{Mahogany} 6} & 66 & \\
19 & & & & & & {\color{Green} 21} &  							6 &		51 & & & & & & {\color{Mahogany} 1} & {\color{PineGreen} 52} 						& 11 \\
20 & & & & & & {\color{Mahogany} 6} & & 									52 & & & & & & & {\color{Mahogany} 21} & \\
21 & & & & & & {\color{Mahogany} 1} & & 									53 & & & & & & & {\color{Mahogany} 6} & \\
22 & & & & & & & &														54 & & & & & & & {\color{Mahogany} 1} & \\

\end{tabular}